\documentclass[11pt,a4paper]{amsart}
\usepackage[T1]{fontenc}
\usepackage[utf8]{inputenc}
\usepackage{lmodern}
\usepackage{mathtools,amssymb,mathrsfs}
\usepackage[margin=27mm]{geometry}
\usepackage{microtype}
\usepackage{longtable,booktabs,array,calc,multirow}
\usepackage{graphicx,xcolor,enumitem,needspace}
\usepackage[colorlinks=true,linkcolor=blue,citecolor=blue,urlcolor=blue,bookmarksnumbered]{hyperref}
\usepackage{bookmark}
\usepackage{xurl}
\setlist{topsep=0.4em,itemsep=0.2em,parsep=0pt}
\allowdisplaybreaks[2]
\providecommand{\sslash}{\mathbin{/\mkern-6mu/}}

\theoremstyle{plain}
\newtheorem{theorem}{Theorem}[section]
\newtheorem{proposition}[theorem]{Proposition}
\newtheorem{lemma}[theorem]{Lemma}
\newtheorem{corollary}[theorem]{Corollary}
\theoremstyle{definition}

\newtheorem{example}[theorem]{Example}
\numberwithin{equation}{section}
\hypersetup{
 pdftitle={Oka manifolds are ubiquitous},
 pdfsubject={Oka manifolds and applications in complex and algebraic geometry},
 pdfauthor={Sicheng An, Yun-Heng Du, Bin Guo, Peng-Chao Wang, Song-Yan Xie},
 pdfkeywords={Oka manifolds, moduli of bundles, holomorphic mapping spaces, character varieties, abelian fibrations, compactified Jacobians}
}
\title[Oka manifolds are ubiquitous]{Oka manifolds are ubiquitous}
\author{Sicheng An}
\address[Sicheng An]{School of Mathematical Sciences, University of Chinese Academy of Sciences, Beijing 100049, China}
\email{ansicheng23@mails.ucas.ac.cn}

\author{Yun-Heng Du}
\address[Yun-Heng Du]{Academy of Mathematics and Systems Science, Chinese Academy of Sciences, Beijing 100190, China}
\email{duyunheng22@mails.ucas.ac.cn}

\author{Bin Guo}
\address[Bin Guo]{Academy of Mathematics and Systems Science, Chinese Academy of Sciences, Beijing 100190, China}
\email{guobin@amss.ac.cn}

\author{Peng-Chao Wang}
\address[Peng-Chao Wang]{Academy of Mathematics and Systems Science, Chinese Academy of Sciences, Beijing 100190, China}
\email{wangpengchao26@mails.ucas.ac.cn}

\author{Song-Yan Xie}
\address[Song-Yan Xie]{State Key Laboratory of Mathematical Sciences, Academy of Mathematics and Systems Science, Chinese Academy of Sciences, Beijing 100190, China; School of Mathematical Sciences, University of Chinese Academy of Sciences, Beijing 100049, China}
\email{xiesongyan@amss.ac.cn}

\date{}
\ifdefined\DeclareUnicodeCharacter
\DeclareUnicodeCharacter{03B4}{\ensuremath{\delta}}
\DeclareUnicodeCharacter{03BB}{\ensuremath{\lambda}}
\DeclareUnicodeCharacter{03BC}{\ensuremath{\mu}}
\DeclareUnicodeCharacter{03C0}{\ensuremath{\pi}}
\DeclareUnicodeCharacter{03C7}{\ensuremath{\chi}}
\DeclareUnicodeCharacter{03B1}{\ensuremath{\alpha}}
\DeclareUnicodeCharacter{03B2}{\ensuremath{\beta}}
\DeclareUnicodeCharacter{03B3}{\ensuremath{\gamma}}
\DeclareUnicodeCharacter{03B8}{\ensuremath{\theta}}
\DeclareUnicodeCharacter{03BA}{\ensuremath{\kappa}}
\DeclareUnicodeCharacter{03C1}{\ensuremath{\rho}}
\DeclareUnicodeCharacter{03C3}{\ensuremath{\sigma}}
\DeclareUnicodeCharacter{03C4}{\ensuremath{\tau}}
\DeclareUnicodeCharacter{03C6}{\ensuremath{\phi}}
\DeclareUnicodeCharacter{03C9}{\ensuremath{\omega}}
\DeclareUnicodeCharacter{0394}{\ensuremath{\Delta}}
\DeclareUnicodeCharacter{03A3}{\ensuremath{\Sigma}}
\DeclareUnicodeCharacter{0398}{\ensuremath{\Theta}}
\DeclareUnicodeCharacter{039B}{\ensuremath{\Lambda}}
\DeclareUnicodeCharacter{03A6}{\ensuremath{\Phi}}
\DeclareUnicodeCharacter{2208}{\ensuremath{\in}}
\DeclareUnicodeCharacter{2209}{\ensuremath{\notin}}
\DeclareUnicodeCharacter{2282}{\ensuremath{\subset}}
\DeclareUnicodeCharacter{2286}{\ensuremath{\subseteq}}
\DeclareUnicodeCharacter{2295}{\ensuremath{\oplus}}
\DeclareUnicodeCharacter{2297}{\ensuremath{\otimes}}
\DeclareUnicodeCharacter{2245}{\ensuremath{\cong}}
\DeclareUnicodeCharacter{2243}{\ensuremath{\simeq}}
\DeclareUnicodeCharacter{2265}{\ensuremath{\geq}}
\DeclareUnicodeCharacter{2264}{\ensuremath{\leq}}
\DeclareUnicodeCharacter{2260}{\ensuremath{\ne}}
\DeclareUnicodeCharacter{2192}{\ensuremath{\to}}
\DeclareUnicodeCharacter{2190}{\ensuremath{\leftarrow}}
\DeclareUnicodeCharacter{2194}{\ensuremath{\leftrightarrow}}
\DeclareUnicodeCharacter{21D2}{\ensuremath{\Rightarrow}}
\DeclareUnicodeCharacter{21D4}{\ensuremath{\Longleftrightarrow}}
\DeclareUnicodeCharacter{00D7}{\ensuremath{\times}}
\DeclareUnicodeCharacter{2205}{\ensuremath{\varnothing}}
\DeclareUnicodeCharacter{221E}{\ensuremath{\infty}}
\DeclareUnicodeCharacter{2202}{\ensuremath{\partial}}
\DeclareUnicodeCharacter{2211}{\ensuremath{\sum}}
\DeclareUnicodeCharacter{2102}{\ensuremath{\mathbb C}}
\DeclareUnicodeCharacter{2124}{\ensuremath{\mathbb Z}}
\DeclareUnicodeCharacter{2119}{\ensuremath{\mathbb P}}
\DeclareUnicodeCharacter{211D}{\ensuremath{\mathbb R}}
\DeclareUnicodeCharacter{2212}{\ensuremath{-}}
\DeclareUnicodeCharacter{2217}{\ensuremath{*}}
\DeclareUnicodeCharacter{2032}{\ensuremath{\prime}}
\DeclareUnicodeCharacter{2082}{\ensuremath{_2}}
\DeclareUnicodeCharacter{2081}{\ensuremath{_1}}
\DeclareUnicodeCharacter{2080}{\ensuremath{_0}}
\fi

\begin{document}
\begin{abstract}
We prove that a holomorphic family of connected complex Lie groups has
Oka total space if and only if its base is Oka. Extension presentations
and generating sections give Oka results for stable and parabolic
bundle moduli and for holomorphic mapping and embedding spaces.
Every nonempty degree component of maps from a compact Oka curve to
projective space is Oka; rational-function spaces are also
holomorphically elliptic. For a vector bundle on a smooth projective curve,
the Oka property of its nested Quot scheme is equivalent to that of
the symmetric powers determined by the length increments. Complete
Hamiltonian flows give holomorphic ellipticity of irreducible
\(\mathrm{SL}_2\)- and \(\mathrm{GL}_2\)-character and representation
loci, for closed surfaces of genus at least two and for surfaces with
fixed regular semisimple boundary classes. The Riemann--Hilbert
correspondence yields Oka moduli spaces of nonresonant parabolic
connections for arbitrary parabolic weights, and holomorphic
ellipticity when reducible monodromy is absent. A proper holomorphic fibration over a Riemann surface, with a
section and connected complex tori as smooth general fibres, has Oka
total space exactly when its base is Oka. For a Schoen fibre product
\(X\) of rational elliptic surfaces with disjoint singular values, we
construct an entire local biholomorphism \(\mathbb C^3\to X\) whose
image contains a dense Zariski-open subset. Further applications treat
Hilbert schemes, compactified relative Jacobians, and logarithmic
complements.
\end{abstract}
\maketitle
\setcounter{tocdepth}{1}
\tableofcontents

\section{Introduction}\label{sec:introduction}

Holomorphic families and moduli spaces provide a natural setting for
the Oka property. Their points represent geometric objects, while
deformations of those objects give holomorphic maps into the parameter
space. In this paper we use relative exponentials, extension spaces,
generating sections, and Hamiltonian flows to construct enough entire
maps to prove that these spaces are Oka. We also study how the Oka
property behaves when a family acquires singular fibres.

A complex manifold \(Y\) is \emph{Oka} if, for every integer \(m\ge1\),
every compact convex set \(K\subset\mathbb C^m\), and every holomorphic
map from a neighbourhood of \(K\) to \(Y\), uniform approximation on
\(K\) by holomorphic maps \(\mathbb C^m\to Y\) is possible. This
convex approximation property is equivalent to the usual Oka principle
with approximation and interpolation \cite{F06}. The theory originates
in the work of Oka and Grauert and was developed through Gromov's
theory of sprays \cite{Oka39,Grauert58,Gromov89}; see \cite{F17} for
background.

Our analytic input is the work of An, Du, Guo, Wang, and Xie
\cite{DW}. A manifold \(Y\) is \emph{strongly dominable} if every
\(y\in Y\) is the image of \(0\) under a holomorphic map
\(F:\mathbb C^{\dim Y}\to Y\) with \(dF_0\) invertible.
Strong dominability is equivalent to the Oka property
\cite[Theorem 1.4]{DW}. A complex manifold with a dense Zariski-open Oka
subset is Oka \cite[Theorem 1.9]{DW}, and a surjective holomorphic map
of connected complex manifolds carries the Oka property from its source
to its target \cite[Theorem 7.5]{DW}. We construct entire maps through
prescribed points, descend from Oka parameter spaces along surjective
classifying maps, and pass from dense Oka open loci to their smooth
ambient manifolds.

Some constructions produce a dominating holomorphic spray: a holomorphic
map \(s:E\to Y\) from a holomorphic vector bundle over \(Y\), with
\(s(0_y)=y\) and surjective fibre differential
\(d(s|_{E_y})_{0_y}:E_y\to T_yY\) at every \(y\in Y\).
A manifold admitting such a spray is \emph{holomorphically elliptic}
and is Oka by Gromov's theorem \cite{Gromov89}.
Every projective Oka manifold is holomorphically elliptic by
Forstneri\v{c}--L\'arusson \cite[Theorem 1.1]{FL2026}.

\subsection*{Holomorphic group families}
A holomorphic family of complex Lie groups is a holomorphic submersion
equipped with holomorphic multiplication, inversion, and an identity
section, inducing a complex Lie group on every fibre. Our first result
relates the total space to the base.

\begin{theorem}\label{prop:connected-group-family}
Let \(p:G\to B\) be a holomorphic family of connected complex Lie groups.
Then \(G\) is strongly dominable if and only if \(B\) is strongly dominable.
Consequently,
\[
G\text{ is Oka}\quad\Longleftrightarrow\quad B\text{ is Oka}.
\]
\end{theorem}

The relative exponential supplies the vertical directions. After
pullback along a dominating map from Euclidean space, the relative Lie
algebra is holomorphically trivial. Products of exponentials then give
a lift through any prescribed point of the total space. Together with
the map to the base, this yields a dominating entire map. The proof is
given in Section~\ref{sec:background}.

\subsection*{Moduli spaces and holomorphic maps}
For vector bundles on a curve, affine extension spaces play the role
of the Lie algebra. We show that the nonstable parameters have
codimension at least two in the ranges below. Their complement is Oka,
and holomorphic descent gives the Oka property of the coarse moduli
space. Parabolic flags add Schubert conditions whose codimension
compensates for the destabilizing weight contribution.

Let \(C\) be a smooth projective connected curve of genus \(g\ge2\),
let \(r\ge2\), and fix \(\xi\in\operatorname{Pic}^{d}(C)\). Write
\[
M_C(r,\xi)=SU_C^s(r,\xi)
\]
for the coarse moduli space of stable rank-\(r\) vector bundles with
determinant \(\xi\); \(SU_C(r,\xi)\) denotes the semistable moduli space.
The corresponding variable-determinant stable space is \(U_C^s(r,d)\).
For distinct marked points
\(\mathbf p=(p_1,\ldots,p_m)\), prescribed flag types \(\mathbf n\),
and real weights
\(0\le\alpha_{i1}<\cdots<\alpha_{it_i}<1\), write
\(\mathcal P^s(C,\mathbf p;r,\xi,\mathbf n,\boldsymbol\alpha)\)
for the corresponding stable parabolic moduli space.

\begin{theorem}\label{thm:A1}
If
\[
\gcd(r,d)=1
\quad\text{or}\quad
(r,g)\ne(2,2),
\]
then \(M_C(r,\xi)\) is Oka. The variable-determinant moduli space
\(U_C^s(r,d)\) is Oka in the same range. If \(\gcd(r,d)=1\), both
\(SU_C(r,\xi)\) and \(U_C^s(r,d)\) are smooth projective Oka
manifolds and are holomorphically elliptic.
\end{theorem}

In the coprime case, \(SU_C(r,\xi)\) is a smooth Fano manifold by
Dr\'ezet--Narasimhan's canonical-bundle calculation
\cite[Theorem F]{DN1989}, so its Oka property also follows from the
theorem on rationally connected manifolds \cite[Theorem 1.10]{DW}. The extension construction
in Section~\ref{sec:bundle-moduli} treats this case together with the
noncoprime stable loci in Theorem~\ref{thm:A1}. Its codimension estimate
has a rank--genus contribution of at least \((r-1)(g-1)\), to which
coprimality adds a positive integer. For parabolic bundles, a nontrivial
flag supplies a positive excess of Schubert codimension over its weight
contribution.

\begin{theorem}\label{thm:A2}
Suppose \(g(C)\ge2\), \(r\ge2\), and at least one prescribed parabolic
flag is nontrivial. For every degree, determinant, and real weight
system as above,
\[
\mathcal P^s(C,\mathbf p;r,\xi,\mathbf n,\boldsymbol\alpha)
\]
is Oka, as is the corresponding variable-determinant stable moduli space.
Whenever semistability equals stability, the full parabolic
moduli spaces with fixed and variable determinant are smooth
projective Oka manifolds and are holomorphically elliptic; this
includes full flags with generic weights.
\end{theorem}

Section~\ref{sec:bundle-moduli} proves these theorems by the same
extension construction. More generally, in every genus, a nonempty
parabolic moduli space for which semistability equals stability is
Oka and holomorphically elliptic
(Corollary~\ref{cor:compact-parabolic-moduli}).
An explicit codimension calculation also gives stable Oka loci in
genera zero and one, including weights on stability walls.

Generating sections lead to mapping spaces
(Section~\ref{sec:mapping-spaces}). For a smooth projective curve
\(C\) of genus \(g\), the space of degree-\(d\) maps to
\(\mathbb P^m\) is Oka if
\[
d\ge2g,\qquad m\ge\max\{2,g+1\}.
\]
The embedding space is Oka if
\[
d\ge2g+1,\qquad m\ge\max\{4,g+3\}.
\]
The degree bounds give global generation and very ampleness,
respectively; the target-dimension bounds provide enough sections for
the incidence and entire-map constructions. Both statements are
specializations of results for higher-dimensional sources.
Separate constructions for target \(\mathbb P^1\) prove that
\(\operatorname{Hol}_d(\mathbb P^1,\mathbb P^1)\) is holomorphically
elliptic for every \(d\ge1\), and
\(\operatorname{Hol}_d(E,\mathbb P^1)\) is Oka for every elliptic curve
\(E\) and every \(d\ge2\). The rational-function result answers Hanysz's
degree-by-degree Oka question \cite[Section 1]{Hanysz2015}; the
higher-degree elliptic-function problem is discussed by Bowman
\cite[Section 1.3]{Bowman2016}. Together with the projective-map theorem,
these results give the Oka property of every nonempty degree component
of maps from a compact Oka curve to any \(\mathbb P^m\),
\(m\ge1\) (Corollary~\ref{cor:oka-curve-projective-maps}).
The same section treats maps to Grassmannians with stable quotient
bundle.

Finite analytic subspaces and quotients give two further classes of
Oka moduli spaces (Section~\ref{sec:hilbert-quot}). If the ordered
configuration space of \(n\) points on a connected complex surface
\(S\) is Oka, then so is its Hilbert--Douady space \(S^{[n]}\).
In particular, all Hilbert--Douady spaces of points on a product of
two connected Oka curves are Oka. The configuration-space assertion holds more
generally for products of any two positive-dimensional connected Oka
manifolds. For a vector bundle \(V\) on a smooth projective curve
\(C\), and a nondecreasing length vector \(\mathbf d\), set
\(d_0=0\) and \(n_j=d_j-d_{j-1}\). Then
\[
\operatorname{Quot}_{\mathbf d}(V)\text{ is Oka}
\quad\Longleftrightarrow\quad
\prod_j\operatorname{Sym}^{n_j}C\text{ is Oka}.
\]
All length vectors are allowed in genus at most one. In genus \(g\ge2\),
every positive increment must be at least \(g\), while increments at
least \(2g-2\) suffice. In genus two the exact condition is
\(n_j\ne1\) for every \(j\) (Theorem~\ref{thm:N2}).

For surface-group character varieties, complete Hamiltonian twist
flows give the entire maps (Section~\ref{sec:character-connections}).
For closed surfaces of genus at least two, the irreducible
\(\mathrm{SL}_2(\mathbb C)\)- and \(\mathrm{GL}_2(\mathbb C)\)-character
and representation loci are holomorphically elliptic. The corresponding
relative irreducible character and representation loci with compatible
fixed regular semisimple boundary classes are also holomorphically elliptic.
In the closed-surface case, the full smooth representation loci are
Oka, as are resolutions of the character varieties that are
isomorphisms over their irreducible loci.
The Riemann--Hilbert correspondence gives Oka moduli spaces of
nonresonant rank-two parabolic connections for arbitrary parabolic
weights in the ranges of Theorem~\ref{thm:M11}. For generic weights this
includes fixed-time initial-data spaces of Garnier systems. When no signed sum of the residue eigenvalues is
integral, the whole connection moduli space is holomorphically
elliptic. We also describe each nonempty reducible inverse image as
the analytic product
\((\mathbb C^*)^{2g}\times\mathbb P^{2g+n-3}\)
(Lemma~\ref{lem:M11-exceptional}). For an elliptic curve with one
puncture, the nonresonant connection spaces are holomorphically elliptic
surfaces (Example~\ref{ex:one-punctured-elliptic-connections}).

\subsection*{Complex torus fibrations and compactifications}
A proper holomorphic fibration over a Riemann surface, with a
holomorphic section and connected complex tori as smooth general
fibres, has Oka total space if and only if the base is Oka.
On an unramified entire parametrization of an Oka base, translation
fields extend to complete commuting vector fields on the pulled-back
total space. Their flows give an entire map covering the general
fibres and having invertible differential at every preimage of them
(Theorem~\ref{thm:abelian-fibration}).

The resulting Calabi--Yau examples include Schoen fibre products of
rational elliptic surfaces with disjoint singular values, and
Gross--Popescu threefolds with their derived-equivalent dual partners.
For a Schoen fibre product we construct
\(F:\mathbb C^3\to X\) with \(dF\) invertible everywhere and image
containing a dense Zariski-open subset
(Theorem~\ref{thm:schoen-threefold}).

Compactified relative Jacobians have a different boundary: their
non-invertible points represent rank-one torsion-free sheaves. For a
one-node degeneration, the local model \(uv=t\) makes the compactified
total space smooth. Analytic extension from the generalized Jacobian
group family and the relative Abel map then give the following result.

\begin{theorem}\label{thm:J1}
Let \(B\) be a smooth connected projective curve and
\(f:S\to B\) a flat projective morphism from a smooth projective
complex surface, with a section. Suppose every fibre is geometrically
integral of arithmetic genus \(g\ge1\), and every singular fibre has a
unique singular point, which is an ordinary node. Then
\(\overline J^d(S/B)\) and \(\operatorname{Hilb}^n(S/B)\) are smooth
and projective for every \(d\in\mathbb Z\) and \(n\ge0\). Moreover,
the compactified Jacobians, and the relative Hilbert schemes with
\(n\ge2g-1\), are Oka exactly when \(g(B)\le1\); in that case they
are holomorphically elliptic.
\end{theorem}

\subsection*{Further geometric applications}
Section~\ref{sec:geometric-constructions} gives applications to
deformation spaces and complements. Deformation to the normal cone
along a smooth closed submanifold has Oka total space exactly when
the ambient connected manifold is Oka. A triangular construction treats
coordinate arrangements in products of positive-dimensional connected
Oka manifolds. It also gives Oka complements and blowups along smooth
centres contained in these arrangements.

For a smooth projective surface \(S\) and a nonzero reduced simple
normal crossings divisor \(D\), ampleness of \(-(K_S+D)\) implies
that \(S\setminus D\) is holomorphically elliptic. The
higher-dimensional result gives the same conclusion for ample boundary
components with independent numerical classes whose deepest intersection
is rationally connected of dimension at least two
(Theorem~\ref{thm:multibranch-log-fano}).
McLean's Stein fourfold is Oka and holomorphically elliptic; it is
diffeomorphic to \(\mathbb R^8\) and not biholomorphic to
\(\mathbb C^4\).

Sections~\ref{sec:background}
and~\ref{sec:bundle-moduli} use relative exponentials and extension
parameters. Sections~\ref{sec:mapping-spaces}--\ref{sec:character-connections}
use generating families, dense torsor loci, and complete Hamiltonian
flows for mapping spaces, nested Quot schemes, and character loci.
Section~\ref{sec:fibrations} treats complex torus fibrations and
compactified Jacobians, proving
Theorems~\ref{thm:abelian-fibration} and~\ref{thm:J1}.

For a complex vector space or holomorphic vector bundle \(V\),
\(\mathbb P(V)\) parametrizes one-dimensional linear subspaces.

\section{Holomorphic families of complex Lie groups}\label{sec:background}

Let \(p:G\to B\) be a holomorphic family of complex Lie groups,
and write \(e:B\to G\) for its identity section. The proof of
Theorem~\ref{prop:connected-group-family} uses the exponential map in
each fibre, with holomorphic dependence on the base.

\begin{lemma}\label{lem:relative-exponential}
Let \(p:G\to B\) be a holomorphic family of complex Lie groups. The vertical
tangent spaces along the identity section form a holomorphic vector bundle
\[
\mathfrak g=e^*(\ker dp)\longrightarrow B.
\]
There is a holomorphic map \(\operatorname{Exp}:\mathfrak g\to G\) over
\(B\) whose restriction to each fibre is its exponential map. Its vertical
differential along the zero section is the identity.
\end{lemma}

\begin{proof}
For \(b\in B\) and \(v\in\mathfrak g_b\), left translation gives
the complete vertical vector field
\[
V_v(g)=d(L_g)_{e(b)}v,\qquad g\in G_b.
\]
In a local holomorphic frame of \(\mathfrak g\), its coefficients
depend holomorphically on \((b,v)\). Holomorphic dependence for
ordinary differential equations makes the flow \(\Phi_v^t(g)\)
holomorphic in \((b,v,t,g)\). Set
\[
\operatorname{Exp}(v)=\Phi_v^1(e(b)),\qquad
\operatorname{Exp}(tv)=\Phi_v^t(e(b)).
\]
Uniqueness of solutions makes the local maps agree on overlaps. The
identity
\[
\left.\frac{d}{dt}\right|_{t=0}\operatorname{Exp}(tv)=v
\]
gives the asserted vertical differential.
\end{proof}

\begin{proof}[Proof of Theorem~\ref{prop:connected-group-family}]
Write \(n=\dim B\) and \(r=\dim G-\dim B\). Suppose first that
\(B\) is strongly dominable, and fix \(g_0\in G_{b_0}\). Choose
\[
f:\mathbb C^n\longrightarrow B,\qquad f(0)=b_0,
\qquad df_0\text{ invertible}.
\]
The Oka--Grauert principle \cite{Grauert58}, in the vector-bundle
form of \cite[Theorem 5.3.1]{F17}, trivializes
\(f^*\mathfrak g\) on the contractible Stein space \(\mathbb C^n\).
Choose a holomorphic frame \(e_1,\ldots,e_r\).

The exponential image of a connected Lie group contains a neighbourhood
of the identity, and that neighbourhood generates the group. Hence
\[
g_0=\exp(v_1)\cdots\exp(v_s),\qquad v_i\in\mathfrak g_{b_0}.
\]
Express each \(v_i\) in the frame at zero and extend it with constant
coefficients to a section \(v_i(u)\) of \(f^*\mathfrak g\). By
Lemma~\ref{lem:relative-exponential}, the product
\[
\sigma(u)=\operatorname{Exp}(v_1(u))\cdots
\operatorname{Exp}(v_s(u))
\]
is a holomorphic lift of \(f\) through \(g_0\). It supplies the
horizontal directions in the map
\[
F:\mathbb C^n\times\mathbb C^r\longrightarrow G,
\qquad
F(u,z)=\sigma(u)\operatorname{Exp}\left(\sum_{j=1}^r z_je_j(u)\right).
\]
Relative to a complex-linear splitting of
\[
0\longrightarrow T_{g_0}G_{b_0}\longrightarrow T_{g_0}G
\xrightarrow{dp_{g_0}}T_{b_0}B\longrightarrow0,
\]
the differential is
\[
dF_{(0,0)}=
\begin{pmatrix}
df_0&0\\
*&d(L_{g_0})_{e(b_0)}\circ E_0
\end{pmatrix},
\qquad E_0(z)=\sum_{j=1}^r z_je_j(0).
\]
Both diagonal blocks are invertible, so \(F\) dominates \(G\) at
\(g_0\).

Conversely, suppose that \(G\) is strongly dominable. For any
\(b_0\in B\), choose \(H:\mathbb C^{n+r}\to G\) with
\(H(0)=e(b_0)\) and \(dH_0\) invertible. The differential of
\(p\circ H\) is surjective at zero. Restriction to an
\(n\)-dimensional complex-linear complement of its kernel gives an
entire map dominating \(B\) at \(b_0\). The equivalence of the Oka
properties follows from \cite[Theorem 1.4]{DW}.
\end{proof}

\section{Moduli of vector bundles}\label{sec:bundle-moduli}\label{sec:A1}

Open subsets of extension spaces parametrize stable bundles.
The proof of Theorem~\ref{thm:A1} estimates the codimension of the
nonstable locus in such a space and proves that its classifying map is
a surjective submersion. Adding flags gives the parabolic result in
Theorem~\ref{thm:A2}; the relative positions of the flags contribute the
additional codimension needed there.

\subsection{Stable bundles}

For a smooth projective connected curve \(C\) of genus \(g\ge2\), an integer \(r\ge2\), and \(\xi\in\operatorname{Pic}^{d}(C)\), the notation
\[
M_C(r,\xi)=SU_C^s(r,\xi)
\]
denotes the coarse moduli space of stable rank-\(r\) vector bundles with determinant isomorphic to \(\xi\).

\begin{proof}[Proof of Theorem~\ref{thm:A1}]
We first twist the bundles to obtain a uniform extension presentation.
Choose a line bundle \(L\) of sufficiently negative degree \(l\), so that
\[
M=\xi\otimes L^{-r},\qquad e=d-rl
\]
satisfy
\begin{equation}
e-rg\ge2,
\qquad
\frac er>2g-1.
\label{eq:A1.1}
\end{equation}
The first condition gives a codimension-two estimate for the
exceptional cohomology locus in the extension space. The second gives
stable bundles the cohomology vanishing needed for global generation
and section lifting.
The space
\[
V=H^1(C,M^{-1})^{\oplus(r-1)}
\]
parametrizes extensions of the fixed bundle \(M\) by
\(\mathcal O_C^{r-1}\). In rank two this is Bertram's extension-space
construction \cite[§2]{B1992}; Heinloth describes the construction for
arbitrary rank \cite[Lectures 1--3]{H2009}:
\begin{equation}
0\longrightarrow\mathcal O_C^{r-1}
\xrightarrow{i_v}E_v\longrightarrow M\longrightarrow0.
\label{eq:A1.2}
\end{equation}
This is Seshadri's extension presentation, as used by
Dr\'ezet--Narasimhan \cite[Sections 0.2 and 7]{DN1989}.
To construct the family on \(C\times V\), take a finite affine cover of
\(C\) and Čech cocycles representing a basis of \(H^1(M^{-1})\).
Their linear combinations give the upper-right blocks of the transition
matrices in \eqref{eq:A1.2}. The resulting algebraic family has a
specified determinant isomorphism \(\det E_v\simeq M\).

Let \(V^s\subset V\) be the Zariski-open locus on which \(E_v\) is
stable. It is enough to prove
\begin{equation}
\operatorname{codim}_V(V\setminus V^s)\ge2,
\qquad
V^s\longrightarrow M_C(r,\xi),\quad v\longmapsto[E_v\otimes L],
\text{ is a surjective submersion}.
\label{eq:A1.3}
\end{equation}
The codimension bound makes \(V^s\) Oka; surjective holomorphic descent
\cite[Theorem 7.5]{DW} then gives the fixed-determinant moduli space.

\smallskip
\noindent\emph{Regular extensions.}
Consider the locus
\[
B_1=\{v\in V:\text{ there is }p\in C
\text{ with }H^1(C,E_v(-p))\ne0\}.
\]
For every \(p\), the degree assumption gives \(H^1(M(-p))=0\).
Tensoring \eqref{eq:A1.2} by \(\mathcal O_C(-p)\) shows that
\(H^1(C,E_v(-p))\ne0\) precisely when the following connecting map
fails to be surjective:
\[
\delta_{v,p}:H^0(M(-p))\longrightarrow
H^1(\mathcal O_C(-p))^{r-1}.
\]
By Serre duality, such a failure is witnessed by a
nonzero tuple
\[
\alpha=(\alpha_1,\ldots,\alpha_{r-1})
\in H^0(K_C(p))^{r-1}=H^0(K_C)^{r-1}.
\]
The equality follows from the inclusion \(H^0(K_C)\subset H^0(K_C(p))\) and Riemann--Roch, which gives dimension \(g\) for both spaces.

For fixed \(p\) and \([\alpha]\), the equation \(\alpha\circ\delta_{v,p}=0\) is linear in \(v\). Its transpose under Serre duality is \[
\mu_\alpha:H^0(M(-p))\longrightarrow
H^0(MK_C)^{r-1},
\qquad
s\longmapsto(s\alpha_1,\ldots,s\alpha_{r-1}).
\]
At least one \(\alpha_i\) is nonzero, and multiplication by that
differential is injective on \(H^0(M(-p))\) because \(C\) is integral.
Thus \(\mu_\alpha\) is injective. Riemann--Roch gives
\[
h^0(M(-p))=e-g,
\]
so the linear conditions on \(V\) have codimension exactly \(e-g\).

The witnesses form the proper variety \[
C\times\mathbb P\bigl(H^0(K_C)^{r-1}\bigr),
\]
of dimension \((r-1)g\). The incidence is algebraic and its projection
to \(V\) is proper. Hence \(B_1\) is closed and
\begin{equation}
\operatorname{codim}_V B_1
\ge(e-g)-(r-1)g=e-rg\ge2.
\label{eq:A1.4}
\end{equation}
Set \(V^{\rm reg}=V\setminus B_1\). The sequence
\[
0\to E_v(-p)\to E_v\to(E_v)_p\to0
\]
shows that every \(E_v\), \(v\in V^{\rm reg}\), is globally generated and has \(H^1(E_v)=0\).

For a stable \(E_v\), a nonzero map \(E_v\to K_C(p)\) would have a
rank-one quotient image of degree at most \(2g-1\). Stability requires
that quotient to have slope larger than \(e/r>2g-1\), a contradiction.
Serre duality therefore gives \(H^1(E_v(-p))=0\) for every \(p\), so
\begin{equation}
V^s\subset V^{\rm reg}.
\label{eq:A1.5}
\end{equation}

\smallskip
\noindent\emph{The classifying morphism.}
The determinant-rigidified stack is
\[
\mathcal B_{r,M}=\{(E,\det E\xrightarrow{\sim}M)\}.
\]
It has dimension \((r^2-1)(g-1)\) \cite[Lectures 1--3]{H2009}.
Here the Picard stack has dimension \(g-1\), whereas the Picard variety
has dimension \(g\); the extension calculation uses the former.

The family \eqref{eq:A1.2} gives a smooth morphism \begin{equation}
V^{\rm reg}\longrightarrow\mathcal B_{r,M}.
\label{eq:A1.6}
\end{equation}
For infinitesimal lifting in \eqref{eq:A1.6}, consider a square-zero
deformation of \(E_v\) with its determinant isomorphism. The obstruction to
lifting each of the \(r-1\) sections in \(i_v\) lies in \(H^1(E_v)\)
tensored with the square-zero ideal, and \(H^1(E_v)=0\) on
\(V^{\rm reg}\). The lifted sections still have rank \(r-1\), since the
nonvanishing of a maximal minor persists under a nilpotent thickening.
They therefore give a lifted extension
\[
0\longrightarrow\mathcal O^{r-1}\longrightarrow\widetilde E_v
\longrightarrow\widetilde M\longrightarrow0,
\qquad \widetilde M\simeq\det\widetilde E_v.
\]
The prescribed determinant identifies \(\widetilde M\) with the pullback
of \(M\), so the extension, including its endpoint identifications, lifts
to the family \eqref{eq:A1.2}. This proves formal smoothness of
\eqref{eq:A1.6}; finite presentation gives smoothness.

In particular, the Kodaira--Spencer map
\[
T_vV=H^1(M^{-1})^{r-1}\longrightarrow
H^1(\operatorname{End}_0 E_v)\longrightarrow0
\]
is surjective.

\smallskip
\noindent\emph{Codimension of the nonstable locus.}
A nonstable bundle \(E\) has a saturated subbundle \(F\subset E\) of rank \(s\), degree \(f\), where \[
1\le s\le r-1,\qquad q=r-s,\qquad A=rf-se\ge0.
\]
For \(Q=E/F\), consider the stack of sequences \[
0\to F\to E\to Q\to0,
\qquad
\det F\otimes\det Q\simeq M.
\]
The stack of pairs \((F,Q)\) with this determinant identification has dimension \[
(s^2+q^2-1)(g-1).
\]
Indeed, the product of the two unrestricted bundle stacks has dimension \((s^2+q^2)(g-1)\), and fixing their product determinant subtracts the dimension \(g-1\) of the Picard stack.

For each pair, the stack of extensions has relative dimension \[
\dim\operatorname{Ext}^1(Q,F)-\dim\operatorname{Hom}(Q,F)
=-\chi(Q^*\otimes F).
\]
Now \(Q^*\otimes F\) has rank \(sq\) and degree \[
qf-s(e-f)=rf-se=A.
\]
Therefore its extension-stack contribution is \(sq(g-1)-A\), and the total extension stack has dimension \[
(r^2-sq-1)(g-1)-A.
\]
The morphism forgetting the subbundle is representable: over a fixed
\(E\), its fibre is the locally free quotient locus in a Quot scheme.
Its image has dimension at most that of the extension stack. Comparing
with \(\dim\mathcal B_{r,M}\) gives
\begin{equation}
\operatorname{codim}(\text{nonstable type }s,f)
\ge sq(g-1)+A.
\label{eq:A1.8}
\end{equation}
The same calculation includes \(A=0\), the slope-equality case.

Only finitely many degrees \(f\) occur in the extension family. A destabilizing \(F\) has \(f\ge\lceil se/r\rceil>0\), so it cannot lie in \(\mathcal O_C^{r-1}\): the determinant of a subsheaf of a trivial bundle has degree at most zero. Its image in \(M\) is \(M(-D)\) for an effective divisor \(D\). The kernel is a subsheaf of \(\mathcal O_C^{r-1}\), again of nonpositive degree, and hence \begin{equation}
\lceil se/r\rceil\le f\le e.
\label{eq:A1.9}
\end{equation}
Thus only finitely many pairs \((s,f)\) occur. Their relative Quot
schemes over \(V\) are of finite type, so the nonstable locus is covered
by finitely many such images.

Smooth pullback along \eqref{eq:A1.6} preserves the codimension bound.
The term \(sq(g-1)\), determined by rank and genus, is at least
\((r-1)(g-1)\). If \(\gcd(r,e)=\gcd(r,d)=1\), the slope excess
\(A=rf-se\) is a positive integer for \(0<s<r\), so \[
sq(g-1)+A\ge(r-1)(g-1)+1\ge2.
\]
For \(\gcd(r,d)>1\), the contribution \((r-1)(g-1)\) is itself at
least two when \((r,g)\ne(2,2)\). This explains the numerical range
of the theorem. Together with \eqref{eq:A1.4}, stability openness gives \[
\operatorname{codim}_V(V\setminus V^s)\ge2
\]
in the range stated in Theorem~\ref{thm:A1}.

\smallskip
\noindent\emph{Surjectivity and descent.}
Let \(E\) be a stable bundle with determinant \(\xi\). Then
\(E'=E\otimes L^{-1}\) has slope \(e/r>2g-1\). The quotient-slope
argument above makes \(E'\) globally generated and gives \(H^1(E')=0\).
At any \(p\in C\), evaluation
\[
H^0(E')^{r-1}\longrightarrow\operatorname{Mat}_{r\times(r-1)}(\mathbb C)
\]
is surjective. The matrices of rank at most \(r-2\) have codimension
two. Allowing \(p\) to vary gives an incidence of codimension at least
one in the section-tuple space, with closed projection because \(C\)
is projective. A tuple outside this projection is everywhere independent
and gives an extension
\[
0\to\mathcal O_C^{r-1}\to E'\to\det E'=M\to0.
\]
Thus \(E'\) occurs in \eqref{eq:A1.2}.

Consequently \(V^s\to M_C(r,\xi)\) is surjective. It is a holomorphic submersion by \eqref{eq:A1.6}: on the stable locus the determinant-rigidified stack is a \(\mu_r\)-gerbe over its smooth coarse moduli space. Over \(\mathbb C\), \(\mu_r\) is finite étale, and this gerbe has no infinitesimal directions; thus the coarse classifying map retains the surjective differential.

The affine codimension-two complement theorem \cite[§6.4]{F17} makes \(V^s\) Oka. Surjective holomorphic descent \cite[Theorem 7.5]{DW} proves the fixed-determinant assertion.
When \(\gcd(r,d)=1\), semistability equals stability, so
\(SU_C(r,\xi)=M_C(r,\xi)\) is smooth and projective.

For variable determinant, consider
\begin{equation}
SU_C^s(r,\xi)\times\operatorname{Pic}^0(C)
\longrightarrow U_C^s(r,d),
\qquad(E,N)\longmapsto E\otimes N,
\label{eq:A1.10}
\end{equation}
After identifying \(\operatorname{Pic}^d(C)\) with
\(\operatorname{Pic}^0(C)\) by \(\xi\), this is the base change of
multiplication by \(r\) along the determinant map. It is therefore a
finite étale surjection. Its source is Oka, since a Jacobian is a
complex Lie group. Covering descent
\cite[Corollaries 2.51--2.52]{For} proves the variable-determinant
assertion. In the coprime case, \(U_C^s(r,d)\) is also projective,
so both moduli spaces are holomorphically elliptic by
\cite[Theorem 1.1]{FL2026}.
\end{proof}

\subsection{Parabolic bundles}

Fix distinct marked points \(p_1,\ldots,p_m\) on a smooth projective connected curve \(C\), with a flag type at each \(p_i\) having positive graded dimensions \[
n_{i1},\ldots,n_{it_i},\qquad \sum_h n_{ih}=r,
\]
and real weights \[
0\le\alpha_{i1}<\cdots<\alpha_{it_i}<1.
\]
We use descending flags, so the deeper graded pieces have the larger weights. A flag type is nontrivial when \(t_i\ge2\). The parabolic degree is \[
\operatorname{pardeg}(E_*)
=\deg E+\sum_{i,h}n_{ih}\alpha_{ih}.
\]
Stability means that every nonzero proper saturated subbundle, equipped with the induced flags and weights, has strictly smaller parabolic slope \cite[§§1--2]{MS1980}.

Real weights give the same moduli spaces as suitable rational weights.
For a subbundle of rank \(s\), degree \(f\), and graded dimensions
\(a_{ih}\), the stability wall is
\[
rf-sd+\sum_{i,h}(r a_{ih}-s n_{ih})\alpha_{ih}=0.
\]
Only finitely many of these walls meet the weight polytope: ranks
and graded dimensions are bounded, and bounded weights bound \(f\)
on a wall \cite[Section 2]{BodenYokogawa1999}. Their defining functions
have rational coefficients; all other stability functions have constant
sign on the polytope. The relative face containing a given real weight
system therefore contains a rational system with the same signs for
all stability functions. Replacing the weights by this
system preserves stable and semistable objects, and hence their
coarse moduli spaces.

\begin{proof}[Proof of Theorem~\ref{thm:A2}]
We add the prescribed flags to the extension family used for
Theorem~\ref{thm:A1}. Choose \(L,M,e,V\) as there, increasing \(e\)
until
\begin{equation}
e-rg\ge2,\qquad \frac er>2g-1+m.
\label{eq:A2.1}
\end{equation}
Tensoring a parabolic bundle by \(L^{-1}\) tensors its flag at each
\(p_i\) by \(L_{p_i}^{-1}\) and preserves the weights and stability.
We work with these twists, whose determinant is \(M\) and degree is
\(e\). The additional \(m\) in \eqref{eq:A2.1} allows for the
parabolic weights in the quotient-slope estimate.

\smallskip
\noindent\emph{The flagged extension space.}
For the universal extension \(\mathcal E\) on \(C\times V\), form
the fibre product
\[
W=\prod_{i=1}^m{}_V
\operatorname{Flag}_{\mathbf n_i}(\mathcal E|_{\{p_i\}\times V}).
\]
For each marked point, restriction of the universal extension gives \[
0\to\mathcal O_V^{r-1}\to\mathcal E_{p_i}
\to M_{p_i}\otimes\mathcal O_V\to0.
\]
The extension splits algebraically because \(V\) is affine and
\(H^1(V,\mathcal O_V)=0\). Hence
\begin{equation}
W\simeq V\times\prod_i\operatorname{Flag}_{\mathbf n_i}(\mathbb C^r).
\label{eq:A2.2}
\end{equation}
The complement of \(W^{\rm reg}=W\times_V V^{\rm reg}\) has
codimension at least two by \eqref{eq:A1.4}. The morphism from
\(W^{\rm reg}\) to the determinant-rigidified quasi-parabolic bundle
stack is smooth. Indeed,
given a deformation of a flagged bundle, lift the \(r-1\) sections
using \(H^1(E)=0\), as in \eqref{eq:A1.6}. The resulting identification
with the extension family transports the prescribed flags. Since
\(W\) parametrizes every flag of these types, the transported flags
give the required lift in \(W\).

Every stable parabolic object is represented in \(W^{\rm reg}\).
To see this, let \(Q_*\) be a proper quotient of a stable parabolic
bundle \(E_*\). Stability gives \[
\operatorname{par}\mu(Q_*)>\operatorname{par}\mu(E_*).
\]
The sum of the average quotient weights over the \(m\) marked points is less than \(m\), while the corresponding sum for \(E_*\) is nonnegative. Therefore \begin{equation}
\mu(Q)>\frac er-m.
\label{eq:A2.3}
\end{equation}
The slope \(\mu(E)\) also satisfies this lower bound. Consequently \[
\mu_{\min}(E)>\frac er-m>2g-1.
\]
There can be no nonzero map \(E\to K_C(p)\) for any \(p\), because its image would be a rank-one quotient of degree at most \(2g-1\). Thus \(H^1(E(-p))=0\) for all \(p\), and \(E\) is globally generated with \(H^1(E)=0\). The rank-drop incidence used in the proof of Theorem~\ref{thm:A1} supplies \(r-1\) everywhere independent sections. Together with the prescribed flags in their fibres, these sections represent the object by a point of \(W^{\rm reg}\).

\smallskip
\noindent\emph{Relative positions and weights.}
We compare the Schubert codimension of each flag position with its
contribution to the parabolic slope, using Sun's estimates for flag
positions and weights \cite[Lemmas 5.1--5.2]{S1999}. Their difference will supply the
extra unit in the integral codimension bound.
For a saturated subbundle \(F\subset E\) of rank \(s\) and degree
\(f\), put
\[
q=r-s,\qquad A=rf-se\in\mathbb Z.
\]
At a marked point, suppressing the index \(i\), the graded dimensions are
\[
a_h=\dim\operatorname{gr}_h(F),\qquad
b_h=n_h-a_h.
\]
Then \(\sum_h a_h=s\) and \(\sum_h b_h=q\).

The locus of flags having this relative position with a fixed \(s\)-plane has codimension \begin{equation}
c(\mathbf a)=\sum_{u>v}a_u b_v.
\label{eq:A2.4}
\end{equation}
Fix the flag instead and vary the \(s\)-plane. Choose its graded subspaces in each block; these choices have dimension \(\sum_h a_hb_h\). In echelon form, the additional graph entries run from a selected vector in block \(u\) to an unselected vector in a deeper block \(v>u\), contributing \(\sum_{u<v}a_ub_v\) parameters. Thus the relative-position stratum in \(\operatorname{Gr}(s,r)\) has dimension \[
\sum_h a_hb_h+\sum_{u<v}a_ub_v.
\]
Subtracting this from \(sq=\sum_{u,v}a_ub_v\) gives \eqref{eq:A2.4}. The same codimension holds when the subspace is fixed and the flag varies: the incidence of subspace--flag pairs is homogeneous over the Grassmannian, and subtracting its base dimension gives exactly this fibre codimension.

The weight contribution to the stability inequality is \begin{equation}
\begin{aligned}
w(\mathbf a)
&=r\sum_h a_h\alpha_h-s\sum_h n_h\alpha_h\\
&=\sum_{u,v}a_u b_v(\alpha_u-\alpha_v).
\end{aligned}
\label{eq:A2.5}
\end{equation}
Subtracting \eqref{eq:A2.5} from \eqref{eq:A2.4} yields the exact identity \begin{equation}
\begin{aligned}
c(\mathbf a)-w(\mathbf a)
={}&\sum_{u>v}a_ub_v(1-\alpha_u+\alpha_v)\\
&+\sum_{u<v}a_ub_v(\alpha_v-\alpha_u).
\end{aligned}
\label{eq:A2.6}
\end{equation}
Every displayed coefficient is strictly positive because the distinct weights increase and have width less than one.

For a nontrivial flag, \eqref{eq:A2.6} is strictly positive. Indeed, if it were zero, every pair with \(a_u>0\) and \(b_v>0\) would have \(u=v\). Both collections are nonempty because \(s,q>0\). It follows that all \(a_u\) and all \(b_v\) are supported in a single common block. Every other graded multiplicity would then be \(n_h=a_h+b_h=0\), contradicting the positive multiplicities in a flag with at least two blocks. A trivial flag contributes zero.

A full flag has selected weight indices \[
I=\{i_1<\cdots<i_s\}\subset\{1,\ldots,r\}.
\]
The formulas then read
\begin{equation}
c(I)=\sum_{k=1}^s(i_k-k),
\qquad
w(I)=r\sum_{j\in I}\alpha_j-s\sum_{j=1}^r\alpha_j.
\label{eq:A2.7}
\end{equation}
For rank two, the generic line receives the lower weight. Equality with the distinguished flag line gives the higher weight and is a codimension-one condition, which also fixes the direction of the indexing convention.

For equally spaced full-flag weights \(\alpha_j=(j-1)/r\),
\eqref{eq:A2.7} becomes
\[
c(I)=\sum_{j\in I}j-\frac{s(s+1)}2,
\qquad
w(I)=\sum_{j\in I}j-\frac{s(r+1)}2.
\]
Their difference is independent of the relative position:
\begin{equation}
c(I)-w(I)=\frac{sq}{2}.
\label{eq:A2.13}
\end{equation}

The subbundle \(F\) violates parabolic stability exactly when
\begin{equation}
A+\sum_i w_i\ge0.
\label{eq:A2.8}
\end{equation}
At least one nontrivial flag makes \(\sum_i(c_i-w_i)>0\). Combining this with \eqref{eq:A2.8}, and using that \(A+\sum_i c_i\) is an integer, gives \begin{equation}
A+\sum_i c_i\ge1.
\label{eq:A2.9}
\end{equation}

The extension-stack computation leading to \eqref{eq:A1.8} applies to
every degree \(f\). As in Sun's codimension calculation
\cite[Proposition 5.1]{S1999}, adding the relative flag conditions gives
the following bound in the quasi-parabolic bundle stack: \begin{equation}
\begin{aligned}
\operatorname{codim}(s,f,\mathbf a_i)
&\ge sq(g-1)+A+\sum_i c_i\\
&\ge sq(g-1)+1\\
&\ge(r-1)(g-1)+1\ge2.
\end{aligned}
\label{eq:A2.10}
\end{equation}
Thus every destabilizing relative position has codimension at least
two, for every weight system.

Only finitely many ranks and relative-position arrays occur on \(W\). Since \(|w_i|<sq\), \eqref{eq:A2.8} implies \[
f>\frac{se-msq}{r}\ge s\left(\frac er-m\right)>0.
\]
A subbundle of positive degree cannot lie in \(\mathcal O_C^{r-1}\), so the projection to \(M\) gives \(f\le e\), as in \eqref{eq:A1.9}. The relative Quot schemes with these finitely many Hilbert polynomials are of finite type. Smooth pullback preserves \eqref{eq:A2.10}, and \(W\setminus W^{\rm reg}\) already has codimension at least two. The rational-weight reduction above makes parabolic stability Zariski open for each fixed weight system. Therefore, if \(W^s\) denotes the stable locus, \begin{equation}
\operatorname{codim}_W(W\setminus W^s)\ge2.
\label{eq:A2.11}
\end{equation}

\smallskip
\noindent\emph{Oka descent.}
Translates of the affine big cell cover each partial flag variety, and
quasi-compactness selects a finite subcover. Thus \eqref{eq:A2.2} gives
\[
W=\bigcup_{\lambda=1}^N U_\lambda,\qquad
U_\lambda\simeq\mathbb A^{\dim W},\qquad
\operatorname{codim}_{U_\lambda}(U_\lambda\setminus W^s)\ge2,
\]
where the last inequality is \eqref{eq:A2.11}. Every
\(U_\lambda\cap W^s\) is Oka by \cite[§6.4]{F17}, and Kusakabe's
localization theorem \cite[Theorem 1.4]{Kus21} makes \(W^s\) Oka.

The classifying map
\[
W^s\longrightarrow\mathcal P^s(C,\mathbf p;r,\xi,\mathbf n,\boldsymbol\alpha),
\]
which sends \((v,\{F_{i\bullet}\})\) to the class of
\((E_v\otimes L,\{F_{i\bullet}\otimes L_{p_i}\})\), is surjective
by the construction of the flagged extensions and is a submersion by
the section-lifting argument. A stable parabolic bundle
has only scalar parabolic endomorphisms, so specifying its determinant
isomorphism leaves stabilizer \(\mu_r\). The obstruction group is a
second cohomology group of the parabolic endomorphism sheaf on a
curve, hence vanishes. The stable coarse moduli space is therefore
smooth, and passage from the stable stack preserves the surjective
differential. Surjective holomorphic descent
\cite[Theorem 7.5]{DW} proves the fixed-determinant assertion.

Twisting by \(N\in\operatorname{Pic}^0(C)\) preserves all flag types, weights, and stability inequalities. The analogue of \eqref{eq:A1.10} is again the base change of \([r]\) on the Jacobian. This proves the variable-determinant assertion.
When semistability equals stability, both moduli spaces are projective
by the parabolic GIT construction \cite[Theorem 1.1]{S1999}. Their
smoothness and Oka property imply holomorphic ellipticity by
\cite[Theorem 1.1]{FL2026}.
\end{proof}

For \(g\ge2\), Boden--Yokogawa proved that the fixed-determinant
semistable parabolic moduli space is rational whenever some graded multiplicity
\(n_{ih}\) equals one \cite[Section 6]{BodenYokogawa1999}. This includes
full flags. The extension construction above treats arbitrary flag
types through the same codimension estimate. For compact moduli
spaces, its surjectivity also gives the following conclusion in every
genus.

\begin{corollary}\label{cor:compact-parabolic-moduli}
Let \(C\) be a smooth connected projective curve of any genus. Fix a
rank \(r\ge2\), degree \(d\), determinant
\(\xi\in\operatorname{Pic}^d(C)\), flag types, and real weights.
Suppose that the fixed-determinant semistable parabolic moduli space
is nonempty and every semistable object is stable. Then both the
fixed- and variable-determinant moduli spaces are smooth projective
and holomorphically elliptic, and hence Oka.
\end{corollary}

\begin{proof}
Choose the negative twist \(L\) so that \(e/r>2g-1+m\). The
quotient-slope argument above makes every twisted stable bundle
globally generated with \(H^1=0\). The rank-drop incidence supplies
\(r-1\) everywhere independent sections, in every genus. Thus the
extension and flag constructions give a surjective algebraic
classifying map from a nonempty open subset \(W^s\) of
\[
V\times\prod_i\operatorname{Flag}_{\mathbf n_i}(\mathbb C^r)
\]
to the fixed-determinant moduli space. This parameter space is
rational, so its image is unirational. The target is smooth by the
deformation argument above and projective because semistability
equals stability. It is therefore rationally connected, and hence
Oka by \cite[Theorem 1.10]{DW}.

Tensoring by degree-zero line bundles gives the finite étale cover
\eqref{eq:A1.10} in the parabolic setting, so the
variable-determinant space is Oka as well. Tensoring also transfers
the equality of stability and semistability to every determinant, so the
variable-determinant GIT space is smooth and projective. Both
projective Oka manifolds are
holomorphically elliptic by \cite[Theorem 1.1]{FL2026}.
\end{proof}

\subsection{Full flags in genus zero and one}

For equally spaced full-flag weights, the Schubert codimension
exceeds the weight contribution by half the product of the subbundle
and quotient ranks. With enough marked points, this excess compensates
for the term \(g-1\) in the codimension estimate, so the extension
construction also applies in genera zero and one.

\begin{theorem}\label{thm:A2-2}

Let \(C\) be a smooth connected projective curve of genus \(g=0\) or
\(1\), equipped with \(m\) distinct marked points. For any rank
\(r\ge2\), degree \(d\), and determinant
\(\xi\in\operatorname{Pic}^d(C)\), assign to each marked point a full
flag with weights \[
\alpha_{ij}=\frac{j-1}{r},\qquad j=1,\ldots,r.
\]
If \begin{equation}
(r-1)\left(g-1+\frac m2\right)>1,
\label{eq:A2.12}
\end{equation}
then the stable parabolic moduli space is nonempty, smooth, and Oka,
both with fixed and with variable determinant. Moreover, there are
arbitrarily close generic rational weight systems for which semistability
equals stability and the resulting moduli spaces are nonempty, smooth,
projective, and holomorphically elliptic.

In genus zero, condition \eqref{eq:A2.12} holds for \(m\ge5\) in
rank two, \(m\ge4\) in rank three, and \(m\ge3\) in every rank
\(r\ge4\). In genus one, the corresponding bounds are \(m\ge3\),
\(m\ge2\), and \(m\ge1\), respectively.

\end{theorem}

For flag types \(\mathbf n_i\), the fixed-determinant stable moduli
space has dimension
\[
(r^2-1)(g-1)+\sum_i\sum_{h<k}n_{ih}n_{ik};
\]
the variable-determinant space has dimension \(g\) larger. For full
flags the sum is \(mr(r-1)/2\). Thus the cases
\((g,r,m)=(0,2,5)\), \((0,4,3)\), and \((1,4,1)\) in
Theorem~\ref{thm:A2-2} give fixed-determinant spaces of dimensions
two, three, and six, respectively.

In rank two with odd \(m\), semistability already equals stability
at the displayed weights. A line
subbundle of degree \(f\) meeting \(k\) marked flag lines would have
equal parabolic slope only if
\[
f+\frac{k}{2}=\frac d2+\frac m4,
\qquad 4f+2k=2d+m,
\]
whose last equality has an even left side and an odd right side.
In particular,
\((g,r,m)=(0,2,5)\) already gives a smooth projective
holomorphically elliptic surface at the displayed weights.

\begin{proof}
For a rank-\(s\) subbundle, put \(q=r-s\). The identity
\eqref{eq:A2.13} gives the same excess \(sq/2\) at every marked
point. A parabolic destabilizer satisfies \eqref{eq:A2.8}. Since
\(A+\sum_i c(I_i)\) is an integer, \[
A+\sum_i c(I_i)\ge\left\lceil\frac{msq}{2}\right\rceil.
\]
The same extension-stack computation therefore gives \begin{equation}
\operatorname{codim}(\text{bad type})
\ge sq(g-1)+\left\lceil\frac{msq}{2}\right\rceil
=\left\lceil sq\left(g-1+\frac m2\right)\right\rceil.
\label{eq:A2.14}
\end{equation}
Under \eqref{eq:A2.12}, the coefficient \(g-1+m/2\) is positive and the smallest value of \(s(r-s)\) is \(r-1\). Thus every number in \eqref{eq:A2.14} is at least two.

We next construct the extension atlas in these genera. If \(C=\mathbb P^1\) and \(e\ge0\), then for every point \(p\), \[
H^1(\mathcal O_C(-p))=0,
\qquad
H^1(M(-p))=0.
\]
The cohomology sequence of \eqref{eq:A1.2} shows \(H^1(E_v(-p))=0\) for every \(v\), so \(V^{\rm reg}=V\). The dual-witness space vanishes in genus zero. In genus one, the argument leading to \eqref{eq:A1.4} applies with \(h^0(K_C)=1\), and the bound \(e-r\ge2\) makes its exceptional set have codimension at least two.

Choose \(e\) large enough that \[
\frac er>m,\qquad \frac er>2g-1+m,
\]
and also \(e-r\ge2\) in genus one. These inequalities give the
quotient-slope bound \eqref{eq:A2.3}, the incidence estimate for
everywhere independent sections, and the section-lifting argument.
Destabilizing degrees lie in the finite interval \(0<f\le e\).
Together with \eqref{eq:A2.14}, the argument for
\eqref{eq:A2.11} therefore gives a complement of codimension at least
two in the nonempty irreducible space \(W\). Its stable locus is
nonempty and Oka, and the classifying map gives the claimed moduli
spaces.

For the generic-weight assertion, fix a stable object at the displayed
weights. Stability of this object persists in a neighbourhood of the
weight system: only finitely many walls meet the bounded weight
polytope, and its relevant stability inequalities are strict. A small
common translation of the weights at each marked point preserves
stability and moves the system into the interior. For full flags, every
stability wall is a proper hyperplane, since its coefficients are
\(r-s\) at selected indices and \(-s\) at unselected ones. Thus arbitrarily
close generic rational systems give nonempty moduli spaces with
semistability equal to stability. Corollary~\ref{cor:compact-parabolic-moduli}
makes these spaces smooth, projective, and holomorphically elliptic.

\end{proof}

\begin{example}\label{ex:four-parabolic-points}
For rank two, degree zero, four ordered points on \(\mathbb P^1\),
and weights \((0,1/2)\), the stable parabolic moduli space is
\[
\mathbb P^1\setminus\{0,1,\infty,\kappa\},
\]
where \(\kappa\notin\{0,1,\infty\}\) is the cross-ratio of the marked
points \cite[Proposition 5.1]{EFK2022}. This hyperbolic curve is
not Oka.

Indeed, write \(E=\mathcal O(a)\oplus\mathcal O(-a)\), \(a\ge0\).
Its parabolic slope is one, so stability forces \(a=0\).
Constant sublines then force the four flag lines to be distinct.
Their cross-ratio parametrizes
\(\mathbb P^1\setminus\{0,1,\infty\}\). A saturated subline
\(\mathcal O(-1)\) gives a degree-one map between the two projective
lines; it meets all four flags exactly when their cross-ratio is
\(\kappa\), producing equal parabolic slope. Every other such
subline has slope at most \(-1+3/2<1\), and a subline
\(\mathcal O(-a)\), \(a\ge2\), has slope at most \(-a+2<1\).
Thus precisely one further parameter is excluded.
Here \((r-1)(g-1+m/2)=1\), at the boundary of
\eqref{eq:A2.12}.
\end{example}

\section{Holomorphic maps and embeddings}\label{sec:mapping-spaces}

Mapping and embedding spaces carry their natural complex structures,
characterized by holomorphic families, and the compact-open topology;
see \cite[\S2.2.1]{Bowman2016}.
A family of globally generated bundles over an Oka base gives spaces
of generating tuples. An entire lift of a dominating map to the base
provides the horizontal directions; complete translations of the
sections provide the vertical directions. We first carry out this
construction, then apply it to projective maps, embeddings, and stable
quotient bundles on curves. Rational and elliptic functions admit
separate constructions in smaller target dimension.

\subsection{Generating sections over an Oka base}\label{sec:generating-sections}

Let \(X\) be a compact complex manifold of dimension \(k\ge1\),
\(B\) an Oka manifold of dimension \(b\), and
\(\mathcal F\to X\times B\) a holomorphic vector bundle of rank \(r\).
Let \(V\to B\) be a rank-\(R\) holomorphic vector bundle equipped with
a surjective evaluation morphism
\[
\epsilon:p^*V\longrightarrow\mathcal F,
\qquad p:X\times B\longrightarrow B.
\]
Each element of \(V_\beta\) gives a section of \(\mathcal F_\beta\).
For this evaluation family, set
\[
\operatorname{Gen}_N(\mathcal F,V/B)
=\{(\beta,s_1,\ldots,s_N)\in V^{\oplus N}:
s_1,\ldots,s_N\text{ generate }\mathcal F_\beta\text{ on }X\}.
\]
Here generation is tested after applying \(\epsilon\). The rank-drop
incidence is closed analytic over the compact manifold \(X\), so its
projection to \(V^{\oplus N}\) is proper. Its image has complement
\(\operatorname{Gen}_N(\mathcal F,V/B)\). When \(V=p_*\mathcal F\)
commutes with base change, this is the space of generating tuples of
sections of the fibres.

\begin{theorem}\label{thm:generating-sections}
If
\begin{equation}\label{eq:generating-sections-bound}
N\ge r+k+\max\{1,b\},
\end{equation}
then \(\operatorname{Gen}_N(\mathcal F,V/B)\) is nonempty, strongly
dominable, and Oka.
\end{theorem}

\begin{proof}
For a fixed \(\beta\in B\), the evaluation family is surjective at
every \(x\in X\). The rank-drop incidence in
\(X\times V_\beta^{\oplus N}\) therefore has codimension
\(N-r+1\). Its proper image in \(V_\beta^{\oplus N}\) has
codimension at least \(N-r+1-k>0\), so generating tuples exist.

Fix \(p_0=(\beta_0,s_1^0,\ldots,s_N^0)\) in the generation space.
We will construct an entire map through \(p_0\) with invertible
differential. In \eqref{eq:generating-sections-bound}, the inequality
\(N\ge r+k+1\) permits a constant change of sections after which
deleting any one preserves generation at the prescribed tuple.
For \(b>0\), the inequality \(N\ge r+k+b\) gives the incidence
estimate for an entire lift over the base. Complete vertical flows
then supply the remaining directions.

\smallskip
\noindent\emph{Redundant generating tuples.}
At \(x\in X\), the prescribed tuple gives a surjection
\[
q_x:\mathbb C^N\longrightarrow(\mathcal F_{\beta_0})_x.
\]
If \(0\ne\ell\in(\mathbb C^N)^*\), then \(q_x|_{\ker\ell}\) fails to be
surjective exactly when
\begin{equation}\label{eq:bad-hyperplane}
\ell\in\operatorname{im}q_x^*.
\end{equation}
Indeed, failure means that a nonzero functional on
\((\mathcal F_{\beta_0})_x\) annihilates \(q_x(\ker\ell)\), so its pullback
is a nonzero multiple of \(\ell\). The bad hyperplanes are therefore
parametrized by the image of
\[
\mathbb P(\mathcal F_{\beta_0}^*)
\longrightarrow\mathbb P((\mathbb C^N)^*).
\]
This image is closed analytic and has dimension at most
\(k+r-1<N-1\). Choose a basis \(\ell_1,\ldots,\ell_N\) of
\((\mathbb C^N)^*\) whose projective points avoid it, successively avoiding
both this proper analytic subset and the span of the preceding vectors.
In the dual basis, deleting any one section leaves a generating tuple.
The corresponding fixed \(\mathrm{GL}_N(\mathbb C)\)-change is an
automorphism of the generation space, so we work with this tuple and undo
the change at the end.

\smallskip
\noindent\emph{An entire lift over the base.}
Since \(B\) is Oka, there is a holomorphic map
\[
\varphi:\mathbb C^b\longrightarrow B,\qquad
\varphi(0)=\beta_0,\qquad d\varphi_0\text{ an isomorphism}.
\]
The Oka--Grauert principle \cite[Theorem 5.3.1]{F17} trivializes
\(\varphi^*V\). Choose an entire frame
\(e_1(u),\ldots,e_R(u)\), and write
\(s_j^0=\sum_a a_{ja}e_a(0)\). Suppose first that \(b\ge1\). With
\(M=NbR\) parameters \(\mathbf c=(c_{j\ell a})\), consider
\begin{equation}\label{eq:generating-parameter-family}
s_j(u;\mathbf c)=\sum_{a=1}^R
\left(a_{ja}+\sum_{\ell=1}^b c_{j\ell a}u_\ell\right)e_a(u),
\qquad 1\le j\le N.
\end{equation}
Every such family agrees with the prescribed tuple at \(u=0\).

On \(X\times(\mathbb C^b\setminus\{0\})\times\mathbb C^M\), the
evaluation matrix with columns \(s_j(u;\mathbf c)(x)\) has, in any local
trivialization, a differential in the \(\mathbf c\)-directions that
surjects onto \(\operatorname{Mat}_{r\times N}(\mathbb C)\): some
\(u_\ell\) is nonzero, and the evaluated frame spans
\((\mathcal F_{\varphi(u)})_x\). The matrices of rank \(t\) form a smooth
stratum of codimension \((r-t)(N-t)\). On a chart where a
\(t\times t\) minor \(A\) is invertible, writing a matrix as
\[
\begin{pmatrix}A&B'\\ C'&D\end{pmatrix}
\]
shows that rank \(t\) is equivalent to \(D=C'A^{-1}B'\), which gives the
claimed codimension; the rank-zero stratum is the zero matrix. Hence the
inverse image of every nonsurjective rank stratum has dimension at most
\begin{equation}\label{eq:generating-incidence-bound}
M+k+b-(N-r+1)<M,
\end{equation}
where the strict inequality follows from \(N\ge r+k+b\).

Each incidence stratum has a countable atlas of real dimension less
than \(2M\). Sard's theorem makes the image of each chart in
\(\mathbb C^M\) a null set. Taking the countable union over charts
and the finite union over ranks still gives a null set; this covers
all \(u\ne0\) in the noncompact base. A parameter \(\mathbf c\)
outside that image gives generation for every \(u\ne0\), and the
tuple at zero remains the prescribed generating tuple. Thus we obtain
an entire section
\[
\sigma:\mathbb C^b\longrightarrow
\varphi^*\operatorname{Gen}_N(\mathcal F,V/B)
\]
through the lift of \(p_0\). If \(b=0\), use the constant section.

\smallskip
\noindent\emph{Complete vertical flows.}
The chosen frame identifies the pulled-back generation space with an open
subset \(\widetilde G\subset\mathbb C^b\times(\mathbb C^R)^N\). For
\(1\le j\le N\), let \(A_j\) be the subset of
\[
T_j=\mathbb C^b\times(\mathbb C^R)^{N-1}
\]
where the sections other than the \(j\)th fail to generate. The
corresponding rank-drop incidence in \(X\times T_j\) is closed analytic
and projects properly to \(T_j\); Remmert's theorem \cite[Chapter 2]{F17} makes \(A_j\) closed
analytic. The point \(p_{0,\widehat j}\) obtained by deleting the
\(j\)th section lies outside \(A_j\). Cartan's theorem A \cite[§2.6]{F17} applied to its
ideal sheaf gives
\[
h_j\in\mathcal O(T_j),\qquad h_j|_{A_j}=0,\qquad
h_j(p_{0,\widehat j})\ne0.
\]

Writing \(s_{ja}\) for the frame coordinates, consider
\[
V_{ja}=h_j(u,s_{\widehat j})\frac{\partial}{\partial s_{ja}},
\qquad 1\le j\le N,\quad1\le a\le R.
\]
The flow is
\begin{equation}\label{eq:generating-shear-flow}
\Phi_{ja}^t:\quad
s_j\longmapsto s_j+t\,h_j(u,s_{\widehat j})e_a.
\end{equation}
The coefficient is independent of \(s_j\), so the flow exists for every
\(t\in\mathbb C\) and is inverted by replacing \(t\) with \(-t\). It
preserves \(\widetilde G\): if the remaining sections generate, changing
\(s_j\) preserves generation; otherwise \(h_j=0\). At the lifted point
\(p_0\), these \(NR\) fields span the vertical tangent space.

Order the fields as \(V_1,\ldots,V_{NR}\) and compose their flows with
\(\sigma\):
\[
F(u,t)=\Phi_{NR}^{t_{NR}}\circ\cdots\circ\Phi_1^{t_1}(\sigma(u)):
\mathbb C^{b+NR}\longrightarrow\widetilde G.
\]
At \((0,0)\), its differential is
\begin{equation}\label{eq:generating-full-rank-block}
dF_{(0,0)}=
\begin{pmatrix}
I_b&0\\
d(\sigma_2)_0&
\operatorname{diag}\bigl(h_1(p_{0,\widehat1})I_R,\ldots,
h_N(p_{0,\widehat N})I_R\bigr)
\end{pmatrix},
\end{equation}
where \(\sigma_2\) is the component of \(\sigma\) in
\((\mathbb C^R)^N\). This matrix is invertible; when \(b=0\), only
the lower-right block occurs.
Since \(d\varphi_0\) is invertible, the natural map from the pullback
to \(\operatorname{Gen}_N(\mathcal F,V/B)\) is locally biholomorphic
at the lifted point. Composing \(F\) with this map and undoing the
constant change of sections gives strong dominability at the original
\(p_0\). The strong dominability criterion \cite[Theorem 1.4]{DW}
now makes the generation space Oka.
\end{proof}

For a curve, the bound in \eqref{eq:generating-sections-bound} becomes
\(N\ge\max\{r+2,r+b+1\}\).

We will use the following pointwise observation for quotient
presentations. If \(q:Z\to Y\) is holomorphic, \(dq_z\) is surjective,
and \(Z\) is strongly dominable at \(z\), compose a dominating entire
map with \(q\) and restrict its source to a linear complement of the
differential kernel. This proves strong dominability of \(Y\) at \(q(z)\).

\subsection{Maps to projective space}\label{sec:projective-maps}

For projective maps, take a smooth connected projective variety
\(X\) of dimension \(k\ge1\), put
\[
q=h^1(X,\mathcal O_X),
\]
and fix a nonempty component \(B=\operatorname{Pic}^{\alpha}(X)\).
Suppose every \(L\in B\) is globally generated and
\begin{equation}
h^0(X,L)=R\qquad\text{for all }L\in B
\label{eq:projective-map-vanishing}
\end{equation}
for a fixed integer \(R\). Vanishing of \(H^1(X,L)\) throughout \(B\)
is a sufficient condition for this constancy.

\begin{theorem}\label{thm:projective-mapping-space}
If
\begin{equation}
m\ge k+\max\{1,q\},
\label{eq:projective-map-bound}
\end{equation}
then the full mapping space
\[
\operatorname{Hol}_{\alpha}(X,\mathbb P^m)
=\{f:X\to\mathbb P^m:f^*\mathcal O(1)\in B\}
\]
is smooth and Oka.
\end{theorem}

\begin{proof}
Normalize a Poincar\'e line bundle \(\mathcal P\) on \(X\times B\)
at a fixed point of \(X\). Grauert's cohomology and base change theorem
\cite[Section 7, Satz 5]{Grauert60}, applied to
\eqref{eq:projective-map-vanishing}, shows that
\[
E=p_*\mathcal P
\]
is a vector bundle with fibre \(H^0(X,L)\), that its formation commutes with base change, and that the evaluation map \(p^*E\to\mathcal P\) is surjective. Let
\[
\operatorname{Gen}_{m+1}(\mathcal P,E/B)\subset E^{\oplus(m+1)}
\]
be the open set of tuples that generate the corresponding line bundle at every point of \(X\).

A map to \(\mathbb P^m\) is represented by a line bundle and
\(m+1\) generating sections, up to a common nonzero scalar.
For a family with classifying map to \(B\), its line bundle differs
from the pulled-back normalized Poincar\'e bundle by a line bundle
from the parameter space. Projectivization absorbs this factor,
giving the isomorphism
\begin{equation}
\operatorname{Hol}_{\alpha}(X,\mathbb P^m)
\cong \operatorname{Gen}_{m+1}(\mathcal P,E/B)/\mathbb C^*,
\label{eq:projective-map-quotient}
\end{equation}
of natural complex spaces, as in Aumonier
\cite[Proposition 2.1]{A2024}. In particular, the mapping space is a
smooth open subset of \(\mathbb P(E^{\oplus(m+1)})\).

The base \(B\) is a complex torus, hence Oka. The numerical condition
\eqref{eq:projective-map-bound} is precisely the rank-one case of
Theorem~\ref{thm:generating-sections} for \(m+1\) sections, so the
generating-tuple space is strongly dominable. The scalar quotient in
\eqref{eq:projective-map-quotient} is a holomorphic submersion, so the
preceding observation proves strong dominability of the mapping space.
The criterion \cite[Theorem 1.4]{DW} gives the Oka property.
\end{proof}

Sufficiently positive line bundles satisfy these hypotheses uniformly
over the Picard component. If \(L_0\) and \(A\) are line bundles on \(X\), with \(A\) ample, relative Serre vanishing \cite[Tag 01XO]{A2026} on \(X\times\operatorname{Pic}^0(X)\) gives an integer \(N_0\) such that every
\[
L_0\otimes A^N\otimes P,
\qquad N\ge N_0,\quad P\in\operatorname{Pic}^0(X),
\]
has vanishing first cohomology and is globally generated. Applying the same argument after tensoring with the universal ideal of the proper length-two Hilbert scheme makes these bundles very ample as well, as in Aumonier's argument \cite[Lemma 3.7]{A2024}.

\begin{corollary}\label{cor:curve-projective-maps}
Let \(C\) be a smooth connected projective curve of genus \(g\). If
\[
d\ge2g,
\qquad
m\ge\max\{2,g+1\},
\]
then \(\operatorname{Hol}_d(C,\mathbb P^m)\) is Oka.
\end{corollary}

\begin{proof}
Every degree-\(d\) line bundle \(L\) is globally generated and has \(H^1(C,L)=0\), because for every \(x\in C\),
\[
H^1(C,L(-x))\cong
H^0(C,K_C\otimes L^{-1}(x))^*=0.
\]
The vanishing follows from
\(\deg(K_C\otimes L^{-1}(x))=2g-1-d<0\).
Theorem~\ref{thm:projective-mapping-space} applies with \(k=1\) and
\(q=g\).
\end{proof}

\subsection{Rational functions}\label{sec:rational-functions}

For \(d\ge1\), write
\[
R_d=\operatorname{Hol}_d(\mathbb P^1,\mathbb P^1).
\]
The space \(R_2\) is homogeneous
\cite[Theorem 1.1 and Corollary 1.2]{Hanysz2015}. Hanysz proved that
\(R_3\) is strongly dominable and strongly \(\mathbb C\)-connected
\cite[Theorems 1.7--1.8]{Hanysz2015}.
Xie subsequently proved holomorphic ellipticity of \(R_3\)
\cite[Theorem 1.3]{XieCubic}. The proposition below answers the
degree-by-degree Oka question posed by Hanysz \cite[\S1]{Hanysz2015}.
For arbitrary positive degree, Euclidean coordinates give a dense Oka
open subset of the full mapping space.

\begin{proposition}\label{prop:rational-functions-oka}
For every \(d\ge1\), the manifold \(R_d\) is smooth affine of dimension
\(2d+1\), Oka, and holomorphically elliptic. For any \(p,q\in\mathbb P^1\),
the based space
\[
\{f\in R_d:f(p)=q\}
\]
is smooth affine of dimension \(2d\), Oka, and holomorphically elliptic.
\end{proposition}

\begin{proof}
Pairs of degree-\(d\) homogeneous polynomials in two variables, modulo
a common nonzero scalar, identify \(R_d\) with the open subset of
\(\mathbb P^{2d+1}\) on which the resultant is nonzero. The resultant
is homogeneous of degree \(2d>0\), so this open subset is affine,
smooth, and of dimension \(2d+1\).

Source and target automorphisms identify all the based spaces.
The fibre of evaluation at infinity over \(0\) is
\[
\mathcal R_d=
\{(q,p):q\text{ monic of degree }d,\ \deg p<d,
\ \operatorname{Res}(q,p)\ne0\}\subset\mathbb C^{2d},
\]
under \((q,p)\mapsto p/q\). The monic normalization of \(q\)
makes this representation unique for maps taking \(\infty\) to
\(0\).

\smallskip
\noindent\emph{Euclidean coordinates.}
We parametrize the locus where each Euclidean division lowers the
degree by one. The linear quotients supply affine parameters, while
the remainder normalizations supply nonzero parameters.
We use the linear-partial-quotient case of Cazanave's decomposition
\cite[Lemma A.1]{FarbWolfson2017}, applied to \(q/p\) and written with
monic remainders.
For
\[
b_1,\ldots,b_d\in\mathbb C,
\qquad a,c_1,\ldots,c_{d-1}\in\mathbb C^*,
\]
set
\begin{equation}\label{eq:rational-functions-recursion}
r_d=1,\qquad r_{d-1}=z-b_d,\qquad
r_{i-1}=(z-b_i)r_i+c_i r_{i+1}
\quad(i=d-1,\ldots,1).
\end{equation}
Backward induction shows that \(r_i\) is monic of degree \(d-i\) and
that adjacent remainders are coprime. Hence
\((q,p)=(r_0,a r_1)\) lies in \(\mathcal R_d\).

Let \(U_d\subset\mathcal R_d\) consist of pairs with \(\deg p=d-1\)
and, when \(d\ge2\), successive nonzero monic remainders of degrees
\(d-2,\ldots,0\). Monic division recovers all the parameters:
\(a=\operatorname{lc}(p)\), \(r_0=q\), \(r_1=p/a\), and
\[
r_{i-1}=(z-b_i)r_i+\widetilde r_{i+1},\qquad
c_i=\operatorname{lc}(\widetilde r_{i+1}),\qquad
r_{i+1}=\widetilde r_{i+1}/c_i.
\]
Here \(\operatorname{lc}\) is the leading coefficient; the last
parameter is recovered from \(r_{d-1}=z-b_d\). At each division,
the required degree is imposed by the nonvanishing of the leading
coefficient. Division is regular on that open set, so \(U_d\) is
Zariski open and the two constructions are algebraic inverses:
\[
U_d\cong\mathbb C^d\times(\mathbb C^*)^d.
\]
For \(d=1\), the recursion is empty and
\(q=z-b_1\), \(p=a\ne0\), so
\(U_1=\mathcal R_1\cong\mathbb C\times\mathbb C^*\).
The open set \(U_d\) is nonempty and Oka. Since
\(\mathcal R_d\) is irreducible, analytic extension
\cite[Theorem 1.9]{DW} makes \(\mathcal R_d\) Oka.
Evaluation at infinity is a holomorphic fibre bundle
\[
\operatorname{ev}_\infty:R_d\longrightarrow\mathbb P^1
\]
with fibre \(\mathcal R_d\). Indeed, over the target chart
\(b=f(\infty)\in\mathbb C\), its trivialization is
\[
f=b+\frac pq,\qquad (q,p)\in\mathcal R_d.
\]
The monic normalization makes \(b,q,p\) depend algebraically on \(f\);
a target M\"obius transformation gives the other local trivialization.
Since the base and fibre are Oka, the bundle theorem
\cite[Corollaries 2.51--2.52]{For} makes \(R_d\) Oka.
Both spaces are affine and Stein, so they are
holomorphically elliptic \cite[Proposition 5.6.15]{F17}.
\end{proof}

\subsection{Elliptic functions}\label{sec:elliptic-functions}

For a complex elliptic curve \(E\), consider
\[
M_d=\operatorname{Hol}_d(E,\mathbb P^1).
\]
Bowman's homogeneity theorem treats \(M_2\)
\cite[Theorem 24]{Bowman2016}. For \(M_3\), his six-sheeted
branched surjection from a smooth Oka manifold
\cite[Theorem 34]{Bowman2016} gives the Oka property by holomorphic
descent \cite[Theorem 7.5]{DW}; Xie proved the stronger holomorphic
ellipticity statement \cite[Theorem 1.2]{XieCubic}. Bowman also proposed
studying the flexibility of these spaces in higher degree
\cite[\S1.3]{Bowman2016}. The construction below gives their Oka
property by exponential interpolation on two ramified fibres and
applies to every \(d\ge3\).

\begin{proposition}\label{prop:elliptic-functions-oka}
For every complex elliptic curve \(E\) and every \(d\ge2\), the
manifold \(M_d\) is smooth, irreducible, and Oka, of dimension \(2d\).
For every \(L\in\operatorname{Pic}^d(E)\), the submanifold
\[
M_d(L)=\{f\in M_d:f^*\mathcal O_{\mathbb P^1}(1)\cong L\}
\]
is smooth affine of dimension \(2d-1\) and holomorphically elliptic.
\end{proposition}

\begin{proof}
For \(d\ge2\), every degree-\(d\) line bundle \(L\) is globally
generated and has \(H^1(E,L)=0\): the positive degree of \(L(-p)\)
gives vanishing first cohomology for every \(p\in E\). A nonzero
section has finitely many zeros, all avoided by a general second
section. The normalized Poincar\'e line bundle on
\(E\times\operatorname{Pic}^d(E)\) therefore has a rank-\(d\)
direct image \(\mathcal V\). As in
\eqref{eq:projective-map-quotient}, \(M_d\) is therefore a nonempty
open subset of \(\mathbb P(\mathcal V\oplus\mathcal V)\), so it is
smooth and irreducible of dimension \(2d\). The forgetful map
\[
\alpha:M_d\longrightarrow\operatorname{Pic}^d(E)
\]
is a submersion. At a generating pair \((a,b)\) of sections of \(L\),
its vertical tangent space is
\begin{equation}\label{eq:elliptic-functions-vertical}
\bigl(H^0(E,L)\oplus H^0(E,L)\bigr)/\mathbb C(a,b).
\end{equation}
The fibre \(M_d(L)\) is affine. Indeed, the common-zero incidence
over \(E\) in \(\mathbb P(H^0(E,L)^{\oplus2})\) is a
\(\mathbb P^{2d-3}\)-bundle. Its projection has finite fibres, since
a nonzero section has finitely many zeros, and its image is the
proper hypersurface of nongenerating pairs. The complement is
\(M_d(L)\).

Bowman's homogeneity theorem gives the Oka assertion for \(d=2\).
In that degree a generating pair is a basis of \(H^0(E,L)\), so
\(M_2(L)\cong\mathrm{PGL}_2(\mathbb C)\) is holomorphically elliptic.
For \(d\ge3\), we construct local centered entire families whose
velocities span the tangent space outside a proper algebraic subset.

\smallskip
\noindent\emph{Two ramified fibres.}
The full tangent space at \(f\in M_d\) is
\(H^0(E,f^*T\mathbb P^1)=H^0(E,L^2)\). For every effective divisor
\(Z\) of length at most four,
\[
H^1(E,L^2(-Z))=0,
\]
because \(2d-\deg Z>0\). Tangent sections therefore realize arbitrary
jets on \(Z\). In \(E\times M_d\), the conditions
\(f'(p)=f''(p)=0\) have independent differentials in the map
variable: in local coordinates these differentials are
\((\eta'(p),\eta''(p))\) for \(\eta\in H^0(E,L^2)\), and length-three
jet evaluation is surjective. This locus has dimension
at most \(2d-1\); its image is closed because \(E\) is
proper. On \((E\times E\setminus\Delta)\times M_d\), the conditions
\[
f'(p)=f'(q)=0,\qquad f(p)=f(q)
\]
have map-variable differentials
\((\eta'(p),\eta'(q),\eta(p)-\eta(q))\). Evaluation on \(2p+2q\)
makes them independent. Their locus has dimension at most \(2d-1\), and the
Zariski closure of its image has dimension at most \(2d-1\).
The jet conditions are algebraic and independent of the local
coordinates. The first image and the closure of the second image
therefore form a proper closed algebraic subset \(A\subset M_d\).

A map \(f\notin A\) has \(2d\) simple ramification points and
distinct branch values, by Riemann--Hurwitz. Fix a target coordinate
in which two chosen branch values are finite. On a sufficiently small
neighbourhood \(V\) of \(f\), the corresponding critical points vary
holomorphically, as do their critical values \(c_1(y),c_2(y)\).
Take holomorphic
generating representatives \((a_y,b_y)\) of \(y\) on \(V\), with
\(L_y=y^*\mathcal O_{\mathbb P^1}(1)\). The sections
\[
q_i(y)=a_y-c_i(y)b_y,\qquad i=1,2,
\]
each have one double zero and \(d-2\) simple zeros. After shrinking
\(V\), their \(d-1\) distinct support points can be labelled
holomorphically.

\smallskip
\noindent\emph{Entire interpolation.}
Fix \(i\), write \(q_y=q_i(y)\) and \(r_y=b_y\), and let
\(D_{i,\mathrm{red}}(y)\) be the reduced zero divisor of \(q_y\).
Simple ramification lowers the length of this support from \(d\) to
\(d-1\). The remaining line bundle
\(L_y(-D_{i,\mathrm{red}}(y))\) has degree one, so its first cohomology
vanishes on the elliptic curve. This permits arbitrary interpolation
of values at the support points: the restriction map
\[
\operatorname{ev}_y:H^0(E,L_y)\longrightarrow\mathbb C^{d-1}
\]
is therefore surjective, where values are expressed in local
trivializations at the moving support points. The evaluation maps vary holomorphically, so after restricting \(V\)
they admit a holomorphic right inverse \(S_y\).
For a holomorphic family of sections \(u_y\), write \(r_j(y)\) and
\(u_j(y)\) for the corresponding evaluations. Every \(r_j(y)\)
is nonzero. We keep \(q_y\) fixed and interpolate the values of
\(r_y\) by exponentials, so that they stay nonzero for all
\(t\in\mathbb C\). The lift by \(S_y\), together with a correction
in \(\ker\operatorname{ev}_y\), gives initial velocity \(u_y\):
\begin{equation}\label{eq:elliptic-functions-interpolation}
\begin{split}
r_y(t)={}&r_y+
S_y\bigl((r_j(y)(e^{t u_j(y)/r_j(y)}-1))_j\bigr)\\
&+t\bigl(u_y-S_y(\operatorname{ev}_y u_y)\bigr).
\end{split}
\end{equation}
This is holomorphic on the full product \(V\times\mathbb C\), with
\[
r_y(0)=r_y,\qquad \partial_t r_y(0)=u_y,\qquad
\operatorname{ev}_j r_y(t)=r_j(y)e^{t u_j(y)/r_j(y)}\ne0.
\]
The nonzero values at every zero of \(q_y\) show that
\((q_y,r_y(t))\) generates \(L_y\) for all \(t\in\mathbb C\).
Returning to the fixed target coordinate gives the pair
\[
\bigl(q_y+c_i(y)r_y(t),\ r_y(t)\bigr)
\]
which represents a degree-\(d\) map and equals \((a_y,b_y)\) at
\(t=0\). In this deformation \(c_i(y)\) is fixed as \(t\) varies.
A local basis of sections for \(u_y\) gives \(d\) centered entire
families for each \(i\).

\smallskip
\noindent\emph{Tangent spanning.}
The velocities of the \(i\)th group in
\eqref{eq:elliptic-functions-vertical} are
\[
[(c_i u,u)],\qquad u\in H^0(E,L).
\]
The two branch values are distinct, so these groups are complementary
before taking the scalar quotient: the linear map
\[
(u,v)\longmapsto(c_1u+c_2v,\ u+v)
\]
is invertible on \(H^0(E,L)^2\). After the scalar quotient, the two
groups therefore span the \((2d-1)\)-dimensional vertical tangent
space.

For fixed \(L\), the same interpolation families lie in \(M_d(L)\).
The bad ramification conditions defining \(A\) are invariant under
source translation. With \(\tau_t(x)=x+t\), the map
\[
\theta_L:E\longrightarrow\operatorname{Pic}^d(E),
\qquad t\longmapsto\tau_t^*L
\]
is multiplication by \(-d\) after taking \(L\) as the origin of
\(\operatorname{Pic}^d(E)\). It is surjective, so the nonempty good
locus meets every fibre \(M_d(L)\). The two interpolation groups span
the tangent space of that fibre. Restricting their initial parameter
to \(V\cap M_d(L)\) gives centered entire families there for every
complex time. Lemma~2.1 and Theorem~1.9 of~\cite{DW} make \(M_d(L)\)
Oka, and its affine Stein structure gives holomorphic ellipticity
\cite[Proposition 5.6.15]{F17}.

The degree-\(d^2\) isogeny \(\theta_L\) trivializes the family
\(\alpha:M_d\to\operatorname{Pic}^d(E)\) after base change:
\[
E\times M_d(L)\ \xrightarrow{\ \sim\ }\
E\times_{\operatorname{Pic}^d(E)}M_d,
\qquad (t,f)\longmapsto(t,f\circ\tau_t).
\]
Its inverse sends \((t,h)\) to \((t,h\circ\tau_{-t})\).
Consequently \(E\times M_d(L)\to M_d\) is a finite unramified covering.
Its source is Oka, so covering permanence
\cite[Corollaries 2.51--2.52]{For} gives the Oka property of \(M_d\).
\end{proof}

Together with Corollary~\ref{cor:curve-projective-maps}, these
constructions treat all projective-space targets for compact Oka curves.

\begin{corollary}\label{cor:oka-curve-projective-maps}
Let \(C=\mathbb P^1\) or let \(C\) be a complex elliptic curve. For every
\(m\ge1\), each nonempty degree component
\(\operatorname{Hol}_d(C,\mathbb P^m)\) is Oka. The nonempty degrees are
\(d\ge0\) for \(\mathbb P^1\), and \(d=0\) or \(d\ge2\) for an elliptic
curve.
\end{corollary}

\begin{proof}
For \(m\ge2\), use Corollary~\ref{cor:curve-projective-maps}; for
\(m=1\) and positive degree, use
Propositions~\ref{prop:rational-functions-oka}
and~\ref{prop:elliptic-functions-oka}. In degree zero the mapping space
is \(\mathbb P^m\). A degree-one line bundle on an elliptic curve has
only one independent section and is not globally generated, so it
cannot be the pullback of \(\mathcal O_{\mathbb P^m}(1)\) by a map.
\end{proof}

\subsection{Embeddings into projective space}\label{sec:projective-embeddings}

Let \(X\) be a smooth connected projective variety of dimension
\(k\ge1\), set \(q=h^1(X,\mathcal O_X)\), and fix a nonempty component
\(B=\operatorname{Pic}^{\alpha}(X)\). Suppose every \(L\in B\) is very
ample and \(h^0(X,L)\) is constant on \(B\), as in
\eqref{eq:projective-map-vanishing}. For the projection
\(p:X\times B\to B\), retain the normalized Poincaré bundle
\(\mathcal P\) and \(E=p_*\mathcal P\) from
Subsection~\ref{sec:projective-maps}; write \(R=\operatorname{rank}E\).
We consider embeddings \(f:X\to\mathbb P^m\) with
\(f^*\mathcal O(1)\in B\). They are detected by separation on length-two
subschemes, which replaces pointwise generation in the preceding
construction.

\begin{theorem}\label{thm:projective-embedding-space}
If
\begin{equation}
m\ge2k+\max\{2,q+1\},
\label{eq:projective-embedding-bound}
\end{equation}
then the space \(\operatorname{Emb}_{\alpha}(X,\mathbb P^m)\) of parametrized closed embeddings is smooth and Oka.
\end{theorem}

\begin{proof}
Put \(H=X^{[2]}\). This is a smooth projective manifold of dimension
\(2k\), with the description
\[
H\cong\operatorname{Bl}_{\Delta}(X\times X)/\mathfrak S_2
\]
recalled in \cite[Introduction, p.~v]{Gottsche94}. A tuple
\(s_0,\ldots,s_m\) of sections of \(L\) defines a closed embedding
precisely when
\begin{equation}\label{eq:length-two-test}
\mathbb C^{m+1}\longrightarrow H^0(Z,L|_Z)
\end{equation}
is surjective for every \(Z\in H\). This is the length-two separation
criterion \cite[Definition 4.1.1 and Remark 4.1.2(1)]{Gottsche94}:
reduced schemes separate points, and nonreduced schemes test tangent
directions.

Let \(\mathcal Z\subset X\times H\) be the universal length-two
subscheme, and let
\(\rho:\mathcal Z\times B\to H\times B\) be the finite flat projection.
Pulling \(\mathcal P\) back from \(X\times B\), form the rank-two
tautological bundle
\[
\mathcal T=\rho_*(\mathcal P|_{\mathcal Z\times B})
\longrightarrow H\times B.
\]
Its fibre at \((Z,L)\) is \(H^0(Z,L|_Z)\). For the projection
\(p_H:H\times B\to B\), very ampleness gives a surjective evaluation
\[
p_H^*E\longrightarrow\mathcal T.
\]
Thus the embedding tuples form exactly
\(\operatorname{Gen}_{m+1}(\mathcal T,E/B)\).

Apply Theorem~\ref{thm:generating-sections} with the compact parameter
manifold \(H\) of dimension \(2k\), the rank-two bundle \(\mathcal T\),
and the Oka base \(B\) of dimension \(q\). Its bound is
\[
m+1\ge2+2k+\max\{1,q\},
\]
which is \eqref{eq:projective-embedding-bound}. The tuple space is
therefore strongly dominable. The length-two test is open, and the
family description \eqref{eq:projective-map-quotient} identifies its
common scalar quotient with
\(\operatorname{Emb}_{\alpha}(X,\mathbb P^m)\), a smooth open subset of
\(\mathbb P(E^{\oplus(m+1)})\). The scalar projection is a submersion,
so the observation following Theorem~\ref{thm:generating-sections}
gives strong dominability of the embedding space. The criterion
\cite[Theorem 1.4]{DW} gives its Oka property.
\end{proof}

The fixed-line-bundle version also applies to compact complex sources.
Let \(X\) be a compact connected complex manifold of dimension \(k\ge1\),
and let \(L\) be globally generated. For \(m\ge k+1\), the space
\[
\operatorname{Hol}_{L}(X,\mathbb P^m)
=\{f:X\to\mathbb P^m:f^*\mathcal O(1)\cong L\}
\]
is nonempty, smooth, and Oka. Indeed, it is the generating open set in
\(\mathbb P(H^0(X,L)^{\oplus(m+1)})\); apply
Theorem~\ref{thm:generating-sections} with base a point and rank one,
then pass to the scalar quotient.
If \(L\) is very ample, the fixed-\(L\) embedding space is likewise
nonempty, smooth, and Oka for \(m\ge2k+2\). Very ampleness makes \(X\)
projective, so its Hilbert square is compact of dimension \(2k\).
The same theorem applied to the rank-two tautological bundle on this
Hilbert square gives the length-two separation condition, and the scalar
quotient gives the embedding space.

\begin{corollary}\label{cor:curve-projective-embeddings}
For a smooth connected projective curve \(C\) of genus \(g\), the space \(\operatorname{Emb}_d(C,\mathbb P^m)\) is Oka whenever
\[
d\ge2g+1,
\qquad
m\ge
\begin{cases}
4,&g=0,\\
g+3,&g\ge1.
\end{cases}
\]
The quotient \(\operatorname{Emb}_d(C,\mathbb P^m)/\operatorname{Aut}(C)\),
which parametrizes embedded images isomorphic to \(C\), is also a
smooth Oka manifold.
\end{corollary}

\begin{proof}
A degree-\(d\) line bundle is very ample for \(d\ge2g+1\), and the
numerical bound is \eqref{eq:projective-embedding-bound} with \(k=1\)
and \(q=g\).

The precomposition action of \(\operatorname{Aut}(C)\) is free and
proper. Freeness follows from injectivity of the embeddings. For
properness, suppose \(f_\nu\to f\) and
\(h_\nu=f_\nu\circ a_\nu\to h\), with both limits embeddings. Fix
Hermitian metrics on \(C\) and \(\mathbb P^m\). Holomorphic convergence
on the compact curve gives uniform convergence of derivatives. Since
\(f\) is immersive, the differentials \(df_\nu\) have a uniform positive
lower bound, while \(dh_\nu\) are uniformly bounded. The identity
\(dh_\nu=df_\nu\circ da_\nu\) therefore bounds \(da_\nu\) uniformly.
Ascoli's theorem gives a subsequence converging to a holomorphic map
\(a:C\to C\), with \(f\circ a=h\). This map is injective and
nonconstant, hence a biholomorphism of compact curves.

The free proper action gives a holomorphic principal
\(\operatorname{Aut}(C)\)-bundle over the smooth quotient. The group
is Oka: it is \(\mathrm{PGL}_2(\mathbb C)\) for \(g=0\), a finite
extension of the translation group of \(C\) for \(g=1\), and finite
for \(g\ge2\). The bundle theorem
\cite[Corollaries 2.51--2.52]{For} proves that the quotient is Oka.
\end{proof}

\subsection{Maps to Grassmannians with stable quotient bundles}\label{sec:stable-quotient-maps}

For Grassmannian targets, a map is represented by a generated vector
bundle. Let \(C\) be a smooth projective
curve of genus \(g\ge2\), and let the rank \(r\ge2\) and degree
\(d\) satisfy
\begin{equation}
d>r(2g-1).
\label{eq:M4.1}
\end{equation}
Assume that
\begin{equation}
(r,g)\ne(2,2)\qquad\text{or}\qquad d\text{ is odd}.
\label{eq:M4.2}
\end{equation}
Use the quotient convention for \(\operatorname{Gr}(r,N)\), and denote its rank-\(r\) universal quotient by \(Q\). Consider the mapping space
\[
\mathcal M^s_{d,N}
=
\{f:C\to\operatorname{Gr}(r,N):
f^*Q\text{ is stable of rank }r\text{ and degree }d\}.
\]
Its fixed-determinant locus for a degree-\(d\) line bundle \(\xi\) is
\[
\mathcal M^s_{\xi,N}
=
\{f\in\mathcal M^s_{d,N}:
\det(f^*Q)\cong\xi\}.
\]
Here the determinant is specified only by its isomorphism class
\(\xi\).

\begin{theorem}\label{thm:M4}

Under \eqref{eq:M4.1}--\eqref{eq:M4.2}, the following spaces are smooth and Oka:
\begin{enumerate}[label=\textup{(\roman*)}]
\item \(\mathcal M^s_{\xi,N}\) for \(N\ge r+(r^2-1)(g-1)+1\), with
\[
\dim\mathcal M^s_{\xi,N}=Nd+r(N-r)(1-g)-g;
\]
\item \(\mathcal M^s_{d,N}\) for \(N\ge r+r^2(g-1)+2\), with
\[
\dim\mathcal M^s_{d,N}=Nd+r(N-r)(1-g).
\]
\end{enumerate}

\end{theorem}

\begin{proof}
We first treat the fixed determinant in \textup{(i)}.
Theorem~\ref{thm:A1} gives the Oka property of the stable-bundle moduli
spaces
\[
U_C^s(r,d),\qquad SU_C^s(r,\xi).
\]
We will lift a dominating entire map to a genuine family
of bundles and apply Theorem~\ref{thm:generating-sections} to its
quotient presentations.

\smallskip
\noindent\emph{Quotient presentations.}
For a stable bundle \(E\) of rank \(r\) and degree \(d\), Serre duality gives \[
H^1(C,E(-x))^*
\cong\operatorname{Hom}(E(-x),K_C).
\]
A nonzero map would give a rank-one torsion-free quotient of \(E(-x)\) whose degree is at most \(2g-2\). Stability forces the slope of a proper positive-rank quotient to be strictly greater than \[
\mu(E(-x))=d/r-1>2g-2,
\]
a contradiction. Thus \(H^1(E(-x))=0\) for every \(x\). The exact sequence for evaluation at \(x\) then proves global generation. The same slope argument gives \(H^1(E)=0\).

Consequently
\begin{equation}
R=h^0(C,E)=d+r(1-g)
\label{eq:M4.4}
\end{equation}
is constant on the stable base. Every genuine family of these
bundles therefore has a locally free direct image that commutes with
base change, and its evaluation morphism is surjective.

A map \(f:C\to\operatorname{Gr}(r,N)\) is equivalent to a quotient \[
\mathcal O_C^N\twoheadrightarrow E
\]
of rank \(r\), by the functorial description of the Grassmannian \cite[Tag 089R]{A2026}.

For fixed stable \(E\), quotient presentations are generating \(N\)-tuples in \(H^0(C,E)\) modulo \(\operatorname{Aut}(E)=\mathbb C^*\). Thus the fibre of the forgetful map to the stable-bundle moduli space is \begin{equation}
\operatorname{Gen}_N(E)/\mathbb C^*.
\label{eq:M4.5}
\end{equation}

\smallskip
\noindent\emph{A genuine family over the dominating source.}
Put \(B=SU_C^s(r,\xi)\). Restricting the universal projective bundle
at a fixed \(y\in C\) gives a holomorphic
\(\mathbb P^{r-1}\)-bundle
\[
p_y:P_y\longrightarrow B.
\]
Its total space carries a genuine universal vector bundle \begin{equation}
\mathcal U_y\longrightarrow C\times P_y
\label{eq:M4.7}
\end{equation}
whose fibre at a point of \(P_y\) represents the stable bundle specified by \(p_y\).

Bhosle and Biswas \cite[§2.1, equations (2.4) and (2.5)]{BB2012}
construct \eqref{eq:M4.7} by realizing \(P_y\) as a parabolic moduli
space with one multiplicity-one parabolic weight. Forgetting the flag
in the universal parabolic bundle gives the required vector bundle.
Their smooth-curve specialization applies in the numerical range
of \eqref{eq:M4.2}.

Analytic local triviality of \(P_y\) gives local holomorphic sections.
Restricting \eqref{eq:M4.7} along such a section gives a local universal
vector bundle. On an overlap, two such families differ by tensoring
with a line bundle from the base. Projectivization absorbs this factor,
so their projective quotient presentations glue.

On an analytic open set of that base carrying a universal family \(\mathcal E\), \eqref{eq:M4.5} gives an open subset of \[
\mathbb P((p_*\mathcal E)^{\oplus N}).
\]
In particular the total mapping space is locally smooth, and the forgetful map is a holomorphic submersion. Its vertical tangent space at a tuple \(s\) is \begin{equation}
H^0(C,E)^{\oplus N}/\mathbb C s.
\label{eq:M4.6}
\end{equation}

For the horizontal directions, put \(b=(r^2-1)(g-1)\) and choose a
dominating entire map
\[
\varphi:\mathbb C^b\to B,\qquad
\varphi(0)=b_0,\qquad d\varphi_0\text{ an isomorphism}.
\]
The projective-space bundle
\(\varphi^*P_y\) has structure group \(\mathrm{PGL}_r(\mathbb C)\).
Its principal bundle of projective frames is holomorphically trivial
by the Oka--Grauert principle of Grauert \cite{Grauert58}, in the
form of \cite[Theorem 5.3.1]{F17}, because \(\mathbb C^b\) is
contractible Stein. Thus \(\varphi^*P_y\) is holomorphically trivial.
An entire section through any prescribed point over zero gives a lift
\[
\widetilde\varphi:\mathbb C^b\to P_y,\qquad
p_y\widetilde\varphi=\varphi.
\]
Pulling back \eqref{eq:M4.7} yields a genuine bundle \[
\mathcal F\longrightarrow C\times\mathbb C^b
\]
whose classifying map is exactly \(\varphi\).

At the prescribed point of \(\mathcal M^s_{\xi,N}\), identify its
quotient bundle with the fibre of \(\mathcal F\) at zero. The
vanishings \(H^1(E)=H^1(E(-x))=0\) and the constant rank
\eqref{eq:M4.4} give a locally free direct image for \(\mathcal F\),
compatible with base change and with surjective evaluation. Hence
Theorem~\ref{thm:generating-sections}, with \(k=1\), gives, whenever
\[
N\ge r+1+\max\{1,b\}=r+b+1,
\]
an entire map dominating the generating-section space at the chosen
tuple. Since \(b=(r^2-1)(g-1)\ge3\), this bound also gives the
redundant tuples required in that construction.

Compose with the map to \(\mathcal M^s_{\xi,N}\). The isomorphism
\(d\varphi_0\) supplies the base directions, and the generating-section
construction supplies all vertical directions before passage to the
scalar quotient \eqref{eq:M4.6}. The composite differential is
surjective. The same linear-slice observation proves strong
dominability, and \cite[Theorem 1.4]{DW} gives assertion \textup{(i)}.

\smallskip
\noindent\emph{Variable determinant.}
Let \(J=\operatorname{Pic}^0(C)\). The map
\begin{equation}
SU_C^s(r,\xi)\times J\longrightarrow U_C^s(r,d),
\qquad(E,L)\longmapsto E\otimes L.
\label{eq:M4.8}
\end{equation}
is finite étale. Indeed, identify \(\operatorname{Pic}^d(C)\) with
\(J\) by tensoring with \(\xi^{-1}\). Then \eqref{eq:M4.8} is the
base change of \[
[r]:J\to J,\qquad L\mapsto L^{\otimes r},
\]
along the determinant map. The deck group \(J[r]\) acts by \[
T\cdot(E,L)=(E\otimes T,L\otimes T^{-1}).
\]
Translation on the second factor is free for \(T\ne\mathcal O_C\), so \eqref{eq:M4.8} is unramified even when \(E\) admits a self-twist.

The product \(SU_C^s(r,\xi)\times J\) is Oka. On a Euclidean
dominating source for this product, pull back the family obtained
from \(P_y\) on the first factor and tensor it with the Poincaré
line bundle from the second factor. The resulting genuine family
has a classifying map to \(U_C^s(r,d)\) with full differential.
The generating-section theorem and the scalar-quotient calculation
apply with
\[
b=r^2(g-1)+1.
\]
The resulting numerical condition is \(N\ge r+r^2(g-1)+2\), which proves assertion \textup{(ii)}.

The variable-determinant dimension is \[
\begin{aligned}
\dim\mathcal M^s_{d,N}
&=r^2(g-1)+1+N(d+r(1-g))-1\\
&=Nd+r(N-r)(1-g).
\end{aligned}
\]
Fixing the determinant subtracts \(g\). The tangent-space calculation
gives the same dimension: the exact sequence
\[
0\to\operatorname{End}(E)\to E^{\oplus N}
\to f^*T\operatorname{Gr}(r,N)\to0
\]
and \(H^1(E)=0\) give
\(H^1(f^*T\operatorname{Gr}(r,N))=0\).
\end{proof}

For example, on any fixed genus-two curve, \(\mathcal M^s_{7,8}\), the space of maps with stable rank-two quotient bundle of degree seven, is Oka of dimension \(44\). With determinant fixed, \(\mathcal M^s_{\xi,6}\) is Oka of dimension \(32\).

\section{Hilbert schemes and nested Quot schemes}\label{sec:hilbert-quot}

Configurations of distinct points give an Oka open part of a
complex surface's Hilbert--Douady space. For nested Quot schemes on
curves, the Oka property is equivalent to that of the symmetric powers
associated with the length increments. A dense affine-torsor locus
proves sufficiency, and a section of the support map proves necessity.

\subsection{Hilbert schemes of points on surfaces}\label{sec:surface-hilbert}

For a complex manifold \(S\), write
\[
F_r(S)=\{(p_1,\ldots,p_r)\in S^r:p_i\ne p_j\text{ for }i\ne j\}.
\]
For a complex surface \(S\), let \(S^{[n]}\) denote its
Hilbert--Douady space of length-\(n\) analytic subspaces; for an
algebraic surface this is the usual Hilbert scheme. The free
symmetric-group quotient \(F_n(S)/\mathfrak S_n\) is its reduced-subspace
locus.

\begin{proposition}\label{thm:surface-hilbert}
Let \(S\) be a connected complex surface. If \(F_n(S)\) is Oka, then
\(S^{[n]}\) is Oka. When \(S\) is projective, \(S^{[n]}\) is also
holomorphically elliptic.
\end{proposition}

\begin{proof}
The analytic form of Fogarty's theorem makes \(S^{[n]}\) a connected
smooth complex manifold of dimension \(2n\), with dense reduced-subspace
locus \(F_n(S)/\mathfrak S_n\); see
\cite[Theorem 2.3.1]{deCataldoMigliorini2000}, and
\cite[Proposition 2.3 and Theorem 2.4]{F1968} in the algebraic case.
The symmetric-group action is free, so covering permanence
\cite[Corollaries 2.51--2.52]{For} makes this quotient Oka.
Its complement is a proper closed analytic subset of \(S^{[n]}\).
Analytic extension \cite[Theorem 1.9]{DW} proves the Oka assertion.
If \(S\) is projective, so is \(S^{[n]}\), and
\cite[Theorem 1.1]{FL2026} gives holomorphic ellipticity.
\end{proof}

\begin{corollary}\label{cor:product-curves-hilbert}
Let \(X,Y\) be positive-dimensional connected Oka manifolds. Then
\(F_n(X\times Y)\) is Oka for every \(n\ge1\). If \(X,Y\) are curves,
then \((X\times Y)^{[n]}\) is Oka as well, and is holomorphically
elliptic when both curves are projective.
\end{corollary}

\begin{proof}
For \(n=1\), the configuration space is \(X\times Y\). For \(n\ge2\),
let \(A\subset X^n\) and \(B\subset Y^n\) be the respective unions of
pairwise diagonals. Both are proper closed analytic subsets, and both
ambient powers are Oka. The product-complement theorem
\cite[Theorem 7.8]{DW} makes
\[
U=(X^n\times Y^n)\setminus(A\times B)
\]
Oka. Under \((X\times Y)^n\cong X^n\times Y^n\), a point of \(U\)
has either all first coordinates distinct or all second coordinates
distinct. Thus \(U\subset F_n(X\times Y)\). The configuration space
is connected, since the pairwise diagonals have complex codimension
\(\dim X+\dim Y\ge2\) in the connected manifold \((X\times Y)^n\).
The complement of \(U\) in the configuration space is the proper closed analytic subset
\(F_n(X\times Y)\cap(A\times B)\). Analytic extension makes the
configuration space Oka. When \(X,Y\) are curves,
Proposition~\ref{thm:surface-hilbert} applies to the product surface.

\end{proof}

In the surface case, either factor may be \(\mathbb P^1\), \(\mathbb C\),
\(\mathbb C^*\), or an elliptic curve. Thus the corollary includes
\(\mathbb C^2\), \(\mathbb C\times\mathbb C^*\),
\((\mathbb C^*)^2\), products with \(\mathbb P^1\), and products of
elliptic curves.

In dimension two, the first construction below gives Kummer K3
surfaces. Forstneri\v{c}--L\'arusson proved their strong dominability
and identified the Oka complement of the sixteen exceptional curves
\cite[Corollary 3 and Lemma 7]{FL2014}; Alarc\'on--Forstneri\v{c}
proved the Oka-1 property for compact complex surfaces bimeromorphic to
Kummer surfaces \cite[Proposition 8.4]{AF2025}.

\begin{proposition}\label{prop:kummer-torus-hilbert}
Let \(T\) be a complex torus of dimension \(m\ge2\). The blowup of
\(T/\{\pm1\}\) at its singular points is a smooth Oka manifold.
If \(m=2\), the Hilbert square \(T^{[2]}\) is Oka. When \(T\) is an
abelian variety, these manifolds are projective and holomorphically
elliptic.
\end{proposition}

\begin{proof}
The complement \(T\setminus T[2]\) is Oka by
\cite[Corollary 2.25]{For}, which treats complex tori of dimension at
least two with finitely many points removed. Its free unramified quotient
is the smooth locus of \(T/\{\pm1\}\), and is Oka. Blowing up the vertex
of the local model \(\mathbb C^m/\{\pm1\}\) gives the total space of
\(\mathcal O_{\mathbb P^{m-1}}(-2)\). Thus the global blowup is smooth,
with the Oka smooth locus as the complement of its closed analytic
exceptional locus. Analytic extension \cite[Theorem 1.9]{DW} proves the
first assertion.

When the torus is two-dimensional, subtraction gives the biholomorphism
\[
F_2(T)\longrightarrow T\times(T\setminus\{0\}),
\qquad (x,y)\longmapsto(x,y-x).
\]
The product is Oka, so Proposition~\ref{thm:surface-hilbert} gives
the assertion for \(T^{[2]}\).
For an abelian variety \(T\), the finite quotient and its blowup are
projective, as is \(T^{[2]}\) when \(m=2\). Their holomorphic ellipticity
follows from \cite[Theorem 1.1]{FL2026}.
\end{proof}

\subsection{Nested Quot schemes on curves}\label{sec:nested-quot}

For a vector bundle on a curve, the support of the successive finite
quotients defines a map to a product of symmetric powers. This product
determines the Oka property of the nested Quot scheme.

\begin{theorem}\label{thm:N2}
Let \(C\) be a smooth connected projective curve of genus \(g\), and
let \(V\) be a vector bundle of rank \(r\ge1\). For integers
\(0=d_0\le d_1\le\cdots\le d_s\), put \(n_j=d_j-d_{j-1}\), and
let \(\mathcal Q=\operatorname{Quot}_{\mathbf d}(V)\) parametrize
\[
K_s\subset\cdots\subset K_1\subset V,\qquad
\operatorname{length}(V/K_j)=d_j.
\]
Then \(\mathcal Q\) is smooth, irreducible, and projective of dimension
\(rd_s\), and
\[
\mathcal Q\text{ is Oka}
\quad\Longleftrightarrow\quad
\prod_j\operatorname{Sym}^{n_j}C\text{ is Oka}
\quad\Longleftrightarrow\quad
\operatorname{Sym}^{n_j}C\text{ is Oka for every }j.
\]
Whenever these conditions hold, \(\mathcal Q\) is holomorphically
elliptic. If \(g\ge2\), they require every positive increment to
satisfy \(n_j\ge g\). They hold if every positive increment satisfies
\(n_j\ge2g-2\). All length vectors are allowed when \(g\le1\);
when \(g=2\), they hold exactly when no increment equals one.
\end{theorem}

\begin{proof}
By \cite[Theorem 1.4 and Proposition 2.1]{MonavariRicolfi2022},
\(\mathcal Q\) is smooth and connected of dimension \(rd_s\), hence
irreducible. It is projective by the Quot construction \cite{Gro61}.
Its universal kernels are vector bundles: they are flat over the
parameter scheme, their fibres are locally free on the smooth curve,
and the fibrewise local-freeness criterion
\cite[Tag 080Q]{A2026} gives local freeness on the product.

\smallskip
\noindent\emph{A dense affine-torsor locus.}
Choose a line subbundle \(L\subset V\) with vector bundle quotient
\(W\). Such a subbundle is obtained by saturating the rank-one subsheaf
defined by a nonzero section of a sufficiently positive twist of \(V\).
Let \(U\subset\mathcal Q\) be the locus where
\[
L\longrightarrow V/K_s
\]
is surjective. It is open: the cokernel has support proper over
\(\mathcal Q\), and the complement of its image is \(U\). For a flag
in \(U\), the maps \(L\to V/K_j\) are all surjective. Hence
\[
L'_j=K_j\cap L=L(-D_j),\qquad
D_1\le\cdots\le D_s,\qquad \deg D_j=d_j,
\]
and every \(K_j\) maps onto \(W\). The divisors are parametrized by
\[
S=\prod_{j=1}^s\operatorname{Sym}^{n_j}C.
\]

Put \(K_0=V\) and \(L'_0=L\). After \(K_{j-1}\) is chosen, the
possible \(K_j\) correspond to splittings of
\begin{equation}\label{eq:nested-top-splitting}
0\longrightarrow L'_{j-1}/L'_j
\longrightarrow K_{j-1}/L'_j
\longrightarrow W\longrightarrow0.
\end{equation}
A splitting lifts to \(K_j\subset K_{j-1}\); its intersection with
\(L\) is \(L'_j\), and its image is \(W\). Conversely,
\(K_j/L'_j\to W\) is an isomorphism and recovers the splitting.
The finite support of \(L'_{j-1}/L'_j\) gives
\[
\operatorname{Ext}^1(W,L'_{j-1}/L'_j)=0.
\]
The splittings form a torsor under
\(\operatorname{Hom}(W,L'_{j-1}/L'_j)\), of dimension \((r-1)n_j\).

These torsors are Zariski locally trivial in families. Indeed, over
a parameter scheme \(B\), the finite sheaf
\(\mathcal T=L'_{j-1}/L'_j\) is flat of relative length \(n_j\).
For \(p:C\times B\to B\), the sheaf
\(\mathcal H=W^\vee\otimes\mathcal T\) is flat of relative length
\((r-1)n_j\), and
\[
p_*\mathcal H\text{ is locally free of rank }(r-1)n_j,
\qquad R^qp_*\mathcal H=0\quad(q>0).
\]
Finite pushforward commutes with base change. On an affine open of
\(B\), Leray gives
\(H^1(C\times B,\mathcal H)=H^1(B,p_*\mathcal H)=0\), so
\eqref{eq:nested-top-splitting} splits there. Differences of splittings
are sections of \(p_*\mathcal H\).

This construction identifies \(U\), in families, with an affine-torsor
tower over \(S\). In the reverse direction, each \(V/K_j\) on \(U\)
is a flat finite quotient of \(L\), so its kernel is \(L(-D_j)\) for
a relative effective Cartier divisor \(D_j\). The quotients and
intersections commute with base change, and the inverse constructions
above recover the flag. Each torsor has nonempty fibres; consequently
\(U\) is nonempty and irreducible, with dimension \(rd_s\).
Since \(\mathcal Q\) is irreducible, \(U\) is dense.

\smallskip
\noindent\emph{The support map and its section.}
Put \(\mathcal K_0=V\) in the universal flag over
\(C\times\mathcal Q\). The exact sequence
\[
0\longrightarrow\mathcal K_{j-1}/\mathcal K_j
\longrightarrow V/\mathcal K_j
\longrightarrow V/\mathcal K_{j-1}\longrightarrow0
\]
has two cumulative quotients flat over \(\mathcal Q\).
Hence \(\mathcal T_j=\mathcal K_{j-1}/\mathcal K_j\) is flat of relative
length \(n_j\), with finite support over \(\mathcal Q\).
The determinant of
\(\mathcal K_j\hookrightarrow\mathcal K_{j-1}\) cuts out a relative
effective Cartier divisor \(\Delta_j\) of degree \(n_j\); its ideal is
\(\operatorname{Fitt}_0(\mathcal T_j)\). Formation of these divisors
commutes with base change, giving the increment-support morphism
\[
\sigma:\mathcal Q\longrightarrow
S=\prod_j\operatorname{Sym}^{n_j}C.
\]
For the usual cumulative support morphism, see
\cite[\S1.2, equation (1.2)]{MonavariRicolfi2022}.

Choose a line bundle quotient \(\pi:V\twoheadrightarrow M\), obtained
by saturating a line subsheaf in \(V^\vee\) and dualizing.
On \(C\times S\), let \(D_j\) be the universal increment divisors and
put \(E_j=D_1+\cdots+D_j\). The kernels
\[
K_j=\ker\bigl(V\longrightarrow M/M(-E_j)\bigr)
\]
form a nested family: their quotients are finite and flat of relative
length \(d_j\). They define a morphism \(i:S\to\mathcal Q\).
A local splitting of \(\pi\) identifies
\(K_j\) with \(\ker\pi\oplus M(-E_j)\), so the successive Fitting
divisors are exactly \(D_j\). Thus \(\sigma\circ i=\mathrm{id}_S\).

\smallskip
\noindent\emph{The Oka property.}
If \(S\) is Oka, its affine-torsor tower \(U\) is Oka by
\cite[Corollaries 2.51--2.52]{For}. The complement
\(\mathcal Q\setminus U\) is proper closed algebraic, so analytic
extension \cite[Theorem 1.9]{DW} makes \(\mathcal Q\) Oka.
Conversely, \(S\) is a holomorphic retract of \(\mathcal Q\), so
inherits the Oka property. Each factor of \(S\) is a retract of the
product, proving the stated equivalences. The projectivity of
\(\mathcal Q\) then gives holomorphic ellipticity
\cite[Theorem 1.1]{FL2026}.

The zeroth symmetric power is a point. For \(n\ge2g-1\), the
Abel--Jacobi morphism
\[
\operatorname{Sym}^{n}C\longrightarrow\operatorname{Pic}^{n}(C)
\]
is a projective bundle: every degree-\(n\) line bundle has vanishing
first cohomology and \(n-g+1\) independent sections, whose spaces form
a vector bundle after normalizing the Poincar\'e bundle. Thus these
symmetric powers are Oka.

For \(g\ge2\) and \(n=2g-2\), restrict the Abel--Jacobi morphism to
\(\operatorname{Pic}^{2g-2}(C)\setminus\{K_C\}\). Serre duality gives
\(H^1(C,L)=0\) there, and Riemann--Roch gives \(h^0(C,L)=g-1\).
Its inverse image is therefore a \(\mathbb P^{g-2}\)-bundle.
The base is a complex torus of dimension \(g\ge2\) with one point
removed, so is Oka by \cite[Corollary 2.25]{For}. The bundle theorem
makes this dense open subset Oka, and analytic extension makes
\(\operatorname{Sym}^{2g-2}C\) Oka. This proves the numerical range.
For the necessary bound, suppose \(0<n<g\) and
\(\operatorname{Sym}^nC\) is Oka. Interpolating a countable dense
sequence at the integers gives an entire curve
\(h:\mathbb C\to\operatorname{Sym}^nC\) with dense image.
The Zariski closure of its Abel--Jacobi image is
\(W_n\subset\operatorname{Pic}^n(C)\). By the Bloch--Ochiai theorem
\cite[Lemma 5.4]{NWY2002}, \(W_n\) is a translate of an abelian
subvariety of the Jacobian. Identify \(\operatorname{Pic}^n(C)\)
with \(\operatorname{Jac}(C)\) using \(\mathcal O_C(np_0)\).
The divisor \(np_0\) gives \(0\in W_n\), so \(W_n\) is an abelian
subvariety. It also contains the Abel curve
\(p\mapsto\mathcal O_C(p-p_0)\), which generates the Jacobian.
Thus \(W_n=\operatorname{Jac}(C)\), contradicting
\(\dim W_n\le n<g\). The equivalence already proved now gives the
necessary bound for every positive increment. Together with the
sufficient range, it gives the exact criterion in genus two.
\end{proof}

\section{Character varieties and parabolic connections}\label{sec:character-connections}

For the rank-two character varieties below, simple trace functions have
complete Hamiltonian twist flows. A finite generating set of traces
therefore gives a dominating spray on the irreducible locus. We use this
locus to treat resolutions, representation varieties, and moduli spaces of
parabolic connections under the Riemann--Hilbert correspondence.

\subsection{Character varieties and Hamiltonian flows}\label{sec:character-representations}\label{sec:M8}

For a closed oriented surface \(\Sigma_g\) of genus \(g\ge2\), its \(\mathrm{SL}_2(\mathbb C)\)-character variety is
\[
\mathcal X_g=
\operatorname{Hom}(\pi_1\Sigma_g,\mathrm{SL}_2(\mathbb C))
\sslash\mathrm{SL}_2(\mathbb C).
\]
Its regular locus is the irreducible-character locus
\cite[Lemma 2.4]{BS2023}. Goldman's construction equips it with a
holomorphic symplectic form, and its dimension is \(6g-6\)
\cite{Gold84}.

For \(G=\mathrm{SL}_2(\mathbb C)\) or \(\mathrm{GL}_2(\mathbb C)\), write
\[
\mathcal X_g(G)=
\operatorname{Hom}(\pi_1\Sigma_g,G)\mathbin{/\mkern-6mu/}G.
\]

For a simple closed curve \(\gamma\), the trace function is
\[
f_\gamma([\rho])=\operatorname{tr}\rho(\gamma).
\]
For \(\mathrm{SL}_2\), the character ring of a closed surface and each
fixed-trace relative character ring have a finite generating family of
simple traces \cite[Theorem 2.1 and its proof in Sections 2.1--2.2]{GX2011}.
The construction first generates the ambient complex character ring;
restriction gives the generators on every fixed-trace fibre. The identity
\begin{equation}
\operatorname{tr}(AB)+\operatorname{tr}(AB^{-1})
=\operatorname{tr}(A)\operatorname{tr}(B),
\qquad A,B\in\mathrm{SL}_2,
\label{eq:M8.2}
\end{equation}
resolves crossings in trace polynomials; see
\cite[Fact 4.1 and Theorem 7.1]{PS1998} for its skein interpretation.

We shall use the following finite-flow construction. Let
\(V_1,\ldots,V_N\) be complete holomorphic vector fields on a complex
manifold \(Y\), with flows \(\Phi_i^t\), whose values span \(T_yY\) at
every \(y\in Y\). Then
\[
\mathfrak s(y,t_1,\ldots,t_N)
=\Phi_N^{t_N}\circ\cdots\circ\Phi_1^{t_1}(y)
\]
is a dominating spray on \(Y\times\mathbb C^N\), since
\[
\mathfrak s(y,0)=y,\qquad
d\mathfrak s_{(y,0)}(0,\tau_1,\ldots,\tau_N)
=\sum_i\tau_iV_i(y).
\]
The same differential makes each orbit of the subgroup generated by
these flows open. Since the orbits partition \(Y\), this subgroup acts
transitively on every connected component; compare
\cite[Lemmas 3.1--3.2]{GX2011}.

This construction also applies to a holomorphic principal
\(\mathrm{PGL}_2\)-bundle \(q:R\to Y\) when the spanning fields on \(Y\)
have complete holomorphic lifts to \(R\). Three complete fundamental
fields from a basis of \(\mathfrak{sl}_2\) span \(\ker dq\), and the
lifts supply the remaining tangent directions. The resulting finite
family of complete fields gives a dominating spray on \(R\).

\begin{proposition}\label{prop:goldman-ellipticity}
For \(G=\mathrm{SL}_2(\mathbb C)\) or \(\mathrm{GL}_2(\mathbb C)\),
the irreducible-character locus \(\mathcal X_g(G)^{\mathrm{irr}}\) is
holomorphically elliptic and hence Oka.
\end{proposition}

\begin{proof}
We first treat \(G=\mathrm{SL}_2(\mathbb C)\). Under the trace pairing,
the variation of \(f_\gamma\) is represented by
\[
F(\rho(\gamma))
=\rho(\gamma)-\tfrac12\operatorname{tr}(\rho(\gamma))I.
\]
The twist along the curve uses the one-parameter subgroup
\begin{equation}
\zeta_t=
\exp\!\left(tF(\rho(\gamma))\right),
\qquad t\in\mathbb C.
\label{eq:M8.1}
\end{equation}
If the curve is nonseparating, cut the surface and multiply the
appropriate HNN stable letter by \(\zeta_t\). If it is separating,
conjugate the representation on one side by \(\zeta_t\). Commutation
of \(\zeta_t\) with \(\rho(\gamma)\) preserves the gluing relation.

The matrix exponential depends holomorphically on the representation
and on every \(t\in\mathbb C\). The twisted representation \(\rho_t\)
satisfies
\[
\rho_t(\gamma)=\rho(\gamma),\qquad \zeta_{-t}=\zeta_t^{-1},
\]
so time \(-t\) gives the inverse. To check preservation of
irreducibility, suppose that a line \(\ell\subset\mathbb C^2\) is
invariant under \(\rho_t\). Then
\(\rho(\gamma)\ell=\ell\), and \eqref{eq:M8.1} gives
\(\zeta_t\ell=\ell\). In the separating case, undoing conjugation by
\(\zeta_t\) preserves \(\ell\) on both sides of the cut. In the
nonseparating case, the cut-surface group and the original HNN stable
letter preserve \(\ell\), since its twisted stable letter differs by
\(\zeta_t\). Thus an invariant line for the twisted representation
would also be invariant for the original one, and \(\rho_t\) is
irreducible whenever \(\rho\) is.

The twists are equivariant under conjugation and descend to holomorphic
\(\mathbb C\)-flows on \(\mathcal X_g^{\mathrm{irr}}\).
These are Goldman's Hamiltonian twist flows \cite{Gold86}, with the
complex \(\mathrm{SL}_2\) formula given in
\cite[§2.3, equation (2.5), and Lemma 2.3.1]{G2004}.

Choose a finite family \(f_1,\ldots,f_N\) of simple-trace generators as above.

The morphism \[
(f_1,\ldots,f_N):\mathcal X_g\hookrightarrow\mathbb C^N
\]
is a closed embedding, so \(df_1,\ldots,df_N\) span the cotangent
space at every smooth point. The symplectic form turns them into
spanning Hamiltonian vector fields, each complete by
\eqref{eq:M8.1}.

The finite-flow construction above gives a dominating spray on
\(\mathcal X_g^{\mathrm{irr}}\times\mathbb C^N\), proving holomorphic
ellipticity and hence the Oka property.

The irreducible locus of \(\mathcal X_g(\mathrm{GL}_2)\) is also
holomorphically elliptic. To construct a spray, put
\[
T=\operatorname{Hom}(\pi_1\Sigma_g,\mathbb C^*)
\cong(\mathbb C^*)^{2g}.
\]
Twisting by rank-one characters gives the finite unramified covering
\[
p:\mathcal X_g(\mathrm{SL}_2)^{\mathrm{irr}}\times T
\longrightarrow\mathcal X_g(\mathrm{GL}_2)^{\mathrm{irr}},
\qquad ([\rho],\chi)\longmapsto[\chi\rho].
\]
Its deck group is
\(H=\operatorname{Hom}(\pi_1\Sigma_g,\{\pm1\})\), acting by
\((\rho,\chi)\mapsto(\varepsilon\rho,\varepsilon\chi)\); the action is
free on \(T\) \cite[Lemma 2.3 and diagram (6)]{BS2023}.

Let \(W_j\) be the Hamiltonian field of the simple trace
\(f_{\gamma_j}\) used above, and write \(\Phi_j^s\) for its complete
flow. Central-character twisting leaves the adjoint local system and
Goldman's symplectic form unchanged, while
\[
f_{\gamma_j}(\varepsilon[\rho])
=\varepsilon(\gamma_j)f_{\gamma_j}([\rho]).
\]
Consequently
\begin{equation}\label{eq:central-twist-equivariance}
\varepsilon\Phi_j^s(x)
=\Phi_j^{\varepsilon(\gamma_j)s}(\varepsilon x).
\end{equation}
The nowhere-zero character monomial
\(a_j(\chi)=\chi(\gamma_j)\) transforms by the same sign. Thus the
vector field
\[
V_j(x,\chi)=\bigl(a_j(\chi)W_j(x),0\bigr)
\]
on the covering space is \(H\)-invariant. It is complete: its flow is
\((x,\chi)\mapsto(\Phi_j^{a_j(\chi)s}(x),\chi)\). The logarithmic
coordinate fields \(\chi_i\partial_{\chi_i}\) on \(T\) are likewise
complete and \(H\)-invariant. These fields descend through \(p\).
Since every \(a_j\) is nonzero, their values span the tangent space of
the covering space at every point. The differential of \(p\) is an
isomorphism, so the descended fields also span everywhere. Their finite
flow composition gives a dominating spray on a trivial vector bundle.
\end{proof}

\begin{corollary}\label{cor:character-resolutions}
For \(G=\mathrm{SL}_2(\mathbb C)\) or \(\mathrm{GL}_2(\mathbb C)\) and
\(g\ge2\), suppose
\[
\pi:\widetilde{\mathcal X}_g(G)\longrightarrow\mathcal X_g(G)
\]
is a resolution of singularities that is an isomorphism over
\(\mathcal X_g(G)^{\mathrm{irr}}\). Then
\(\widetilde{\mathcal X}_g(G)\) is Oka. In particular, for \(g=2\),
the symplectic resolutions obtained by blowing up the reduced singular
loci of the \(\mathrm{SL}_2\) and \(\mathrm{GL}_2\) character varieties
are Oka.
\end{corollary}

\begin{proof}
The irreducible locus is the smooth locus of \(\mathcal X_g(G)\)
\cite[Lemma 2.4]{BS2023}. It is Oka by
Proposition~\ref{prop:goldman-ellipticity}. Since \(\pi\) is an
isomorphism there, its inverse image
is Oka. The complement is a proper closed complex analytic subset of the
smooth resolution, so \cite[Theorem 1.9]{DW} applies.

The genus-two blowups are smooth symplectic resolutions by
\cite[Corollaries 2.6--2.7]{BS2023}; their centres are contained in the
singular loci, so they are isomorphisms over the irreducible loci.
\end{proof}

\begin{proposition}\label{prop:relative-goldman-ellipticity}
Let \(\Sigma_{g,n}\) be a connected oriented surface of genus \(g\)
with \(n\ge1\) punctures or boundary components. For
\(G=\mathrm{SL}_2(\mathbb C)\) or \(\mathrm{GL}_2(\mathbb C)\), fix
regular semisimple boundary conjugacy classes
\(\mathcal C_1,\ldots,\mathcal C_n\) satisfying
\(\prod_i\det\mathcal C_i=1\).
The irreducible relative-character and representation loci are
holomorphically elliptic and hence Oka.
\end{proposition}

\begin{proof}
Write \(R_{\mathcal C}(G)\) for the representations with boundary values
in \(\mathcal C_1,\ldots,\mathcal C_n\), and
\(X_{\mathcal C}(G)=R_{\mathcal C}(G)\sslash G\) for their character
variety. First let \(G=\mathrm{SL}_2(\mathbb C)\). Its irreducible
relative-character locus \(Y\) is a smooth complex symplectic manifold.
Indeed, on
\(\mathrm{SL}_2^{\,2g}\times\prod_i\mathcal C_i\), the cokernel of the
relator differential is dual, under the trace pairing, to the common
Lie centralizer of the representation. This centralizer vanishes at
irreducible representations. The relative representation space is
therefore smooth there, and its free \(\mathrm{PGL}_2\)-quotient is
smooth. Its tangent space is parabolic cohomology; the trace pairing and
relative Poincar\'e duality give the nondegenerate holomorphic two-form,
whose closedness follows from the relative Goldman construction
\cite[Sections 3--4 and 7--8]{GHJW1997}.
Choose simple-trace generators \(f_j=f_{\beta_j}\) of
\(\mathbb C[X_{\mathcal C}(\mathrm{SL}_2)]\), omitting the constant
boundary traces. On \(Y\), with Goldman form \(\omega\), their
Hamiltonian fields satisfy
\[
\operatorname{span}\{df_i(x):1\le i\le N\}=T_x^*Y,
\qquad \iota_{W_i}\omega=df_i,
\qquad \operatorname{span}\{W_i(x):1\le i\le N\}=T_xY.
\]
Twists along interior simple curves preserve the prescribed boundary
classes \cite[Section 4]{GX2011}. The cutting construction
\eqref{eq:M8.1} preserves irreducibility and gives each \(W_i\) a
holomorphic flow for every complex time. Their finite flow composition
is a dominating spray on \(Y\times\mathbb C^N\).

The quotient of the irreducible representation locus by
\(\mathrm{PGL}_2\) is a holomorphic principal bundle: the
\(\mathrm{SL}_2\)-stabilizer is its centre, so Luna's slice theorem
gives analytic local sections \cite[proof of Theorem 53]{S2012}.
Restricting this bundle to the fixed boundary classes gives
\[
R_{\mathcal C}(\mathrm{SL}_2)^{\mathrm{irr}}\longrightarrow Y.
\]
The twists lift to complete fields \(\widetilde W_i\) on this
representation locus. The principal-bundle construction above, using
these lifts and the conjugation fields, gives its dominating spray.

For \(\mathrm{GL}_2\), choose
\(\kappa_i^2=\det\mathcal C_i\) with \(\prod_i\kappa_i=1\), and put
\(\mathcal C'_i=\kappa_i^{-1}\mathcal C_i\subset\mathrm{SL}_2\).
Rank-one characters with boundary values \(\kappa_i\) form
\[
T_\kappa=
\{\chi:\pi_1\Sigma_{g,n}\to\mathbb C^*:
\chi(\gamma_i)=\kappa_i\}
\cong(\mathbb C^*)^{2g},
\]
where \(\gamma_i\) are the positively oriented boundary loops.
Taking square roots of the handle determinant values shows that
scalar twisting gives finite unramified coverings
\[
\begin{aligned}
X_{\mathcal C'}(\mathrm{SL}_2)^{\mathrm{irr}}\times T_\kappa
&\longrightarrow X_{\mathcal C}(\mathrm{GL}_2)^{\mathrm{irr}},\\
R_{\mathcal C'}(\mathrm{SL}_2)^{\mathrm{irr}}\times T_\kappa
&\longrightarrow R_{\mathcal C}(\mathrm{GL}_2)^{\mathrm{irr}}.
\end{aligned}
\]
Their deck group consists of the sign characters with boundary values
all equal to one; it is \(\{\pm1\}^{2g}\) and acts freely on
\(T_\kappa\). Formula \eqref{eq:central-twist-equivariance} makes
\(\chi(\beta_j)W_j\) and
\(\chi(\beta_j)\widetilde W_j\) complete and deck-invariant, exactly
as in the closed case. Together with the logarithmic handle-coordinate
fields, and with the conjugation fields in the representation case,
they descend and span everywhere. Finite flow composition gives the
required sprays.
\end{proof}

For the four-holed sphere, the relative character varieties are the
Fricke cubic surfaces \(\{F=0\}\) of Cantat--Loray
\cite[§1.1, equations (1.4)--(1.5)]{CantatLoray2009}. Their three coordinate
Hamiltonian fields are
\[
W_x=F_z\partial_y-F_y\partial_z,\qquad
W_y=F_x\partial_z-F_z\partial_x,\qquad
W_z=F_y\partial_x-F_x\partial_y.
\]
They span the tangent space at every smooth point. Each fixes one
coordinate, while the other two satisfy an affine linear differential
equation with constant coefficients, so its flow exists for every
complex time \cite[Lemma 2.1 and Remark 2.4]{Andrist2025}. The composition
of the three flows is a dominating spray on the smooth locus.

\subsection{Smooth representation varieties}

For the representation variety itself, reducible representations can
also be smooth. The Schur condition identifies the smooth locus, and its
reducible part has an explicit bundle structure.
Write
\[
R_g=\operatorname{Hom}(\pi_1\Sigma_g,\mathrm{SL}_2(\mathbb C)),
\qquad g\ge2.
\]

A representation is called Schur if
\(\operatorname{End}_{\pi_1\Sigma_g}(\mathbb C^2)=\mathbb C I\).

\begin{theorem}\label{thm:M10}
For a closed surface \(\Sigma_g\) of genus \(g\ge2\), the following hold.
\begin{enumerate}[label=\textup{(\roman*)}]
\item The smooth locus of
\(R_g=\operatorname{Hom}(\pi_1\Sigma_g,\mathrm{SL}_2(\mathbb C))\)
is precisely its Schur locus.
\item The reducible Schur locus is a closed smooth Oka submanifold of
\((R_g)_{\mathrm{reg}}\). Sending a representation to its unique invariant
line and the character on that line makes it a holomorphically locally
trivial bundle over
\[
\mathbb P^1\times
\bigl((\mathbb C^*)^{2g}\setminus\{\pm1\}^{2g}\bigr)
\]
with fibre \(\mathbb C\times(\mathbb C^{2g-2}\setminus\{0\})\).
\item The irreducible loci of the \(\mathrm{SL}_2(\mathbb C)\)- and
\(\mathrm{GL}_2(\mathbb C)\)-representation varieties are
holomorphically elliptic. Both smooth representation loci are Oka, of
dimensions \(6g-3\) and \(8g-3\), respectively.
\end{enumerate}
\end{theorem}

\begin{proof}
\textup{(i)} The \(\mathrm{GL}_n\) surface-group representation scheme
is reduced and a complete intersection for \(g>1\), by Simpson
\cite[Proposition 11.3 and Corollary 11.6]{Simpson94}; see also
Bellamy--Schedler \cite[Theorem 2.1]{BS2023}. For \(n=2\), its
dimension is \(8g-3\).
The finite étale determinant-square-root map
\[
R_g\times(\mathbb C^*)^{2g}
\longrightarrow
\operatorname{Hom}(\pi_1\Sigma_g,\mathrm{GL}_2)
\]
transfers reducedness and gives pure dimension \(6g-3\) for \(R_g\).

An infinitesimal deformation written as
\(\rho_\varepsilon(\gamma)=(I+\varepsilon u(\gamma))\rho(\gamma)\),
with \(\varepsilon^2=0\), gives
\[
u(\gamma\delta)=u(\gamma)+\operatorname{Ad}_{\rho(\gamma)}u(\delta).
\]
Thus \(T_\rho R_g=Z^1(\Sigma_g,\operatorname{ad}\rho)\).
The standard cell decomposition has cochain groups
\[
C^0=\mathfrak{sl}_2,\qquad
C^1=\mathfrak{sl}_2^{\,2g},\qquad
C^2=\mathfrak{sl}_2.
\]
Put \(h^0=\dim H^0(\Sigma_g,\operatorname{ad}\rho)\). The adjoint
local system is self-dual under the trace pairing, so Poincaré duality
gives \(\dim H^2=h^0\).
The Euler characteristic and \(\dim B^1=3-h^0\) now give
\[
\dim H^1=6g-6+2h^0,\qquad
\dim Z^1=\dim H^1+3-h^0,
\]
and hence
\begin{equation}
\dim T_\rho R_g=6g-3+h^0.
\label{eq:M10.1}
\end{equation}
Reducedness identifies the scheme tangent space with the tangent space
of the representation variety; see Sikora \cite{S2012} for this
distinction. Comparing \eqref{eq:M10.1} with the pure dimension
\(6g-3\) shows that \(\rho\) is smooth exactly when \(h^0=0\), which
is the Schur condition.

\textup{(ii)} Put a reducible representation in upper-triangular form
and write its diagonal character as
\[
\lambda=(\lambda_1,\mu_1,\ldots,\lambda_g,\mu_g)
\in(\mathbb C^*)^{2g}.
\]
It is Schur precisely when the diagonal characters
\(\lambda,\lambda^{-1}\) are distinct and the extension is nonsplit.

If \(\lambda^2\ne1\), diagonalize a generator with distinct eigenvalues.
A commuting traceless matrix then has the form \(A=\operatorname{diag}(a,-a)\).
For any other generator with matrix \(\left(\begin{smallmatrix}u&c\\0&u^{-1}\end{smallmatrix}\right)\),
\[
\left[A,\begin{pmatrix}u&c\\0&u^{-1}\end{pmatrix}\right]
=\begin{pmatrix}0&2ac\\0&0\end{pmatrix}.
\]
A nonsplit representation has some \(c\ne0\), forcing \(a=0\); a split
one admits \(A=\operatorname{diag}(1,-1)\) in its centralizer. If
\(\lambda^2=1\), all matrices are scalar multiples of upper unipotent
matrices and commute with \(\left(\begin{smallmatrix}0&1\\0&0\end{smallmatrix}\right)\).

The diagonal-character parameters therefore range over
\[
B=(\mathbb C^*)^{2g}\setminus\{\pm1\}^{2g}.
\]
The decomposition
\((\mathbb C^*)^{2g}=(\mathbb C^*)^g\times(\mathbb C^*)^g\) has two
positive-dimensional Oka factors, and the omitted set is finite.
Thus \(B\) is Oka by \cite[Corollary 7.9(ii)]{DW}.

The remaining parameters are cocycles. In an upper-triangular frame,
the representation is
\[
\rho(\gamma)=
\begin{pmatrix}
\lambda(\gamma)&c(\gamma)\lambda(\gamma)^{-1}\\
0&\lambda(\gamma)^{-1}
\end{pmatrix}.
\]
The homomorphism condition reads
\[
c(\gamma\delta)
=c(\gamma)+\lambda(\gamma)^2c(\delta).
\]
For the standard generators \(a_i,b_i\), put \(p_i=c(a_i)\) and
\(q_i=c(b_i)\). The surface relation is
\begin{equation}
\sum_{i=1}^g
\bigl((1-\mu_i^2)p_i+(\lambda_i^2-1)q_i\bigr)=0.
\label{eq:M10.2}
\end{equation}

To check this relation, set \(\chi=\lambda^2\). Applying the cocycle
identity to \(\gamma\gamma^{-1}=1\) gives
\[
c(\gamma^{-1})=-\chi(\gamma)^{-1}c(\gamma).
\]
For \([a_i,b_i]=a_i b_i a_i^{-1}b_i^{-1}\), it then gives
\[
\begin{aligned}
c([a_i,b_i])
&=p_i+\lambda_i^2q_i
 -\lambda_i^2\mu_i^2\lambda_i^{-2}p_i
 -\mu_i^2\mu_i^{-2}q_i\\
&=(1-\mu_i^2)p_i+(\lambda_i^2-1)q_i.
\end{aligned}
\]
Since \(\chi([a_i,b_i])=1\), the product of commutators gives
\eqref{eq:M10.2}.

On \(B\), the linear functional in \eqref{eq:M10.2} is nonzero, so
its kernel is a rank-\((2g-1)\) vector bundle \(\mathcal Z\). Split
extensions form the line subbundle \(\mathcal B\) generated by
\begin{equation}
(1-\lambda_1^2,1-\mu_1^2,\ldots,
1-\lambda_g^2,1-\mu_g^2).
\label{eq:M10.3}
\end{equation}
This vector is nowhere zero on \(B\). In a frame adapted to
\((\mathcal Z,\mathcal B)\), the nonsplit locus in a fibre is
\begin{equation}
\mathbb C\times(\mathbb C^{2g-2}\setminus\{0\}).
\label{eq:M10.4}
\end{equation}

To let the unique invariant line vary in \(\mathbb P^1\), associate
\(\mathcal Z\setminus\mathcal B\) to the Borel principal bundle
\[
\mathrm{SL}_2\to\mathbb P^1.
\]
Upper-triangular conjugation preserves the diagonal character and the
split-extension subbundle. Local triangular frames, together with frames
adapted to \((\mathcal Z,\mathcal B)\), give a holomorphically locally
trivial bundle over
\[
\mathbb P^1\times B
\]
with fibre \eqref{eq:M10.4}.

Indeed, the conjugation action is
\[
\begin{pmatrix}t&u\\0&t^{-1}\end{pmatrix}
\begin{pmatrix}\lambda&c\lambda^{-1}\\0&\lambda^{-1}\end{pmatrix}
\begin{pmatrix}t&u\\0&t^{-1}\end{pmatrix}^{-1}
=
\begin{pmatrix}
\lambda&
\bigl(t^2c+tu(1-\lambda^2)\bigr)\lambda^{-1}\\
0&\lambda^{-1}
\end{pmatrix}.
\]
Thus a cocycle vector transforms by \(c\mapsto t^2c+tu\,v_\lambda\),
where \(v_\lambda\) is the nowhere-zero vector \eqref{eq:M10.3}.
This is a holomorphic affine automorphism of each fibre of
\(\mathcal Z\), preserving both the line \(\mathcal B\) and its
complement. On an open subset of \(B\), take a vector-bundle frame with
first vector \(v_\lambda\). The complement is then
\[
\{(z_0,z_1,\ldots,z_{2g-2}):(z_1,\ldots,z_{2g-2})\ne0\}.
\]
Combining these frames with local sections of
\(\mathrm{SL}_2\to\mathbb P^1\) gives the local trivializations. The
total space parametrizes the reducible Schur representations.

The reducible locus is closed in \(R_g\), being the proper image of
the closed invariant-line incidence in \(R_g\times\mathbb P^1\).
Its intersection with the Schur locus is consequently closed.

A nonsplit reducible representation has a unique invariant line, and
that line depends holomorphically on the representation. To see this
locally near \(\rho_0\), choose a generator \(\gamma\) with distinct
eigenvalues and label its holomorphically varying eigenlines
\(\ell_+(\rho),\ell_-(\rho)\) so that \(\ell_+(\rho_0)\) is the common
invariant line. Choose another generator \(\eta\) satisfying
\[
\rho_0(\eta)\ell_-(\rho_0)\ne\ell_-(\rho_0).
\]
This inequality persists near \(\rho_0\). On this reducible
neighbourhood, \(\ell_+(\rho)\) is therefore the unique common
invariant line. It gives a holomorphic inverse to the bundle
parametrization \eqref{eq:M10.2}--\eqref{eq:M10.4}, for the induced
complex structure on the closed reducible Schur locus.

The base \(\mathbb P^1\times B\) is Oka. Since \(2g-2\ge2\), the
fibre \eqref{eq:M10.4} is Oka as well, proving the Oka assertion in
\textup{(ii)}.

\textup{(iii)} The effective \(\mathrm{PGL}_2\)-action on
\(R_g^{\mathrm{irr}}\) is free. Since the stabilizer for the original
\(\mathrm{SL}_2\)-action is its centre, Luna's slice theorem gives
analytic local sections of the quotient map
\cite[proof of Theorem 53]{S2012}. Thus
\[
q:R_g^{\mathrm{irr}}\longrightarrow\mathcal X_g^{\mathrm{irr}}
\]
is a holomorphic principal \(\mathrm{PGL}_2\)-bundle.

The twists in \eqref{eq:M8.1} give complete holomorphic lifts
\(\widetilde W_j\) of the spanning fields \(W_j\). The principal-bundle
construction above therefore proves the holomorphic ellipticity of
\(R_g^{\mathrm{irr}}\).

For \(\mathrm{GL}_2\), use the finite étale cover
\[
R_g^{\mathrm{irr}}\times T\longrightarrow
\operatorname{Hom}(\pi_1\Sigma_g,\mathrm{GL}_2)^{\mathrm{irr}}
\]
with the same deck action as above. Formula
\eqref{eq:central-twist-equivariance} also holds for the twists before
taking conjugacy classes, as follows directly from their stable-letter
multiplication and separating-side conjugation formulas. Hence the
weighted fields \(\chi(\gamma_j)\widetilde W_j\) are complete and
deck-invariant. The three conjugation fields and the logarithmic torus
fields are also complete and deck-invariant. Their values span the
tangent bundle of the covering space, so the descended fields give a
dominating spray on the irreducible \(\mathrm{GL}_2\)-representation
locus.

By \textup{(ii)}, the reducible locus in \((R_g)_{\mathrm{reg}}\) is
closed analytic of dimension \(1+2g+(2g-1)=4g\). Its codimension is
\(2g-3\), equal to one for \(g=2\). It contains no component of the
smooth locus, and its complement \(R_g^{\mathrm{irr}}\) is Oka by the
ellipticity just proved. Analytic extension \cite[Theorem 1.9]{DW} now
makes \((R_g)_{\mathrm{reg}}\) Oka.

Finally, the finite étale determinant-square-root quotient preserves
smooth loci and gives the Oka property for the smooth
\(\mathrm{GL}_2\)-representation locus. Its dimension is \(8g-3\);
the \(\mathrm{SL}_2\) locus has dimension \(6g-3\).
\end{proof}

\subsection{Parabolic connections and Garnier spaces}\label{sec:M11}

Fix a smooth projective curve \(C\) of genus \(g\), a reduced marked
divisor
\[
D=t_1+\cdots+t_n
\]
and exponents
\begin{equation}
\lambda_i\in\mathbb C,\qquad 2\lambda_i\notin\mathbb Z.
\label{eq:M11.1}
\end{equation}
Consider rank-two logarithmic parabolic connections with trivial
determinant and its trivial connection. At each puncture, the residue
eigenvalues are \(\lambda_i,-\lambda_i\), and the marked line
\(\ell_i\subset E_{t_i}\) satisfies
\[
(\operatorname{Res}_{t_i}\nabla-\lambda_i I)\ell_i=0.
\]
As in Section~\ref{sec:bundle-moduli}, use descending flags. For
\(E_{t_i}\supset\ell_i\supset0\), choose arbitrary real weights
\(0\le\alpha_{i,1}<\alpha_{i,2}<1\) on
\(E_{t_i}/\ell_i\) and \(\ell_i\), respectively.
An invariant line with residues \(\varepsilon_i\lambda_i\) has degree
\(-\sum_i\varepsilon_i\lambda_i\), by the residue formula for a line
connection. Thus only sign patterns with integral sum can occur, and
their stability scores are
\begin{equation}
\delta_\varepsilon(\boldsymbol\alpha)
=-\sum_i\varepsilon_i\lambda_i+
\frac12\sum_i\varepsilon_i(\alpha_{i,2}-\alpha_{i,1}).
\label{eq:M11-score}
\end{equation}
Stability depends only on the signs of these finitely many rational
affine functions. Adding the same small constant to both weights at
each marked point preserves every score. Choose these shifts to make
all weights strictly positive and pairwise distinct across marked
points. A sufficiently near rational point in the same relative face
of the wall arrangement preserves these properties and gives exactly
the same stable objects. We use the resulting moduli space for the
original real weights.

We work in the range
\begin{equation}
g=0,\ n\ge4;\qquad
g=1,\ n\ge2;\qquad
g\ge2,\ n\ge1.
\label{eq:M11.2}
\end{equation}
Write the stable moduli space as
\[
\mathcal M_{\mathrm{dR}}^{\boldsymbol\alpha}(C,D;\lambda).
\]

\begin{theorem}\label{thm:M11}

Under \eqref{eq:M11.1}--\eqref{eq:M11.2}, the moduli space
\(\mathcal M_{\mathrm{dR}}^{\boldsymbol\alpha}(C,D;\lambda)\) is a
nonempty smooth irreducible quasiprojective Oka variety of dimension
\[
6g-6+2n.
\]
If
\[
\sum_i\varepsilon_i\lambda_i\notin\mathbb Z
\qquad\text{for every }\varepsilon\in\{\pm1\}^n,
\]
then it is holomorphically elliptic. This includes \(n=1\).
For \(C=\mathbb P^1\) and generic parabolic weights, fixing the
marked curve gives the \(2(n-3)\)-dimensional initial-data moduli space
of a nonresonant Garnier system in \(n-3\) variables; \(n=4\) gives
Painlev\'e VI.

\end{theorem}

The nonresonance and sign-sum conditions together are the generic-exponent
conditions of Inaba \cite[Definition 2.2]{I2012}.

Over a component of reducible characters, a nonempty stable inverse
image under the Riemann--Hilbert morphism is a character torus times a
projective space. We construct these products in holomorphic families,
then compare arbitrary weights with nearby generic weights. Put
\[
\mu_i=e^{-2\pi i\lambda_i}.
\]
By \eqref{eq:M11.1}, \(\mu_i\ne\mu_i^{-1}\). Let \(\mathcal X\) be
the relative \(\mathrm{SL}_2\) character variety with puncture traces
\(\mu_i+\mu_i^{-1}\).

Inaba--Iwasaki--Saito \cite[Theorems 1.1, 1.3, 1.4, and 1.5]{IIS2006}
construct the smooth moduli space in genus zero, and Inaba
\cite[Theorem 2.1]{I2012} constructs the fine moduli space for general
marked curves. The determinant map is smooth \cite[Section 5]{I2012};
its fibre with fixed determinant and connection is therefore smooth and
quasiprojective. For generic weights this fibre is nonempty and
irreducible in the ranges \eqref{eq:M11.2}
\cite[Propositions 5.1--5.3 and Remark 6.1]{I2012}.
The holomorphic Riemann--Hilbert morphism is
\[
\mathrm{RH}:\mathcal M_{\mathrm{dR}}^{\boldsymbol\alpha}
\longrightarrow\mathcal X.
\]
Changing the weights leaves the irreducible connection locus
unchanged. For generic weights, \(\mathrm{RH}\) is an analytic
isomorphism over irreducible characters
\cite[Section 6, pp.~461--462]{I2012}, and hence gives the same
isomorphism for every weight system. Over a reducible component, its
inverse image has the following form. Define the finite set
\[
\mathcal E(\lambda)=
\left\{\varepsilon\in\{\pm1\}^n:
\sum_i\varepsilon_i\lambda_i\in\mathbb Z\right\}
\big/\bigl(\varepsilon\sim-\varepsilon\bigr).
\]

\begin{lemma}\label{lem:M11-exceptional}
The reducible character locus is a disjoint union of closed tori
\(Z_{[\varepsilon]}\cong(\mathbb C^*)^{2g}\), indexed by
\([\varepsilon]\in\mathcal E(\lambda)\).
If \(\delta_\varepsilon(\boldsymbol\alpha)=0\), then
\(\mathrm{RH}^{-1}(Z_{[\varepsilon]})=\varnothing\).
For each component with
\(\delta_\varepsilon(\boldsymbol\alpha)\ne0\),
\begin{equation}
\mathrm{RH}^{-1}(Z_{[\varepsilon]})\cong
(\mathbb C^*)^{2g}\times\mathbb P^{2g+n-3}
\label{eq:M11.9}
\end{equation}
as complex manifolds. Each nonempty inverse image is a closed smooth
Oka submanifold of the stable connection moduli space.
\end{lemma}

\begin{proof}
The product coordinates consist of handle monodromies for the line
character and normalized cocycles for its extension class. We first
construct the logarithmic correspondence in holomorphic families.
The stability score selects the extension direction, and recovering
the same coordinates from a connection gives the holomorphic inverse.

\smallskip
\noindent\emph{The nonresonant family correspondence.}
With the prescribed residue branches, the distinct local eigenlines
determine the logarithmic lattice uniquely. Deligne's regular-singular
extension \cite[Chapter II, Proposition 5.4]{Del70}, followed by the
prescribed integer elementary transformations along those eigenlines,
gives the required residue branches. Inaba
\cite[Proposition 3.1 and the proof of Proposition 6.1]{I2012}
uses these transformations and the relative extension theorem in
holomorphic families. The local construction below
gives this dependence for the monodromy families used in the product
parametrization.

For a puncture monodromy matrix \(T_i\) in a holomorphic family, set
\[
P_i^+=\frac{T_i-\mu_i^{-1}I}{\mu_i-\mu_i^{-1}},
\qquad
P_i^-=\frac{T_i-\mu_i I}{\mu_i^{-1}-\mu_i},
\qquad
A_i=\lambda_iP_i^+-\lambda_iP_i^-.
\]
Cayley--Hamilton gives \((T_i-\mu_iI)(T_i-\mu_i^{-1}I)=0\), and hence
\[
(P_i^\pm)^2=P_i^\pm,\quad
P_i^+P_i^-=0,\quad P_i^++P_i^-=I,\quad
e^{-2\pi iA_i}=T_i.
\]
Consequently
\[
d+A_i\,\frac{dz}{z}
\]
is a holomorphic family of logarithmic connections on the disc with the
required residue branches and local monodromy. Over the punctured disc,
a frame at one point determines a flat isomorphism with the given local
system. This isomorphism is holomorphic in the parameters. Gluing the
disc models to the flat bundle on \(C\setminus D\) gives the family.
Its determinant monodromy is trivial, and
\(\operatorname{tr}A_i=0\) at every puncture. A nonzero flat determinant
section on \(C\setminus D\) therefore extends without zeros in the disc
models, giving
\[
(\det E,\det\nabla)\simeq(\mathcal O_C,d)
\]
locally in the parameter space.

To check uniqueness and holomorphic dependence, choose holomorphic local
eigenframes in the parameter space and write a logarithmic connection as
\[
d+A(z)\frac{dz}{z},\qquad
A(z)=A_0+\sum_{j\ge1}A_jz^j,\qquad
A_0=\operatorname{diag}(\lambda_i,-\lambda_i).
\]
A gauge \(G(z)=I+\sum_{j\ge1}G_jz^j\) to the constant-residue model
must satisfy
\[
zG'=GA_0-A(z)G,
\]
and comparing coefficients gives
\[
\bigl(qI+\operatorname{ad}_{A_0}\bigr)G_q
=-\sum_{j=1}^qA_jG_{q-j},
\qquad q\ge1.
\]
The eigenvalues on the left are \(q,q,q+2\lambda_i,q-2\lambda_i\),
all nonzero by \eqref{eq:M11.1}. One constant \(K\) bounds the inverse
norm by \(K/q\): this holds for large \(q\), and enlarging \(K\)
handles the finitely many smaller \(q\). On a sufficiently small
parameter neighbourhood, the coefficient norms \(a_j=\sup\|A_j\|\)
have positive convergence radius. The following majorant gives a common
disc for the normalizing gauges. For \(g_0=1\) and
\[
qg_q=K\sum_{j=1}^qa_jg_{q-j},
\]
induction gives \(\|G_q\|\le g_q\). The majorant series is
\[
\sum_{q\ge0}g_qz^q
=\exp\!\left(K\sum_{j\ge1}\frac{a_j}{j}z^j\right),
\]
so \(G\) converges locally uniformly and depends holomorphically on
the parameters. Since \(G(0)=I\), it is invertible on a smaller disc.

A single-valued horizontal morphism \(H\) between two constant-residue
models satisfies
\[
zH'=HA_0-A_0H.
\]
Each matrix coefficient has the form \(c\,z^\beta\), with \(\beta=0\)
on the diagonal and \(\beta=\pm2\lambda_i\) off the diagonal.
Single-valuedness forces the off-diagonal coefficients to vanish, while
the diagonal coefficients are constant. A flat isomorphism therefore
extends uniquely across the puncture and preserves the labelled
eigenlines. This completes the family logarithmic correspondence.

\smallskip
\noindent\emph{Reducible characters and extension parameters.}
Every reducible semisimple character is \(\nu\oplus\nu^{-1}\). At
the punctures its line character satisfies
\begin{equation}
\nu(\gamma_i)=\mu_i^{\varepsilon_i},\qquad \varepsilon_i\in\{1,-1\},
\qquad
\prod_{i=1}^n\nu(\gamma_i)=1.
\label{eq:M11.3}
\end{equation}
The product relation is equivalent to
\(\sum_i\varepsilon_i\lambda_i\in\mathbb Z\). For each such pattern,
the \(2g\) handle values of \(\nu\) range over the torus
\(T_\varepsilon\cong(\mathbb C^*)^{2g}\).

Simultaneous inversion pairs the patterns. The two patterns in each
pair are distinct because the local eigenvalues are distinct. Choosing
one representative of each \([\varepsilon]\in\mathcal E(\lambda)\),
the handle values parametrize the corresponding component
\(Z_{[\varepsilon]}\) by \((\mathbb C^*)^{2g}\).

The parametrization has an algebraic inverse. For a puncture loop
\(\gamma_0\) with selected eigenvalue \(\mu_0\) and any handle generator
\(a\),
\begin{equation}
\nu(a)=
\frac{\operatorname{tr}\rho(\gamma_0a)
-\mu_0^{-1}\operatorname{tr}\rho(a)}
{\mu_0-\mu_0^{-1}}.
\label{eq:M11.4}
\end{equation}
Trace coordinates therefore recover the line character. Applying the
same formula to the other puncture loops recovers their discrete
eigenvalue choices. The invariant-line incidence is closed in
the product of the relative representation variety with
\(\mathbb P^1\), and its projection is proper. Its image is conjugation
invariant and closed, so its image
in the good quotient is closed as well. The finitely many spectral
patterns recovered by the trace coordinates separate this closed
reducible locus into the disjoint closed components
\(Z_{[\varepsilon]}\).

If \(n=1\), the rank-one relation would force \(\nu(\gamma_1)=1\),
contrary to \eqref{eq:M11.1}. The reducible locus is then empty, and
the irreducible argument covers the whole moduli space.

Now suppose a reducible component exists, so \(n\ge2\). The
punctured-surface group is free of rank
\[
s=2g+n-1.
\]
Choose free generators \(\eta_0,\ldots,\eta_{s-1}\), with
\(\eta_0=\gamma_{i_0}\) a puncture loop. Its fixed value satisfies
\(\nu(\eta_0)^2\ne1\).

Nonsplit extensions with invariant line character \(\nu\) are
classified by nonzero classes in
\begin{equation}
\operatorname{Ext}^1(\nu^{-1},\nu)
=H^1(C\setminus D,\nu^2).
\label{eq:M11.5}
\end{equation}
A cocycle on the free group is specified by arbitrary values at its
\(s\) generators. A coboundary changes these values by
\[
c(\gamma)\longmapsto
c(\gamma)+(1-\nu(\gamma)^2)h,
\qquad h\in\mathbb C.
\]
The fixed nonzero scalar \(1-\nu(\eta_0)^2\) gives a unique choice
of \(h\) for which \(c(\eta_0)=0\). The values
\(c(\eta_1),\ldots,c(\eta_{s-1})\) are unrestricted. This scalar is
constant along \(T_\varepsilon\), so the same normalization depends
holomorphically on the character throughout the torus.

Let \(\mathcal H_\varepsilon\to T_\varepsilon\) be the cohomology bundle with fibre
\((\mathcal H_\varepsilon)_\nu=H^1(C\setminus D,\nu^2)\).
The remaining cocycle values therefore give the global trivialization
\begin{equation}
\mathcal H_\varepsilon\cong T_\varepsilon\times\mathbb C^{2g+n-2}.
\label{eq:M11.6}
\end{equation}
To pass from extension classes to isomorphism classes, quotient by scalar
changes of the two line factors. This gives
\begin{equation}
\mathbb P(\mathcal H_\varepsilon)\cong
T_\varepsilon\times\mathbb P^{2g+n-3}.
\label{eq:M11-projectivization}
\end{equation}
Indeed, an upper-triangular change of frame preserving
\(c(\eta_0)=0\) has zero unipotent parameter. A diagonal
determinant-one change acts on the remaining coordinates by
\(c\mapsto t^2c\). Since \(t\mapsto t^2\) surjects onto
\(\mathbb C^*\), its nonzero orbits are the scalar projective classes.
On the chart where \(c_j\) is nonzero, the ratios \(c_\ell/c_j\),
for \(\ell\ne j\), are holomorphic quotient coordinates with the
usual projective transition functions.

\smallskip
\noindent\emph{The stable extension direction.}
Stability depends on the parabolic degree of the invariant line. The
logarithmic extension of the rank-one local system \(\nu\) has residues
\(\varepsilon_i\lambda_i\), which fix its ordinary degree:
\begin{equation}
\deg L_\nu=-\sum_i\varepsilon_i\lambda_i.
\label{eq:M11.7}
\end{equation}
To compute this degree, choose a nonzero meromorphic section
\(s\) of \(L_\nu\) and write \(\nabla s=\omega s\). In a local
holomorphic frame at \(p\), writing \(s=z^{m_p}e\) gives
\[
\omega=m_p\,\frac{dz}{z}+\frac{\nabla e}{e}.
\]
Thus \(\operatorname{res}_p\omega=m_p\) away from \(D\), while
\(\operatorname{res}_{t_i}\omega=m_{t_i}+\varepsilon_i\lambda_i\)
at a puncture. The sum of the residues of \(\omega\) is zero, so
\[
0=\sum_p m_p+\sum_i\varepsilon_i\lambda_i
=\deg L_\nu+\sum_i\varepsilon_i\lambda_i,
\]
which proves \eqref{eq:M11.7}. The sum is integral by
\eqref{eq:M11.3}. At each puncture, the induced parabolic weight is
determined by the chosen eigenline and the fixed weights, independently
of the handle monodromy and extension class.

By the flag convention above, the induced weight is
\[
w_i(\nu)=
\begin{cases}
\alpha_{i,2},&\varepsilon_i=1,\\
\alpha_{i,1},&\varepsilon_i=-1.
\end{cases}
\]
The stability score is
\[
\delta(\nu)=
\deg L_\nu+\sum_iw_i(\nu)
-\frac12\sum_i(\alpha_{i,1}+\alpha_{i,2}).
\]
The opposite line character has the opposite degree and complementary
weights, giving
\begin{equation}
\delta(\nu^{-1})=-\delta(\nu).
\label{eq:M11.8}
\end{equation}
Here \(\delta(\nu)=\delta_\varepsilon(\boldsymbol\alpha)\) in
\eqref{eq:M11-score}. A nonsplit extension has a unique invariant line,
and is stable exactly when that line has negative score. If the score
is zero, both extension directions have score zero and the stable
inverse image of \(Z_{[\varepsilon]}\) is empty. If it is nonzero,
exactly one direction is stable, and every nonzero class in that
direction is stable. A split connection has both invariant lines, whose
opposite scores make it unstable for every weight system.

For each nonzero-score component, choose its representative
\(\varepsilon\) so that \(\delta(\nu)<0\). The stable classes then fill
the projective fibre in \eqref{eq:M11-projectivization}.

\smallskip
\noindent\emph{The analytic product structure.}
Fix a nonzero-score component \(Z=Z_{[\varepsilon]}\) and use its stable
extension direction. On each projective chart of
\(\mathbb P(\mathcal H_\varepsilon)\), the normalized cocycles
\(c(\eta_0)=0\) in \eqref{eq:M11.5}--\eqref{eq:M11.6} and the ratios
\(c_\ell/c_j\) give a holomorphic monodromy family. Changes of
frame exist locally on the overlaps, and conjugate families induce the
same map to the coarse moduli space. The family logarithmic
correspondence therefore gives compatible maps to the stable connection
moduli space. The line character specifies their image under the
Riemann--Hilbert morphism, yielding a holomorphic map
\[
T_\varepsilon\times\mathbb P^{2g+n-3}
\longrightarrow\mathrm{RH}^{-1}(Z).
\]

Conversely, restrict the universal family supplied by
\cite[Theorem 2.1]{I2012} to the fixed-determinant connection moduli
space and choose a local frame at a fixed point of \(C\setminus D\).
On the reducible locus, the unique invariant line is recovered from the distinct
eigenlines of the fixed puncture monodromy and varies holomorphically.
The character on that line gives the handle coordinates. Applying the
normalization \(c(\eta_0)=0\) gives
\[
(E,\nabla)\longmapsto
\bigl(\nu,[c(\eta_1):\cdots:c(\eta_{s-1})]\bigr),
\qquad s=2g+n-1.
\]
The projective coordinates are invariant under the remaining frame
changes by \eqref{eq:M11-projectivization}, so these local expressions
give a holomorphic map on the moduli space. The two constructions recover
the same handle values and projective cocycle coordinates; their
compositions are identities. This proves the analytic product
\eqref{eq:M11.9}.

The product identification \eqref{eq:M11.9} proves smoothness. Since
\(Z\) is closed analytic in the character variety,
\(\mathrm{RH}^{-1}(Z)\) is closed analytic in the connection moduli
space, hence a closed complex submanifold. Both factors in
\eqref{eq:M11.9} are Oka, so the inverse image is Oka.
\end{proof}

\begin{proof}[Proof of Theorem~\ref{thm:M11}]
First replace \(\boldsymbol\alpha\) by rational weights
\(\boldsymbol\alpha_0\) in the same relative face, as above. Choose
generic rational weights \(\boldsymbol\beta\) sufficiently close to
\(\boldsymbol\alpha_0\) to preserve the signs of all its nonzero scores.
For \(\boldsymbol\beta\), the connection moduli space is nonempty,
smooth, irreducible and quasiprojective, of dimension \(6g-6+2n\), by
the results of Inaba recalled above. Lemma~\ref{lem:M11-exceptional}
identifies its reducible locus as a finite union of closed smooth
submanifolds, each of dimension
\[
2g+(2g+n-3)=4g+n-3.
\]
Their codimension is \(2g+n-3>0\) by \eqref{eq:M11.2}. Its irreducible
connection locus \(U\) is therefore nonempty and dense. Under
\(\mathrm{RH}\), it is analytically isomorphic to the irreducible
character locus, and hence Oka by
Proposition~\ref{prop:relative-goldman-ellipticity}.

Every \(\boldsymbol\alpha_0\)-stable connection is
\(\boldsymbol\beta\)-stable. Indeed, irreducible connections are stable
for every weight system, and a stable reducible connection has a unique
invariant line with negative score, which remains negative under the
chosen perturbation. Inaba's stability openness
\cite[Proposition 4.1]{I2012}, applied to the universal family over the
\(\boldsymbol\beta\)-moduli space, gives a Zariski open immersion
\[
\mathcal M_{\mathrm{dR}}^{\boldsymbol\alpha}
=\mathcal M_{\mathrm{dR}}^{\boldsymbol\alpha_0}
\hookrightarrow
\mathcal M_{\mathrm{dR}}^{\boldsymbol\beta}.
\]
Its image is obtained by deleting the entire reducible inverse images
whose score at \(\boldsymbol\alpha_0\) is zero. It contains the same
nonempty dense open set \(U\), so it is nonempty, smooth, irreducible
and quasiprojective, with the stated dimension. The remaining reducible
inverse images form a proper closed analytic subset. Analytic extension
\cite[Theorem 1.9]{DW} proves the Oka property.

If every sign sum is nonintegral, then
\(\mathcal E(\lambda)=\varnothing\), so every character is irreducible.
The Riemann--Hilbert morphism is an analytic isomorphism on the whole
moduli space, and transporting the spray from
Proposition~\ref{prop:relative-goldman-ellipticity} proves holomorphic
ellipticity. For \(n=1\), condition \eqref{eq:M11.1} already makes
\(\pm\lambda_1\) nonintegral.
\end{proof}

\begin{example}\label{ex:cayley-parabolic-connections}
Let \(C=\mathbb P^1\), take four marked points, and set
\(\lambda_i=1/4\) for every \(i\). The admissible sign patterns are
the all-positive and all-negative patterns, together with the six
patterns having two signs of each kind. Modulo simultaneous inversion,
they give four reducible characters. All four puncture traces are zero,
so the relative character surface is the Cayley cubic
\cite[§1.1, equations (1.4)--(1.5)]{CantatLoray2009},
\[
x^2+y^2+z^2+xyz=4.
\]
Its four nodes are
\[
(-2,-2,-2),\qquad (-2,2,2),\qquad
(2,-2,2),\qquad (2,2,-2).
\]
For generic parabolic weights, their inverse images under
\(\mathrm{RH}\) are four pairwise disjoint copies of \(\mathbb P^1\).
Theorem~\ref{thm:M11} makes the smooth connection surface Oka for every
parabolic weight system.
\end{example}

\begin{example}[One puncture on an elliptic curve]
\label{ex:one-punctured-elliptic-connections}
Let \(E\) be an elliptic curve with one marked point and let
\(2\lambda\notin\mathbb Z\). For every parabolic weight system, the
rank-two stable logarithmic connection moduli space with trivial
determinant connection and residue eigenvalues \(\lambda,-\lambda\)
is a nonempty,
smooth, irreducible quasiprojective surface and is holomorphically
elliptic.

Put \(t=2\cos(2\pi\lambda)\), so \(t\ne\pm2\). The relative character
surface is the Fricke cubic
\[
V_t=\{x^2+y^2+z^2-xyz-2=t\}
\]
in the coordinates
\((x,y,z)=(\operatorname{tr}A,\operatorname{tr}B,\operatorname{tr}AB)\)
\cite[Introduction]{GoldmanNeumann2005}. For
\(\kappa=x^2+y^2+z^2-xyz-2\), its critical-point equations are
\(2x=yz\), \(2y=xz\), \(2z=xy\). They imply either
\((x,y,z)=0\), with critical value \(-2\), or
\(x^2=y^2=z^2=4\) and \(xyz=8\), with critical value \(2\).
Thus \(V_t\) is smooth. It is nonempty by taking \(x=y=0\) and
\(z^2=t+2\). Viewed as a monic quadratic in \(z\), its equation has
discriminant
\[
(x^2-4)y^2+4(t+2-x^2).
\]
This is not a square in \(\mathbb C[x,y]\), since a square of degree two
in \(y\) would have a square leading coefficient in \(\mathbb C[x]\),
whereas \(x^2-4\) is not a square. Hence \(V_t\) is irreducible.
A reducible pair \(A,B\) is simultaneously upper triangular and has
\(\operatorname{tr}[A,B]=2\). All characters on \(V_t\) are therefore
irreducible.

Inaba's construction \cite[Theorem 2.1]{I2012} supplies the stable
connection moduli space. The irreducible representation quotient over
\(V_t\) is a holomorphic principal \(\mathrm{PGL}_2\)-bundle, so local
sections give holomorphic monodromy families. The nonresonant family
construction in the proof of Lemma~\ref{lem:M11-exceptional} extends
them to logarithmic families with the prescribed residues and trivial
determinant connection. Every such connection is irreducible and stable
for all weights. Conjugate monodromy families define the same classifying
map, so the local maps glue. Conversely, local framings of the universal
connection family give holomorphic monodromy maps. Uniqueness of the
nonresonant logarithmic extension makes these maps inverse to one
another. Thus the Riemann--Hilbert morphism is a biholomorphism onto
\(V_t\). Proposition~\ref{prop:relative-goldman-ellipticity} gives its
holomorphic ellipticity, and the construction and the cubic model give
the remaining assertions.
\end{example}

\section{Complex torus fibrations and compactifications}\label{sec:fibrations}

\subsection{Complex torus fibrations over curves}\label{sec:general-abelian-fibrations}

The translation fields of smooth complex torus fibres extend over
singular fibres as commuting vertical vector fields. Properness makes
their flows complete. Applying these flows to a section gives the
following criterion and entire maps.

\begin{theorem}\label{thm:abelian-fibration}
Let \(f:X\to B\) be a proper surjective holomorphic map from a
connected complex manifold of dimension \(r+1\), \(r\ge1\), to a
connected Riemann surface, with a holomorphic section \(\sigma\).
Let \(B^\circ\subset B\) be a nonempty open set with proper closed
analytic complement. Suppose that \(f\) is smooth over \(B^\circ\)
and that every fibre there is a connected complex torus of dimension
\(r\). Then \(X\) is Oka if and only if \(B\) is Oka.

Suppose \(B\) is Oka. Let \(\alpha:\mathbb C\to B\) be a
biholomorphism onto an affine chart when \(B=\mathbb P^1\), or a
universal covering otherwise. There is a holomorphic map
\[
F:\mathbb C\times\mathbb C^r\longrightarrow X
\]
such that
\[
f(F(z,t))=\alpha(z),\qquad
U:=f^{-1}\bigl(B^\circ\cap\alpha(\mathbb C)\bigr)
\subset F(\mathbb C^{r+1}),
\]
and \(dF_q\) is an isomorphism for every \(q\in F^{-1}(U)\).
If \(X\) is projective and \(B\) is Oka, then \(X\) is also
holomorphically elliptic.
\end{theorem}

\begin{proof}
The identity \(f\circ\sigma=\mathrm{id}_B\) makes \(B\) a
holomorphic retract of \(X\), so \(X\) Oka implies \(B\) Oka.
For the converse, assume \(B\) is Oka. By
\cite[Corollary 2.10]{For}, it is \(\mathbb P^1\), \(\mathbb C\),
\(\mathbb C^*\), or an elliptic curve, so \(\alpha\) exists as stated.

The coherent sheaf of vertical vector fields and its direct image are
\[
\mathcal V=\ker(df:T_X\longrightarrow f^*T_B),
\qquad \mathcal E=f_*\mathcal V.
\]
Grauert's proper direct image theorem makes \(\mathcal E\) coherent
on \(B\) \cite[Section 6, Hauptsatz I]{Grauert60}.
It is torsion-free: for a nonzero local holomorphic function \(h\)
on \(B\) and a local section \(v\) of \(\mathcal E\),
\[
hv=0\quad\Longrightarrow\quad(f^*h)v=0\quad\Longrightarrow\quad v=0.
\]
Here \(f^*h\) is not identically zero on any open subset of the
connected manifold \(X\), and \(T_X\) is locally free. Thus
\(\mathcal E\) is locally free on the smooth curve \(B\).

Over \(X^\circ=f^{-1}(B^\circ)\), one has
\(\mathcal V=T_{X^\circ/B^\circ}\). Every holomorphic vector field
on a complex torus \(A\) is translation-invariant, since its
coefficients in the translation trivialization are holomorphic
functions on a compact connected manifold. Hence
\[
H^0(A,T_A)\longrightarrow T_xA
\quad\text{is an isomorphism for every }x\in A,
\qquad h^q(A,T_A)=r\binom rq.
\]
Grauert's cohomology and base change theorem
\cite[Section 7, Satz 5]{Grauert60} gives
\[
\mathcal E\otimes k(b)\simeq H^0(X_b,T_{X_b}),
\qquad b\in B^\circ,
\]
and the evaluation isomorphism
\begin{equation}\label{eq:abelian-evaluation}
f^*(\mathcal E|_{B^\circ})\longrightarrow T_{X^\circ/B^\circ}.
\end{equation}
In particular, \(\mathcal E\) has rank \(r\).

Form the pullback
\[
Y=X\times_B\mathbb C,\qquad
g:Y\to\mathbb C,\qquad q:Y\to X.
\]
The map \(q\) is a local biholomorphism and \(g\) is proper, so \(Y\)
is smooth. The Oka--Grauert principle
\cite[Theorem 5.3.1]{F17} trivializes \(\alpha^*\mathcal E\).
Evaluating a holomorphic frame gives vertical holomorphic vector
fields \(V_1,\ldots,V_r\) on \(Y\). Over
\(\alpha^{-1}(B^\circ)\), their values form a basis of the fibre
tangent spaces by \eqref{eq:abelian-evaluation}; the fields are
translation-invariant there. Their brackets vanish on this dense
open set and hence on all of \(Y\).

Properness makes the fields complete. For a compact neighbourhood
\(K\subset\mathbb C\), the set \(g^{-1}(K)\) is compact. The local
holomorphic ODE theorem on finitely many coordinate charts gives a
common \(\varepsilon_K>0\) for the flow \(\Phi_j(t,y)\) whenever
\(y\in g^{-1}(K)\) and \(|t|<\varepsilon_K\).
Verticality gives
\[
g(\Phi_j(t,y))=g(y).
\]
For any \(t\in\mathbb C\), choose \(N\) with
\(|t|/N<\varepsilon_K\) and iterate the local flow \(N\) times.
All intermediate points stay in the same compact fibre.
Uniqueness and holomorphic dependence give a jointly holomorphic
flow \(\Phi_j:\mathbb C\times Y\to Y\), with
\[
\Phi_j(t+s,y)=\Phi_j(t,\Phi_j(s,y)).
\]
This argument also applies over the singular fibres.

The commuting flows define a \(g\)-preserving holomorphic action
\[
\Psi(t,y)=\Phi_r(t_r,\cdot)\circ\cdots\circ\Phi_1(t_1,\cdot)(y).
\]
On every fibre over \(\alpha^{-1}(B^\circ)\), the orbit maps have
surjective differential. All orbits are open; connectedness of the
fibre makes the action transitive. The pulled-back section is
\(\widetilde\sigma(z)=(\sigma(\alpha(z)),z)\). Set
\[
F(z,t)=q\bigl(\Psi(t,\widetilde\sigma(z))\bigr).
\]
Then \(f\circ F=\alpha\circ\operatorname{pr}_z\), and transitivity
shows that \(F\) covers \(U\). At each \(a=(z,t)\in F^{-1}(U)\),
the vectors \(\partial_{t_j}F(a)\) form a basis of the vertical
tangent space, while
\[
df_{F(a)}(\partial_zF(a))
=d\alpha_z(\partial/\partial z)\ne0.
\]
Thus \(dF_a\) is an isomorphism.

The complement \(X\setminus U\) is proper closed analytic.
Translating \(F\) at a preimage of each point of \(U\) gives the
dominating entire maps required by Proposition~2.2 and
Theorem~1.9 of~\cite{DW}; hence \(X\) is Oka. If \(X\) is
projective, \cite[Theorem 1.1]{FL2026} gives holomorphic ellipticity.
\end{proof}

\subsection{Calabi--Yau threefolds from fibrations}\label{sec:calabi-yau-fibrations}\label{sec:gross-popescu}

In particular, a smooth connected projective Calabi--Yau threefold
with a morphism to \(\mathbb P^1\), a section, and connected abelian
surfaces as general fibres is Oka and holomorphically elliptic.

\subsubsection{Schoen fibre products}

The two elliptic fibrations in Schoen's construction \cite{Schoen}
provide two additive actions on the fibre product. Composing them along
a section gives an entire map that covers a dense Zariski-open subset
and is a local biholomorphism at every point of its source.

\begin{theorem}\label{thm:schoen-threefold}
Let \(p_i:S_i\to\mathbb P^1\), \(i=1,2\), be relatively minimal elliptic
fibrations with connected fibres and sections on smooth projective
rational surfaces. If their sets of singular values are disjoint, then
\[
X=S_1\times_{\mathbb P^1}S_2
\]
is a smooth simply connected projective threefold that is Oka and
holomorphically elliptic, with \(K_X\cong\mathcal O_X\) and
\(H^1(X,\mathcal O_X)=H^2(X,\mathcal O_X)=0\). There is a holomorphic
local biholomorphism \(F:\mathbb C^3\to X\) whose image contains a
dense Zariski-open subset.
\end{theorem}

\begin{proof}
Disjointness of the singular-value sets makes \(p_1\times p_2\)
transverse to the diagonal in \(\mathbb P^1\times\mathbb P^1\), so
\(X\) is smooth. The canonical bundle formula for a rational elliptic
surface with section gives
\(K_{S_i}\cong p_i^*\mathcal O_{\mathbb P^1}(-1)\); adjunction gives
\(K_X\cong\mathcal O_X\). The remaining Calabi--Yau and
simple-connectivity assertions are part of Schoen's construction
\cite{Schoen}; the section and canonical-bundle calculations are also
recorded in \cite[§2]{DOPW}.

Let \(f_i:X\to S_i\) be the projections, \(s_i\) sections of \(p_i\),
and \(\Omega\) a nowhere-vanishing holomorphic three-form on \(X\).
Contraction defines bundle maps
\[
V_i:f_i^*K_{S_i}\longrightarrow TX,
\qquad
\iota_{V_i(x,\eta)}\Omega_x=(df_i)_x^*\eta.
\]
By Xie--Zhao's construction \cite[Lemma 3.1 and Proposition 3.3]{XZ26},
these maps integrate to complete additive holomorphic actions
\(\mathcal A_1,\mathcal A_2\) along the fibres of the respective
projections.

Choose a biholomorphism
\(\alpha:\mathbb C\to\mathbb P^1\setminus\{b_\infty\}\) and set
\[
j(z)=\bigl(s_1(\alpha(z)),s_2(\alpha(z))\bigr).
\]
The canonical line bundle pulled back along \(f_1\circ j\) is
holomorphically trivial on \(\mathbb C\). Choose a nowhere-vanishing
section
\[
e_1\in H^0(\mathbb C,(f_1\circ j)^*K_{S_1})
\]
and define
\[
G(z,u)=\mathcal A_1(j(z),u e_1(z)).
\]
Likewise, choose a nowhere-vanishing section
\[
e_2\in H^0(\mathbb C^2,(f_2\circ G)^*K_{S_2})
\]
of the trivial pulled-back line bundle on \(\mathbb C^2\), and define
\[
F(z,u,v)=\mathcal A_2(G(z,u),v e_2(z,u)).
\]
Both maps are entire. Away from the singular values and \(b_\infty\),
the two actions are transitive on the two elliptic factors. Thus the
image of \(F\) contains the inverse image of this open subset of the
base.

We prove that \(F\) has invertible differential at every point of its
source. At \(x=(x_1,x_2)\in X\),
\[
T_xX=\{(v_1,v_2):dp_1(v_1)=dp_2(v_2)\}.
\]
Hence \(df_1\) is surjective wherever \(dp_2\) is surjective, and
\(df_2\) is surjective wherever \(dp_1\) is surjective.

Set \(H=f_2\circ G\). The identity
\(p_2(H(z,u))=\alpha(z)\) gives
\[
dp_2(\partial_zH)=d\alpha_z(\partial/\partial z)\ne0,
\qquad dp_2(\partial_uH)=0
\]
for every \((z,u)\in\mathbb C^2\). Thus \(p_2\) is a submersion at
\(H(z,u)\), and \(df_1\) has rank two at \(G(z,u)\). The action law
and contraction formula \cite[Proposition 3.3]{XZ26} give
\[
\partial_uG(z,u)=V_1(G(z,u),e_1(z))\ne0.
\]
Since \(f_1\circ G\) is independent of \(u\), this vector lies in the
kernel of \(df_1\). The differential \((df_1,df_2)\) of the inclusion
\(X\hookrightarrow S_1\times S_2\) is injective, so
\(\partial_uH=df_2(\partial_uG)\ne0\). Its image under \(dp_2\)
is zero, whereas the image of \(\partial_zH\) is nonzero. Consequently
\(dH\) is invertible on all of \(\mathbb C^2\).

The second action gives
\[
f_2(F(z,u,v))=H(z,u),
\qquad p_1(f_1(F(z,u,v)))=\alpha(z).
\]
Differentiating the second identity in \(z\) shows that \(dp_1\) is
nonzero at \(f_1(F(z,u,v))\) for every \((z,u,v)\). Thus \(df_2\)
has rank two there. The section \(e_2(z,u)\) is nonzero and independent
of \(v\), so the contraction formula gives
\[
\partial_vF(z,u,v)=V_2(F(z,u,v),e_2(z,u))\ne0.
\]
This vector spans \(\ker d(f_2)_{F(z,u,v)}\). If
\(\xi=(\xi_z,\xi_u,\xi_v)\in T_{(z,u,v)}\mathbb C^3\) satisfies
\(dF(\xi)=0\), apply \(d(f_2)\) to obtain
\[
dH_{(z,u)}(\xi_z,\xi_u)=0.
\]
Invertibility of \(dH\) gives \(\xi_z=\xi_u=0\), and
nonvanishing of \(\partial_vF\) then gives \(\xi_v=0\). Hence
\(dF\) is an isomorphism at every point of \(\mathbb C^3\).

For each point \(x\) in the dense Zariski-open subset covered by \(F\),
choose \(a\in\mathbb C^3\) with \(F(a)=x\). The map
\(w\mapsto F(a+w)\) dominates \(X\) at \(x\). Proposition~2.2 and
Theorem~1.9 of~\cite{DW} make \(X\) Oka.
It is holomorphically elliptic by \cite[Theorem 1.1]{FL2026}.
\end{proof}

\subsubsection{Gross--Popescu threefolds}

For a sufficiently general parameter \(y\in\mathbb P^2\), Gross and
Popescu \cite[§6]{GP01} construct a complete intersection
\(V_{8,y}\subset\mathbb P^7\) of four quadrics with 64 ordinary double
points and a smooth projective small resolution
\[
X=V^1_{8,y}\longrightarrow V_{8,y}.
\]
Gross--Pavanelli \cite[Theorem 1.1]{GP08} describe this resolution and
the morphism \(\pi:X\to\mathbb P^1\), whose general fibre is a
\((1,8)\)-polarized abelian surface. Their Proposition 1.3 identifies
the eight singular fibres as elliptic translation scrolls with smooth
elliptic singular loci \(E_1,\ldots,E_8\). A section gives
\[
G=X\setminus(E_1\sqcup\cdots\sqcup E_8)
\]
the structure of a smooth group scheme over \(\mathbb P^1\) with
connected fibres \cite[Lemma 1.2]{GP08}.

The fibrewise translation action of
\(G_8=(\mathbb Z/8\mathbb Z)^2\) on \(X\) is free
\cite[paragraph preceding Theorem 1.4]{GP08}. Schnell identifies the
quotient \(Y=X/G_8\) with the relative dual construction and proves
\(D^b(X)\simeq D^b(Y)\)
\cite[Theorem 4.1 and Lemmas 5.1 and 5.4]{Sch}.

\begin{theorem}\label{thm:gross-popescu}
The Gross--Popescu threefolds \(X\) and \(Y=X/G_8\) are smooth
projective Oka Calabi--Yau threefolds and are holomorphically elliptic.
\end{theorem}

\begin{proof}
The group family \(G\) is Oka by
Theorem~\ref{prop:connected-group-family}. Its complement
\[
E=E_1\sqcup\cdots\sqcup E_8
\]
is a proper closed analytic subset of \(X\), so the analytic extension
theorem \cite[Theorem 1.9]{DW} makes \(X\) Oka.

The free action makes \(X\to Y\) a finite unramified covering, and the
Oka property descends to \(Y\). Fibrewise translations preserve the
relative volume form and fix the base. Thus the holomorphic three-form
on \(X\) descends to a nowhere-vanishing holomorphic three-form on
\(Y\). Projectivity and the remaining Calabi--Yau properties follow
from the Gross--Popescu construction cited above. Both \(X\) and
\(Y\) are holomorphically elliptic by \cite[Theorem 1.1]{FL2026}.
\end{proof}

\subsection{Compactified Jacobians and relative Hilbert schemes}\label{sec:compactified-jacobians}

The nodal smoothing \(uv=t\) has a singular central fibre and a smooth
total space. In the relative compactified Jacobian, this is the
transverse model at a non-invertible sheaf, with non-invertible locus
\(u=v=0\) of codimension two. Analytic extension carries the Oka
property from the group family of degree-zero invertible sheaves to the
compactification. The relative Abel map then realizes the relative
Hilbert schemes for \(n\ge2g-1\) as projective bundles over the
compactified Jacobians.

\begin{proof}[Proof of Theorem~\ref{thm:J1}]
Write \(\sigma:B\to S\) for the section. The identity
\(df\circ d\sigma=\mathrm{id}\) places it in the relative smooth
locus. At the node of a singular fibre, smoothness of \(S\) and the
holomorphic Morse lemma give coordinates in which \(f\) has equation
\(xy=t\).

For every \(n\ge0\), the relative Hilbert scheme is projective over
\(B\), and hence projective \cite{Gro61}. It is smooth near subschemes
supported in the relative smooth locus. Near a subscheme meeting a node,
disjoint support neighbourhoods express the Hilbert problem as the
fibre product over \(B\) of a relatively smooth factor and the relative
Hilbert scheme of \(xy=t\). The latter has
smooth total space for every length by the local calculation in
\cite[Theorem 3.4]{Smith}. This proves smoothness for all \(n\); for
\(n=0\) the relative Hilbert scheme is \(B\).

Altman--Kleiman's representability and projectivity theorems
\cite[Theorems 8.1 and 8.5]{AK80} give the relative compactified Picard
space. Its components are indexed by Euler characteristic; the
degree-\(d\) component therefore has index \(d+1-g\).

Consider a one-node fibre \(C\), with normalization
\(\nu:\widetilde C\to C\) and preimages \(p,q\) of the node. We use
the local description in Smith
\cite[§3.2, Theorem 3.4 and Proposition 3.8]{Smith}. The invertible
degree-\(d\) sheaves form a \(\mathbb C^*\)-bundle over
\(\operatorname{Pic}^d(\widetilde C)\). The normalization of the
compactified Jacobian is the corresponding \(\mathbb P^1\)-bundle;
the compactified Jacobian identifies its zero and infinity sections
by translation by \(\mathcal O(p-q)\). Every
non-invertible rank-one torsion-free sheaf is uniquely of the form
\(\nu_*L\), with \(L\in\operatorname{Pic}^{d-1}(\widetilde C)\).
The degree shift follows from
\[
\chi(\nu_*L)=\chi(L),\qquad
\chi(\mathcal O_C)=\chi(\mathcal O_{\widetilde C})-1.
\]
At such a sheaf, the relative deformation has analytic model
\[
\{uv=t\}\times\mathbb C^{g-1}.
\]
Eliminating \(t\) shows that the total space is smooth. The
non-invertible locus is \(u=v=0\) and has codimension two in the total
space.

The invertible-sheaf locus in degree zero is the group family
\[
G=\operatorname{Pic}^0(S/B).
\]
Over a singular value of \(f\), the fibre of \(G\) is the generalized
Jacobian of the corresponding curve. The group family is smooth,
connected, and commutative over the entire base, with
\[
\operatorname{Lie}(G)=R^1f_*\mathcal O_S.
\]
The complement \(\overline J^0\setminus G\) is a proper closed
analytic subset of the smooth total space. If \(B\) is Oka,
Theorem~\ref{prop:connected-group-family} makes \(G\) Oka, and
analytic extension \cite[Theorem 1.9]{DW} makes \(\overline J^0\) Oka.
The section of \(f\) is a Cartier divisor on \(S\) of degree one on
every fibre. Tensoring by its \(d\)-th power identifies
\(\overline J^0\) with \(\overline J^d\). Thus
\(\overline J^d\) is Oka in all degrees when \(B\) is Oka.

For the Hilbert scheme, let \(C\) be any fibre and \(F\) a rank-one
torsion-free sheaf of degree \(n\ge2g-1\). The curve \(C\) is
integral and Gorenstein, so Serre duality gives
\[
H^1(C,F)^*=\operatorname{Hom}(F,\omega_C).
\]
A nonzero homomorphism on the right is generically injective. Its
kernel is torsion, hence zero, and its cokernel has finite length. It
would follow that \(\deg F\le\deg\omega_C=2g-2\), contrary to the
choice of \(n\). Thus
\[
H^1(C,F)=0,\qquad h^0(C,F)=n-g+1.
\]

A nonzero section \(s:\mathcal O_C\to F\) is injective by
integrality and torsion-freeness, and its cokernel \(T\) has length
\(n\). On a Gorenstein curve, torsion-free sheaves are maximal
Cohen--Macaulay, equal to their double duals, and have vanishing local
\(\operatorname{Ext}^1(-,\mathcal O_C)\). Dualizing gives
\[
0\to F^\vee\to\mathcal O_C\to
\mathcal Ext^1(T,\mathcal O_C)\to0.
\]
Local Gorenstein duality makes the last sheaf length \(n\). Hence
\(F^\vee\) is the ideal of a unique length-\(n\) subscheme.
Conversely, dualizing the ideal of a length-\(n\) subscheme gives a
rank-one torsion-free degree-\(n\) sheaf with a nonzero section. These
operations are inverse, with sections taken modulo nonzero scalars.

To obtain the relative Abel map, rigidify along the section and choose a
universal degree-\(n\) rank-one torsion-free sheaf \(\mathcal U\) on
\(S\times_B\overline J^{n}\). Write
\(q:S\times_B\overline J^{n}\to\overline J^{n}\) for the
projection. The vanishing above and cohomology and base change
\cite[Tag 07VJ]{A2026} make \(q_*\mathcal U\) a vector bundle of rank
\(n-g+1\). Dualization of these families is flat and commutes with
base change by \cite[Theorem 1.10(ii)]{AK80}, since the fibres are
Gorenstein and the sheaves are torsion-free. The relative linear-system
description of the Abel map \cite[Lemma 5.17(ii)]{AK80} identifies
\[
\operatorname{Hilb}^n(S/B)
\simeq\mathbb P_{\overline J^n(S/B)}(q_*\mathcal U),
\qquad\operatorname{rk}(q_*\mathcal U)=n-g+1,
\]
where the projective bundle parametrizes nonzero sections modulo scalars. The degree conventions are related by
\[
Z\longmapsto I_Z\longmapsto I_Z^\vee,
\qquad \deg I_Z=-n,\quad \deg I_Z^\vee=n.
\]
Tensoring the ideal by the relative dualizing line bundle relates this
map to the dualizing-sheaf Abel map \cite[equation (8.2.3)]{AK80}.
Its smoothness for \(n\ge2g-1\), including equality, follows from
\cite[Theorem 8.4(v)]{AK80}.
The relative Hilbert scheme is therefore a holomorphically locally
trivial \(\mathbb P^{n-g}\)-bundle over \(\overline J^{n}\). When
\(B\) is Oka, it is Oka by the bundle theorem \cite[Corollaries 2.51--2.52]{For}.

Conversely, the family
\(\mathcal O_S(d\sigma(B))|_{S_b}\) gives a holomorphic section of
\(\overline J^d\to B\). The Cartier thickening \(n\sigma(B)\)
is finite flat of degree \(n\) over \(B\), locally with algebra
\(\mathcal O_B[u]/(u^n)\); it gives a section of
\(\operatorname{Hilb}^n(S/B)\to B\).
Thus \(B\) is a holomorphic retract of each total space, and either
total space being Oka implies that \(B\) is Oka. For a projective
curve this is equivalent to \(g(B)\le1\)
\cite[Corollary 2.10]{For}. In this range the smooth projective total
spaces are holomorphically elliptic by \cite[Theorem 1.1]{FL2026}.
\end{proof}

\section{Deformation spaces and complements}\label{sec:geometric-constructions}

We conclude with applications to deformation spaces and analytic
complements. The first constructions exhibit an Oka open subset or
a dominating entire map explicitly. For logarithmic complements,
the required curves come from the geometry of the boundary.

\subsection{Deformation to the normal cone}
Let \(B\) be a complex manifold and \(J\subset B\) a smooth closed
submanifold. The deformation to the normal cone
\cite[Section 5.1]{Fulton98} is the space
\[
\mathcal D_J(B)=\operatorname{Bl}_{J\times\{0\}}(B\times\mathbb C)
\setminus\widetilde{B\times\{0\}}.
\]
Here the tilde denotes the strict transform.

\begin{theorem}\label{thm:N6}
Let \(B\) be a connected complex manifold and \(J\subset B\) a smooth
closed complex submanifold. Then \(\mathcal D_J(B)\) is Oka if and only
if \(B\) is Oka.
\end{theorem}

\begin{proof}
Suppose \(B\) is Oka. The complement of the central fibre is
\(B\times\mathbb C^*\), hence is
Oka.  The total space is smooth: in coordinates with
\(J=\{x_1=\cdots=x_c=0\}\), the blowup charts meeting the central fibre
have coordinates
\[
(u_1,\ldots,u_c,x_{c+1},\ldots,x_N,t),
\qquad x_i=tu_i.
\]
The central fibre is a proper closed analytic subset, so the analytic
extension theorem \cite[Theorem 1.9]{DW} applies.

Conversely, the blowdown followed by projection gives a holomorphic
map \(r:\mathcal D_J(B)\to B\), and the inclusion \(i\) of the
fibre at \(t=1\) satisfies \(r\circ i=\mathrm{id}_B\).
Composing convex approximation maps with \(i\) and then \(r\)
shows that the Oka property of \(\mathcal D_J(B)\) implies that of \(B\).
\end{proof}

For a smooth affine variety \(B=\operatorname{Spec}A\) and a smooth
closed algebraic subvariety \(J\) with ideal \(I\subset A\), the
deformation space is
\[
\mathcal D_J(B)=\operatorname{Spec}A[t,I/t],
\qquad
A[t,I/t]\subset A[t,t^{-1}].
\]
The ideal \(I\) is finitely generated, so this space is affine; it is
smooth by the local model above. If \(B\) is Oka,
Theorem~\ref{thm:N6} makes \(\mathcal D_J(B)\) Oka and hence
holomorphically elliptic \cite[Proposition 5.6.15]{F17}.
For example, the smooth affine curve
\(J=\{x^3+y^3=1\}\subset\mathbb C^2\) gives the smooth affine Oka
threefold
\[
\{(x,y,t,u)\in\mathbb C^4:tu=x^3+y^3-1\}.
\]
On \(t\ne0\), solving for \(u\) identifies this open subset with
\(\mathbb C^2\times\mathbb C^*\); the central fibre is
\(J\times\mathbb C\). Thus the equation displays the Oka open part and
the central fibre of the deformation. Being Stein and Oka, this
threefold is holomorphically elliptic \cite[Proposition 5.6.15]{F17}.

\subsection{Coordinate complements}
In a product of Oka manifolds, coordinate conditions can be avoided by
successive entire substitutions. Each substitution depends only on
the preceding coordinates, which keeps the differential triangular.

\begin{theorem}\label{thm:coordinate-arrangements}
Let \(X_1,\ldots,X_s\) be positive-dimensional connected Oka manifolds
and \(A_i\subsetneq X_i\) closed complex analytic subsets. Let
\(\mathcal H\) be a finite collection of subsets of
\(\{1,\ldots,s\}\), each containing at least two elements, and put
\[
M=\prod_{i=1}^sX_i,\qquad
W=\bigcup_{\tau\in\mathcal H}
\{x:x_i\in A_i\text{ for every }i\in\tau\}.
\]
For every closed complex analytic subset \(Z\subset M\) contained in
\(W\), the complement \(M\setminus Z\) is Oka. If \(Z\) is smooth,
the blowup \(\operatorname{Bl}_Z M\) is Oka as well.
\end{theorem}

\begin{proof}
We first prove the assertion for \(Z=W\).
Fix a point \(p=(p_i)\) of \(M\setminus W\).  Call \(j\) good when
\(p_j\notin A_j\), and reorder the factors so that the good indices precede
the others.  Every \(\tau\in\mathcal H\) contains a good index.  If
\(i=\max\tau\), choose a good
\(j_\tau\in\tau\setminus\{i\}\).  This is possible because either \(i\)
is not good or all indices in \(\tau\) are good and \(|\tau|\ge2\).

Strong dominability supplies entire maps
\(f_i:\mathbb C^{n_i}\to X_i\), where \(n_i=\dim X_i\), with
\(f_i(0)=p_i\) and \(df_i(0)\) invertible. The inverse function theorem
makes the image of \(f_i\) contain a neighbourhood of \(p_i\). Since
the proper analytic subset \(A_i\) has empty interior, choose
\(c_i\) with \(f_i(c_i)\notin A_i\). For every good \(j\),
Cartan's theorem A \cite[Section 2.6]{F17} gives an entire function
\(h_j\) that vanishes on \(f_j^{-1}(A_j)\) and satisfies \(h_j(0)=1\).
Writing \(z=(z_1,\ldots,z_s)\), form successively
\[
u_i=c_i+(z_i-c_i)
\prod_{\substack{\tau\in\mathcal H\\\max\tau=i}}
h_{j_\tau}(u_{j_\tau}),
\qquad
F(z)=\bigl(f_1(u_1),\ldots,f_s(u_s)\bigr),
\]
where an empty product is one.  Every \(u_i\) depends only on \(z_i\) and
the preceding \(u_j\), so the construction is entire.  Induction gives
\(u_i(0)=0\), and the differential of \(z\mapsto(u_i)\) at zero is block
triangular with identity diagonal blocks.  Thus \(F(0)=p\) and \(dF_0\)
is invertible.

If \(F(z)\) belonged to the forbidden locus indexed by \(\tau\), then
\(h_{j_\tau}(u_{j_\tau})=0\).  For \(i=\max\tau\), this would give
\(u_i=c_i\), contradicting \(f_i(u_i)\in A_i\).  Hence \(F\) takes
values in the complement, which is strongly dominable and therefore Oka
by \cite[Theorem 1.4]{DW}.

For \(Z\subset W\), the Oka manifold \(M\setminus W\) is a dense
Zariski-open subset of the connected manifold \(M\setminus Z\);
the arrangement \(W\) has codimension at least two.
Theorem~1.9 of~\cite{DW} makes \(M\setminus Z\) Oka.
When \(Z\) is smooth, the blowup is smooth and contains
\(M\setminus Z\) as a dense Zariski-open subset, so the same theorem
gives its Oka property.
\end{proof}

\subsection{Logarithmic complements}\label{sec:complements}
For a simple normal crossings boundary with ample components,
Chen--Zhu's theorem produces \(\mathbb A^1\)-curves from a rationally
connected deepest stratum.

\begin{theorem}\label{thm:multibranch-log-fano}
Let \(D=D_1+\cdots+D_r\), with \(r\ge1\), be a reduced simple normal
crossings divisor on a smooth connected projective manifold \(X\).
Suppose that
\begin{enumerate}[label=\textup{(\roman*)}]
\item every \(D_i\) is ample;
\item the classes \([D_1],\ldots,[D_r]\) are linearly independent in
\(N^1(X)_{\mathbb Q}\);
\item the intersection \(Y=D_1\cap\cdots\cap D_r\) is rationally
connected and has dimension at least two.
\end{enumerate}
Then \(X\setminus D\) is Oka and holomorphically elliptic.
\end{theorem}

\begin{proof}
These are the hypotheses of Chen--Zhu's theorem
\cite[Theorem 1.5]{ChenZhu}, which makes \(X\setminus D\)
\(\mathbb A^1\)-connected. It is therefore Oka by
\cite[Corollary 9.4]{DW}. The divisor \(D\) is nonzero and ample, so
its complement is affine and hence Stein. A Stein Oka manifold is
holomorphically elliptic \cite[Proposition 5.6.15]{F17}.
\end{proof}

The log Fano case follows by adjunction. Under the first two
hypotheses, suppose that \(-(K_X+D)\) is ample and
\(\dim X-r\ge2\). The simple normal crossings condition and successive
connectedness for ample divisors make \(Y\) smooth and connected.
Adjunction gives
\[
-K_Y=-(K_X+D)|_Y.
\]
Thus \(Y\) is Fano and rationally connected \cite{Campana,KMM}, so
Theorem~\ref{thm:multibranch-log-fano} applies.

For example, take \(X=\mathbb P^3\times\mathbb P^3\), a general
\(D_1\in|\mathcal O_X(1,1)|\), and a general
\(D_2\in|\mathcal O_X(1,2)|\). Bertini's theorem gives smooth
transverse divisors; their classes are independent, and
\[
-(K_X+D_1+D_2)=\mathcal O_X(2,1).
\]
Thus \(X\setminus(D_1\cup D_2)\) is Oka and holomorphically elliptic.

\begin{example}\label{ex:bidegree-complement}
Let \(d\ge1\) and let \(D\) be a general divisor of bidegree
\((1,d)\) on \(X=\mathbb P^1\times\mathbb P^2\). Then \(X\setminus D\)
is Oka and holomorphically elliptic.
\end{example}

\begin{proof}
Write
\[
D=\{([s:t],x):sF_0(x)+tF_1(x)=0\},
\]
where \(F_0,F_1\) are general homogeneous polynomials of degree \(d\).
Their zero curves meet transversely in a set \(B\) of \(d^2\) points.
The incidence equation realizes \(D\) as
\(\operatorname{Bl}_B\mathbb P^2\): away from \(B\) it is the graph
of \(x\mapsto[F_1(x):-F_0(x)]\), and near a base point its equation
is \(su+tv=0\) in local coordinates \(u=F_0,v=F_1\).
Thus \(D\) is a smooth rational surface. Its line bundle
\(\mathcal O_X(1,d)\) is ample, so
Theorem~\ref{thm:multibranch-log-fano} applies. Here
\[
-(K_X+D)=\mathcal O_X(1,3-d),
\]
which is not ample for \(d\ge3\).

Projection onto \(\mathbb P^2\) also describes the complement as an
affine-line torsor over \(\mathbb P^2\setminus B\). On
\(\{F_0\ne0\}\), put \(h=F_1/F_0\) and
\(z=t/(s+ht)\); on \(\{F_1\ne0\}\), put
\(w=s/(h^{-1}s+t)\). The transition is
\[
w=h-h^2z.
\]
These two trivializations give a holomorphic fibre bundle with fibre
\(\mathbb C\). The base \(\mathbb P^2\setminus B\) is Oka: each
standard affine chart deletes only finitely many points, so the
codimension-two complement theorem \cite[Section 6.4]{F17} and
Zariski localization \cite[Theorem 1.4]{Kus21} apply. The Oka
fibre-bundle theorem \cite[Corollaries 2.51--2.52]{For} gives a second
proof of the Oka conclusion.
\end{proof}

For surfaces, Maeda's classification reduces the complement to a
homogeneous space or the total space of a line bundle on
\(\mathbb P^1\), giving holomorphic ellipticity.

\begin{theorem}\label{thm:log-fano-surfaces}
Let \(S\) be a smooth projective complex surface and \(D\ne0\) a
reduced simple normal crossings divisor on \(S\). If \(-(K_S+D)\) is ample, then
\(S\setminus D\) is holomorphically elliptic and Oka.
\end{theorem}

\begin{proof}
Write \(\mathbb F_m\) for the Hirzebruch surface, \(C_m\) for a section
with \(C_m^2=-m\), and \(F\) for a ruling fibre. Maeda's classification
\cite[Section 3.4, pp. 89--90]{Maeda1986} gives the following pairs and
complement models:
\begin{center}
\begin{tabular}{@{}ll@{}}
\toprule
\((S,D)\) & \(S\setminus D\)\\
\midrule
\(\mathbb P^2,\ \text{a line}\) & \(\mathbb C^2\)\\
\(\mathbb P^2,\ \text{two distinct lines}\) & \(\mathbb C\times\mathbb C^*\)\\
\(\mathbb P^2,\ \text{a smooth conic}\) &
\(\big((\mathbb P^1)^2\setminus\Delta\big)/\mathfrak S_2\)\\
\(\mathbb F_m,\ C_m\), \(m\ge0\) & \(\operatorname{Tot}\mathcal O_{\mathbb P^1}(m)\)\\
\(\mathbb F_m,\ C_m+F\), \(m\ge0\) & \(\mathbb C^2\)\\
\(\mathbb F_1,\ D^2=1\ \text{a section}\) &
\(\operatorname{Bl}_0\mathbb C^2=\operatorname{Tot}\mathcal O_{\mathbb P^1}(-1)\)\\
\(\mathbb F_0,\ D^2=2\ \text{a section}\) &
\((\mathbb P^1\times\mathbb P^1)\setminus\Delta\)\\
\bottomrule
\end{tabular}
\end{center}
The first two models are complex Lie groups. Binary quadratic forms
identify the conic complement with unordered pairs of distinct points
of \(\mathbb P^1\); the diagonal \(\operatorname{PGL}_2\)-action is
transitive. These homogeneous manifolds are holomorphically elliptic
\cite[§5.6]{F17}.

Deleting \(C_m\) gives the line bundle \(\mathcal O(m)\); deleting
one ruling fibre as well trivializes it over \(\mathbb C\).
A self-intersection-one section on \(\mathbb F_1\) is the strict
transform of a line avoiding the blowup point, giving the sixth model.
A self-intersection-two section on \(\mathbb F_0\) is the graph of
an automorphism of \(\mathbb P^1\); its complement is isomorphic to
the last model, which is homogeneous under the diagonal
\(\operatorname{PGL}_2\)-action.

Every \(X=\operatorname{Tot}\mathcal O_{\mathbb P^1}(k)\),
\(k\in\mathbb Z\), admits an explicit dominating algebraic spray.
Write \(\pi:X\to\mathbb P^1\) for the projection and give
\(\mathcal O(k)\) its standard algebraic
\(\operatorname{SL}_2\)-linearization. With \(E_{12},E_{21}\) the
matrix units, put
\[
\begin{aligned}
s:\bigl(X\times\mathbb C^2\bigr)\oplus\pi^*\mathcal O(k)&\longrightarrow X,\\
s(x;a,b,v)&=(I+aE_{12})(I+bE_{21})\cdot(x+v).
\end{aligned}
\]
This is algebraic and equals \(x\) at \(a=b=v=0\). The
\(v\)-differential supplies the vertical tangent line. The
\(a\)- and \(b\)-differentials project to the two root vector fields
on \(\mathbb P^1\), whose unique zeros are respectively
\(\infty\) and \(0\). They have no common zero, so the spray is
dominating. These line-bundle models are therefore algebraically
elliptic.
\end{proof}

A quadric and a transverse hyperplane
give a higher-dimensional example with an explicit affine model.

\begin{example}\label{prop:quadric-hyperplane}
Let \(Q\subset\mathbb P^n\), \(n\ge2\), be a smooth quadric and
\(H\subset\mathbb P^n\) a hyperplane transverse to \(Q\). Then
\(\mathbb P^n\setminus(Q\cup H)\) is holomorphically elliptic and Oka.
For \(n\ge3\), the reduced simple normal crossings pair
\((\mathbb P^n,Q+H)\) is log Fano, although its boundary classes are
numerically dependent.
\end{example}

\begin{proof}
In the affine chart \(\mathbb P^n\setminus H\), transversality
and completion of the square identify the complement with
\[
\{z\in\mathbb C^n:z_1^2+\cdots+z_n^2\ne1\}.
\]
For \(n=2\), this is linearly equivalent to
\(\mathbb C^2\setminus\{xy=1\}\), whose Oka property follows from
Hanysz's graph-complement theorem \cite[arXiv v2, Theorem 4.6]{Hanysz2014},
with \(m(x)=1/x\). Xie's theorem gives holomorphic ellipticity for
all \(n\ge2\) \cite[Theorem 6.1 and Corollary 6.3]{XieCubic}. Moreover,
\[
\mathcal O_{\mathbb P^n}(-(K_{\mathbb P^n}+Q+H))
\cong\mathcal O_{\mathbb P^n}(n-2),
\]
which is ample exactly when \(n\ge3\), while \([Q]=2[H]\) in
\(N^1(\mathbb P^n)_{\mathbb Q}\).
\end{proof}

\subsection{Toric boundaries and McLean's fourfold}
L\'arusson's theorem on smooth toric varieties also applies to
torus-invariant boundary complements.

\begin{example}\label{prop:toric-boundary}
Let \(X\) be a smooth complex toric variety and \(D\) a reduced
union of torus-invariant prime divisors. Then \(X\setminus D\) is Oka. In
particular, every reduced torus-invariant log Fano pair on a smooth
projective toric variety has Oka complement.
\end{example}

\begin{proof}
The complement is a smooth torus-invariant open subvariety containing the
dense torus, hence is itself a smooth toric variety. It is Oka by
L\'arusson's theorem \cite[Theorem 3]{LarussonToric}.
\end{proof}

McLean's modification of \(\mathbb C^4\) has an Oka open
subset that is particularly simple in affine coordinates.

\begin{proposition}\label{prop:mclean-fourfold}
For
\[
P=x^7+y^2+z^2+w^2,\qquad p=(0,0,1,i),
\]
the manifold
\(X=\operatorname{Bl}_p(\mathbb C^4)\setminus\widetilde{\{P=0\}}\),
where the tilde denotes strict transform, is Stein, Oka, and
holomorphically elliptic. It is diffeomorphic to \(\mathbb R^8\) and
not biholomorphic to \(\mathbb C^4\).
\end{proposition}

\begin{proof}
Write \(u=y+iz\) and \(v=y-iz\), so \(P=x^7+uv+w^2\). Let
\(\pi:X\to\mathbb C^4\) be the restriction of the blowdown. It identifies
\[
U=\{q\in X:(uP)(\pi(q))\ne0\}
\]
with \(\{uP\ne0\}\subset\mathbb C^4\). The coordinates
\[
(x,u,v,w)\longmapsto(x,w,u,t=P),
\qquad
v=\frac{t-x^7-w^2}{u},
\]
give \(U\cong\mathbb C^2\times(\mathbb C^*)^2\), so \(U\)
is Oka. Its complement is the zero set of \((uP)\circ\pi\), a
proper closed analytic subset of the connected manifold \(X\).
Analytic extension \cite[Theorem 1.9]{DW} makes \(X\) Oka.

McLean proves that this affine fourfold is Stein, contractible,
and diffeomorphic to \(\mathbb R^8\), and that its finite-type Stein
structure has nonzero symplectic homology
\cite[Sections 3.1--3.2]{McLean09}.
Since \(X\) is Stein and Oka, it is holomorphically elliptic
\cite[Proposition 5.6.15]{F17}.
If \(X\) were biholomorphic to \(\mathbb C^4\), pulling back the
standard quadratic exhaustion would give a complete finite-type
strictly plurisubharmonic exhaustion on \(X\).
Seidel--Smith's embedding lemma \cite[Lemma 7]{SeidelSmith05} would
then give an exact symplectic embedding of McLean's Stein structure
into standard \(\mathbb C^4\). The image of a sufficiently large
Stein sublevel would lie in a subcritical ball, contradicting
\cite[Corollary 10.4]{McLean09}.
\end{proof}

\section*{Grant support}
S.-Y. Xie acknowledges partial support from the National Key R\&D
Program of China under Grants No.~2023YFA1010500 and
No.~2021YFA1003100, and from the National Natural Science Foundation of
China under Grants No.~12288201 and No.~12471081, as well as support
from the Xiaomi Young Talents Program.

\section*{Statement on the use of artificial intelligence}
Generative artificial intelligence assisted in the preparation and revision of this manuscript. The authors have reviewed the text and take full responsibility for its mathematical content and final form.


\begin{thebibliography}{999}

\small

\bibitem{AF2025}
Antonio Alarc\'on and Franc Forstneri\v{c}. \emph{Oka-1 manifolds}.
Mathematische Zeitschrift 311 (2025), no. 2, article 33, 34 pp.
\href{https://doi.org/10.1007/s00209-025-03833-4}{doi:10.1007/s00209-025-03833-4}.

\bibitem{AK80}
Allen B. Altman and Steven L. Kleiman. \emph{Compactifying the Picard scheme}. Advances in Mathematics 35 (1980), no. 1, 50--112.
\href{https://doi.org/10.1016/0001-8708(80)90043-2}{doi:10.1016/0001-8708(80)90043-2}.

\bibitem{DW}
Sicheng An, Yun-Heng Du, Bin Guo, Peng-Chao Wang, and Song-Yan Xie. \emph{A Local Approximation Characterization of Oka Manifolds}. Revised manuscript, 14 October 2026.
\href{https://xiesongyan.github.io/papers/an_du_guo_wang_xie_local_approximation_oka_manifolds_2026.pdf}{Author's revised version}.
\href{https://arxiv.org/abs/2610.10475}{arXiv:2610.10475}.

\bibitem{Andrist2025}
Rafael B. Andrist. \emph{Holomorphic automorphisms of Markov-type surfaces}. Preprint, 2025.
\href{https://arxiv.org/abs/2502.15101}{arXiv:2502.15101}.

\bibitem{A2024}
Alexis Aumonier. \emph{The topology of spaces of holomorphic maps to projective space}. Advances in Mathematics 483 (2025), Article 110686.
\href{https://doi.org/10.1016/j.aim.2025.110686}{doi:10.1016/j.aim.2025.110686}.

\bibitem{BS2023}
Gwyn Bellamy and Travis Schedler. \emph{Symplectic resolutions of character varieties}. Geometry \& Topology 27 (2023), no. 1, 51--86.
\href{https://doi.org/10.2140/gt.2023.27.51}{doi:10.2140/gt.2023.27.51}.

\bibitem{B1992}
Aaron Bertram. \emph{Moduli of rank-2 vector bundles, theta divisors, and the geometry of curves in projective space}. Journal of Differential Geometry 35 (1992), no. 2, 429--469.
\href{https://doi.org/10.4310/jdg/1214448083}{doi:10.4310/jdg/1214448083}.

\bibitem{BB2012}
Usha N. Bhosle and Indranil Biswas. \emph{Brauer group and birational type of moduli spaces of torsionfree sheaves on a nodal curve}. Communications in Algebra 42 (2014), no. 4, 1769--1784.
\href{https://doi.org/10.1080/00927872.2012.749263}{doi:10.1080/00927872.2012.749263}.

\bibitem{BodenYokogawa1999}
Hans U. Boden and K\^oji Yokogawa. \emph{Rationality of moduli spaces of parabolic bundles}. Journal of the London Mathematical Society (2) 59 (1999), no. 2, 461--478.
\href{https://doi.org/10.1112/S0024610799007061}{doi:10.1112/S0024610799007061}.

\bibitem{Bowman2016}
David Bowman. \emph{Holomorphic Flexibility Properties of Spaces of Elliptic Functions}. Ph.D. thesis, University of Adelaide, 2016.
\href{https://arxiv.org/abs/1609.07184}{arXiv:1609.07184}.

\bibitem{Campana}
Fr\'ed\'eric Campana. \emph{Connexit\'e rationnelle des vari\'et\'es de Fano}. Annales scientifiques de l'\'Ecole Normale Sup\'erieure (4) 25 (1992), 539--545.
\href{https://doi.org/10.24033/asens.1658}{doi:10.24033/asens.1658}.

\bibitem{CantatLoray2009}
Serge Cantat and Frank Loray. \emph{Dynamics on character varieties and Malgrange irreducibility of Painlev\'e VI equation}. Annales de l'Institut Fourier 59 (2009), no. 7, 2927--2978.
\href{https://doi.org/10.5802/aif.2512}{doi:10.5802/aif.2512}.

\bibitem{ChenZhu}
Qile Chen and Yi Zhu. \emph{\(\mathbb A^1\)-curves on log smooth varieties}. Journal f\"ur die reine und angewandte Mathematik 756 (2019), 1--35.
\href{https://doi.org/10.1515/crelle-2017-0028}{doi:10.1515/crelle-2017-0028}.

\bibitem{deCataldoMigliorini2000}
Mark Andrea A. de Cataldo and Luca Migliorini. \emph{The Douady space of a
complex surface}. Advances in Mathematics 151 (2000), no. 2, 283--312.
\href{https://doi.org/10.1006/aima.1999.1896}{doi:10.1006/aima.1999.1896}.

\bibitem{Del70}
Pierre Deligne. \emph{\'Equations diff\'erentielles \`a points singuliers r\'eguliers}. Lecture Notes in Mathematics, vol. 163. Springer, 1970.
\href{https://doi.org/10.1007/BFb0061194}{doi:10.1007/BFb0061194}.

\bibitem{DN1989}
Jean-Marc Dr\'ezet and M. S. Narasimhan. \emph{Groupe de Picard des
vari\'et\'es de modules de fibr\'es semi-stables sur les courbes
alg\'ebriques}. Inventiones Mathematicae 97 (1989), 53--94.
\href{https://doi.org/10.1007/BF01850655}{doi:10.1007/BF01850655}.

\bibitem{DOPW}
Ron Donagi, Burt A. Ovrut, Tony Pantev, and Daniel Waldram. \emph{Spectral involutions on rational elliptic surfaces}. Advances in Theoretical and Mathematical Physics 5 (2001), no. 3, 499--561.
\href{https://doi.org/10.4310/ATMP.2001.v5.n3.a4}{doi:10.4310/ATMP.2001.v5.n3.a4}.

\bibitem{EFK2022}
Pavel Etingof, Edward Frenkel, and David Kazhdan. \emph{Analytic Langlands correspondence for \(\mathrm{PGL}_2\) on \(\mathbb P^1\) with parabolic structures over local fields}. Geometric and Functional Analysis 32 (2022), 725--831.
\href{https://doi.org/10.1007/s00039-022-00603-w}{doi:10.1007/s00039-022-00603-w}.

\bibitem{FarbWolfson2017}
Benson Farb and Jesse Wolfson. \emph{Topology and arithmetic of resultants, II: the resultant $=1$ hypersurface}. With an appendix by Christophe Cazanave. Algebraic Geometry 4 (2017), no. 3, 337--352.
\href{https://doi.org/10.14231/AG-2017-019}{doi:10.14231/AG-2017-019}.

\bibitem{F1968}
John Fogarty. \emph{Algebraic families on an algebraic surface}. American Journal of Mathematics 90 (1968), no. 2, 511--521.
\href{https://doi.org/10.2307/2373541}{doi:10.2307/2373541}.

\bibitem{F06}
Franc Forstneri\v{c}. \emph{Runge approximation on convex sets implies the Oka property}. Annals of Mathematics (2) 163 (2006), no. 2, 689--707.
\href{https://doi.org/10.4007/annals.2006.163.689}{doi:10.4007/annals.2006.163.689}.

\bibitem{For}
Franc Forstneri\v{c}. \emph{Oka manifolds: From Oka to Stein and back}. Annales de la Facult\'e des Sciences de Toulouse 22 (2013), no. 4, 747--809.
\href{https://doi.org/10.5802/afst.1388}{doi:10.5802/afst.1388}.

\bibitem{F17}
Franc Forstneri\v{c}. \emph{Stein Manifolds and Holomorphic Mappings: The Homotopy Principle in Complex Analysis}. Second edition. Ergebnisse der Mathematik und ihrer Grenzgebiete, 3. Folge, vol. 56. Springer, 2017.
\href{https://doi.org/10.1007/978-3-319-61058-0}{doi:10.1007/978-3-319-61058-0}.

\bibitem{FL2014}
Franc Forstneri\v{c} and Finnur L\'arusson. \emph{Holomorphic flexibility
properties of compact complex surfaces}. International Mathematics Research
Notices 2014, no. 13, 3714--3734.
\href{https://doi.org/10.1093/imrn/rnt044}{doi:10.1093/imrn/rnt044}.

\bibitem{FL2026}
Franc Forstneri\v{c} and Finnur L\'arusson. \emph{Every projective Oka
manifold is elliptic}. Mathematical Research Letters 33 (2026), no. 1,
297--312.
\href{https://doi.org/10.4310/MRL.260503003231}{doi:10.4310/MRL.260503003231}.
\href{https://arxiv.org/abs/2502.20028v6}{arXiv:2502.20028v6}.

\bibitem{Fulton98}
William Fulton. \emph{Intersection Theory}. Second edition. Ergebnisse der Mathematik und ihrer Grenzgebiete, 3. Folge, vol. 2. Springer, 1998.
\href{https://doi.org/10.1007/978-1-4612-1700-8}{doi:10.1007/978-1-4612-1700-8}.

\bibitem{Gold84}
William M. Goldman. \emph{The symplectic nature of fundamental groups of surfaces}. Advances in Mathematics 54 (1984), no. 2, 200--225.
\href{https://doi.org/10.1016/0001-8708(84)90040-9}{doi:10.1016/0001-8708(84)90040-9}.

\bibitem{Gold86}
William M. Goldman. \emph{Invariant functions on Lie groups and Hamiltonian flows of surface group representations}. Inventiones Mathematicae 85 (1986), no. 2, 263--302.
\href{https://doi.org/10.1007/BF01389091}{doi:10.1007/BF01389091}.

\bibitem{G2004}
William M. Goldman. \emph{The complex-symplectic geometry of $\mathrm{SL}_2(\mathbb C)$-characters over surfaces}. In \emph{Algebraic Groups and Arithmetic}, Tata Institute of Fundamental Research, Mumbai, 2004, 375--407.
\href{https://arxiv.org/abs/math/0304307}{arXiv:math/0304307}.

\bibitem{GoldmanNeumann2005}
William M. Goldman and Walter D. Neumann. \emph{Homological action of the modular group on some cubic moduli spaces}. Mathematical Research Letters 12 (2005), no. 4, 575--591.
\href{https://arxiv.org/abs/math/0402039v1}{arXiv:math/0402039v1}.

\bibitem{GX2011}
William M. Goldman and Eugene Z. Xia. \emph{Ergodicity of mapping class group actions on $\mathrm{SU}(2)$-character varieties}. In \emph{Geometry, Rigidity, and Group Actions}, University of Chicago Press, 2011, 591--608.
\href{https://arxiv.org/abs/0901.1402}{arXiv:0901.1402}.

\bibitem{Gottsche94}
Lothar G\"ottsche. \emph{Hilbert Schemes of Zero-Dimensional Subschemes of Smooth Varieties}. Lecture Notes in Mathematics, vol. 1572. Springer, Berlin, 1994.
\href{https://doi.org/10.1007/BFb0073491}{doi:10.1007/BFb0073491}.

\bibitem{Grauert58}
Hans Grauert. \emph{Analytische Faserungen \"uber holomorph-vollst\"andigen R\"aumen}. Mathematische Annalen 135 (1958), 263--273.
\href{https://doi.org/10.1007/BF01351803}{doi:10.1007/BF01351803}.

\bibitem{Grauert60}
Hans Grauert. \emph{Ein Theorem der analytischen Garbentheorie und die
Modulr\"aume komplexer Strukturen}. Publications Math\'ematiques de
l'IH\'ES 5 (1960), 5--64.
\href{https://doi.org/10.1007/BF02684746}{doi:10.1007/BF02684746}.
Correction: ibid. 16 (1963), 35--36.
\href{https://doi.org/10.1007/BF02684291}{doi:10.1007/BF02684291}.

\bibitem{Gromov89}
Mikhail Gromov. \emph{Oka's principle for holomorphic sections of elliptic bundles}. Journal of the American Mathematical Society 2 (1989), no. 4, 851--897.
\href{https://doi.org/10.1090/S0894-0347-1989-1001851-9}{doi:10.1090/S0894-0347-1989-1001851-9}.

\bibitem{GP08}
Mark Gross and Simone Pavanelli. \emph{A Calabi--Yau threefold with Brauer group $(\mathbb Z/8\mathbb Z)^2$}. Proceedings of the American Mathematical Society 136 (2008), no. 1, 1--9.
\href{https://doi.org/10.1090/S0002-9939-07-08840-5}{doi:10.1090/S0002-9939-07-08840-5}.

\bibitem{GP01}
Mark Gross and Sorin Popescu. \emph{Calabi--Yau threefolds and moduli of abelian surfaces I}. Compositio Mathematica 127 (2001), no. 2, 169--228.
\href{https://doi.org/10.1023/A:1012076503121}{doi:10.1023/A:1012076503121}.

\bibitem{Gro61}
Alexander Grothendieck. \emph{Techniques de construction et th\'eor\`emes d'existence en g\'eom\'etrie alg\'ebrique IV: les sch\'emas de Hilbert}. S\'eminaire Bourbaki, 13e ann\'ee, expos\'e 221, in S\'eminaire Bourbaki, vol. 6 (1961), 249--276.
\href{https://numdam.org/item/SB_1960-1961__6__249_0/}{numdam.org/item/SB\_1960-1961\_\_6\_\_249\_0}.

\bibitem{GHJW1997}
K. Guruprasad, J. Huebschmann, L. Jeffrey, and A. Weinstein. \emph{Group systems, groupoids, and moduli spaces of parabolic bundles}. Duke Mathematical Journal 89 (1997), no. 2, 377--412.
\href{https://arxiv.org/abs/dg-ga/9510006v1}{arXiv:dg-ga/9510006v1}.

\bibitem{Hanysz2014}
Alexander Hanysz. \emph{Oka properties of some hypersurface complements}.
Proceedings of the American Mathematical Society 142 (2014), no. 2,
483--496.
\href{https://doi.org/10.1090/S0002-9939-2013-11754-5}{doi:10.1090/S0002-9939-2013-11754-5}.
\href{https://arxiv.org/abs/1111.6655v2}{arXiv:1111.6655v2}.

\bibitem{Hanysz2015}
Alexander Hanysz. \emph{Holomorphic flexibility properties of the space of cubic rational maps}. Journal of Geometric Analysis 25 (2015), 1620--1649.
\href{https://doi.org/10.1007/s12220-014-9487-0}{doi:10.1007/s12220-014-9487-0}.

\bibitem{H2009}
Jochen Heinloth. \emph{Lectures on the moduli stack of vector bundles on a curve}. In \emph{Affine Flag Manifolds and Principal Bundles}, Trends in Mathematics, Birkh\"auser, 2010, 123--152.
\href{https://doi.org/10.1007/978-3-0346-0288-4_4}{doi:10.1007/978-3-0346-0288-4\_4}.

\bibitem{I2012}
Michi-aki Inaba. \emph{Moduli of parabolic connections on curves and the Riemann--Hilbert correspondence}. Journal of Algebraic Geometry 22 (2013), 407--480.
\href{https://doi.org/10.1090/S1056-3911-2013-00621-9}{doi:10.1090/S1056-3911-2013-00621-9}.

\bibitem{IIS2006}
Michi-aki Inaba, Katsunori Iwasaki, and Masa-Hiko Saito. \emph{Moduli of stable parabolic connections, Riemann--Hilbert correspondence and geometry of Painlev\'e equation of type VI, Part I}. Publications of the Research Institute for Mathematical Sciences 42 (2006), no. 4, 987--1089.
\href{https://doi.org/10.2977/prims/1166642194}{doi:10.2977/prims/1166642194}.

\bibitem{KMM}
J\'anos Koll\'ar, Yoichi Miyaoka, and Shigefumi Mori. \emph{Rational connectedness and boundedness of Fano manifolds}. Journal of Differential Geometry 36 (1992), 765--779.
\href{https://doi.org/10.4310/jdg/1214453188}{doi:10.4310/jdg/1214453188}.

\bibitem{Kus21}
Yuta Kusakabe. \emph{Elliptic characterization and localization of Oka manifolds}. Indiana University Mathematics Journal 70 (2021), no. 3, 1039--1054.
\href{https://doi.org/10.1512/iumj.2021.70.8454}{doi:10.1512/iumj.2021.70.8454}.

\bibitem{LarussonToric}
Finnur L\'arusson. \emph{Smooth toric varieties are Oka}. arXiv:1107.3604, 2013.
\href{https://arxiv.org/abs/1107.3604}{arXiv:1107.3604}.

\bibitem{Maeda1986}
Hironobu Maeda. \emph{Classification of logarithmic Fano threefolds}. Compositio Mathematica 57 (1986), 81--125.
\href{https://www.numdam.org/item/CM_1986__57_1_81_0/}{numdam.org/item/CM\_1986\_\_57\_1\_81\_0}.

\bibitem{McLean09}
Mark McLean. \emph{Lefschetz fibrations and symplectic homology}. Geometry \& Topology 13 (2009), no. 4, 1877--1944.
\href{https://doi.org/10.2140/gt.2009.13.1877}{doi:10.2140/gt.2009.13.1877}.

\bibitem{MS1980}
V. B. Mehta and C. S. Seshadri. \emph{Moduli of vector bundles on curves with parabolic structures}. Mathematische Annalen 248 (1980), 205--239.
\href{https://doi.org/10.1007/BF01420526}{doi:10.1007/BF01420526}.

\bibitem{MonavariRicolfi2022}
Sergej Monavari and Andrea T. Ricolfi. \emph{On the motive of the nested
Quot scheme of points on a curve}. Journal of Algebra 610 (2022), 99--118.
\href{https://doi.org/10.1016/j.jalgebra.2022.07.011}{doi:10.1016/j.jalgebra.2022.07.011}.
\href{https://arxiv.org/abs/2106.11817v2}{arXiv:2106.11817v2}.

\bibitem{NWY2002}
Junjiro Noguchi, J\"org Winkelmann, and Katsutoshi Yamanoi.
\emph{The second main theorem for holomorphic curves into semi-Abelian
varieties}. Acta Mathematica 188 (2002), no. 1, 129--161.
\href{https://doi.org/10.1007/BF02392797}{doi:10.1007/BF02392797}.

\bibitem{Oka39}
Kiyoshi Oka. \emph{Sur les fonctions analytiques de plusieurs variables. III---Deuxi\`eme probl\`eme de Cousin}. Journal of Science of the Hiroshima University, Series A 9 (1939), 7--19.
\href{https://doi.org/10.32917/HMJ/1558490525}{doi:10.32917/HMJ/1558490525}.

\bibitem{PS1998}
J\'ozef H. Przytycki and Adam S. Sikora. \emph{On skein algebras and $\mathrm{SL}_2(\mathbb C)$-character varieties}. Topology 39 (2000), no. 1, 115--148.
\href{https://doi.org/10.1016/S0040-9383(98)00062-7}{doi:10.1016/S0040-9383(98)00062-7}.

\bibitem{Sch}
Christian Schnell. \emph{The fundamental group is not a derived invariant}. In Yujiro Kawamata, editor, \emph{Derived Categories in Algebraic Geometry: Tokyo 2011}, EMS Series of Congress Reports, European Mathematical Society, 2013, 279--285.
\href{https://doi.org/10.4171/115-1/10}{doi:10.4171/115-1/10}.

\bibitem{Schoen}
Chad Schoen. \emph{On fiber products of rational elliptic surfaces with section}. Mathematische Zeitschrift 197 (1988), no. 2, 177--199.
\href{https://doi.org/10.1007/BF01215188}{doi:10.1007/BF01215188}.

\bibitem{SeidelSmith05}
Paul Seidel and Ivan Smith. \emph{The symplectic topology of Ramanujam's surface}. Commentarii Mathematici Helvetici 80 (2005), no. 4, 859--881.
\href{https://doi.org/10.4171/CMH/37}{doi:10.4171/CMH/37}.

\bibitem{S2012}
Adam S. Sikora. \emph{Character varieties}. Transactions of the American Mathematical Society 364 (2012), no. 10, 5173--5208.
\href{https://doi.org/10.1090/S0002-9947-2012-05448-1}{doi:10.1090/S0002-9947-2012-05448-1}.

\bibitem{Simpson94}
Carlos T. Simpson. \emph{Moduli of representations of the fundamental group of a smooth projective variety II}. Publications Math\'ematiques de l'IH\'ES 80 (1994), 5--79.
\href{https://doi.org/10.1007/BF02698895}{doi:10.1007/BF02698895}.

\bibitem{Smith}
Ivan Smith. \emph{Serre--Taubes duality for pseudoholomorphic curves}. Topology 42 (2003), no. 5, 931--979.
\href{https://doi.org/10.1016/S0040-9383(02)00025-3}{doi:10.1016/S0040-9383(02)00025-3}.

\bibitem{A2026}
The Stacks Project Authors. \emph{The Stacks Project}. 2026. Accessed 19 October 2026.
\href{https://stacks.math.columbia.edu/}{stacks.math.columbia.edu}.

\bibitem{S1999}
Xiaotao Sun. \emph{Degeneration of moduli spaces and generalized theta functions}. Journal of Algebraic Geometry 9 (2000), no. 3, 459--527.
\href{https://arxiv.org/abs/math/9907188}{arXiv:math/9907188}.

\bibitem{XieCubic}
Song-Yan Xie. \emph{Elliptic complements of cubic hypersurfaces}. arXiv:2608.02413, 2026.
\href{https://arxiv.org/abs/2608.02413}{arXiv:2608.02413}.

\bibitem{XZ26}
Song-Yan Xie and Shengyuan Zhao. \emph{Periods of Oka K3 surfaces are dense}. Preprint, 2026.
\href{https://arxiv.org/abs/2608.24392}{arXiv:2608.24392}.

\end{thebibliography}
\end{document}